\documentclass[12pt,a4paper]{amsart}
\usepackage{amsfonts,color}
\usepackage{amsthm}
\usepackage{amsmath}
\usepackage{amscd}
\usepackage[latin2]{inputenc}
\usepackage{t1enc}
\usepackage[mathscr]{eucal}
\usepackage{indentfirst}
\usepackage{graphicx}
\usepackage{graphics}
\usepackage{pict2e}
\usepackage{epic}
\usepackage{url}
\usepackage{epstopdf}
\usepackage{comment}
\usepackage{todonotes}
\usepackage{amssymb}
\usepackage[pdftex, bookmarks=true]{hyperref}

\usepackage{enumerate}
\usepackage{mathtools}

\tikzset{
    root/.style={red, rectangle, fill, inner sep=2.2pt}
}

\numberwithin{equation}{section}
\usepackage[margin=1in]{geometry}

\makeatletter
\newcommand{\addresseshere}{%
  \enddoc@text\let\enddoc@text\relax
}
\makeatother

\theoremstyle{plain}
\newtheorem{Thm}{Theorem}[section]
\newtheorem{Lemma}[Thm]{Lemma}

\newtheorem{St}[Thm]{Statement}

\newtheorem{Prop}[Thm]{Proposition}

\theoremstyle{definition}
\newtheorem{Def}[Thm]{Definition}
\newtheorem{Conj}[Thm]{Conjecture}
\newtheorem{Rem}[Thm]{Remark}
\newtheorem{?}[Thm]{Problem}

\allowdisplaybreaks

\DeclareMathOperator{\dist}{dist}
\DeclareMathOperator{\adj}{adj}

\newcommand{\x}{\underline{x}}
\newcommand{\y}{\underline{y}}
\newcommand{\z}{\underline{z}}

\renewcommand{\v}{\underline{v}}
\newcommand{\m}{\mathrm{main}}
\newcommand{\e}{\mathrm{err}}

\def\xt{\widetilde{\underline{x}}}

\def\Ht{\widetilde{H}}
\def\Bt{\widetilde{B}}
\def\St{\widetilde{S}}
\def\Tt{\widetilde{T}}
\def\Qt{\widetilde{Q}}

\def\Uc{\mathcal{U}}
\def\Pc{\mathcal{P}}
\def\Jh{\hat{J}}
\def\Sh{\hat{S}}
\def\Th{\hat{T}}

\def\al{\alpha}
\def\be{\beta}
\def\ga{\gamma}
\def\gat{\widetilde{\gamma}}
\def\Ga{\Gamma}

\def\eps{\varepsilon}
\def\ka{\kappa}
\def\la{\lambda}
\def\lat{\widetilde{\lambda}}

\def\Vact{V_{\mathrm{act}}}
\def\bestar{\beta_{\star}}
\def\betr{\beta_{\mathrm{tr}}}

\def\sm{\setminus}

\def \Esix{
\begin{tikzpicture}[baseline=-0.6ex]
\draw (0,0) -- (0.3,0) -- (0.6,0) -- (0.9,0) -- (1.2,0);
\draw (0.6,0) -- (0.6,0.3);
\filldraw (0,0) circle (1.1pt);
\filldraw (0.3,0) circle (1.1pt);
\filldraw (0.6,0) circle (1.1pt);
\filldraw (0.9,0) circle (1.1pt);
\filldraw (1.2,0) circle (1.1pt);
\filldraw (0.6,0.3) circle (1.1pt);
\end{tikzpicture}
}

\def \Ehatsix{
\begin{tikzpicture}[baseline=-0.6ex]
\draw (0,0) -- (0.3,0) -- (0.6,0) -- (0.9,0) -- (1.2,0);
\draw (0.6,0) -- (0.6,0.3) -- (0.6,0.6);
\filldraw (0,0) circle (1.1pt);
\filldraw (0.3,0) circle (1.1pt);
\filldraw (0.6,0) circle (1.1pt);
\filldraw (0.9,0) circle (1.1pt);
\filldraw (1.2,0) circle (1.1pt);
\filldraw (0.6,0.3) circle (1.1pt);
\filldraw (0.6,0.6) circle (1.1pt);
\end{tikzpicture}
}

\def \Eseven{
\begin{tikzpicture}[baseline=-0.6ex]
\draw (0,0) -- (0.3,0) -- (0.6,0) -- (0.9,0) -- (1.2,0) -- (1.5,0);
\draw (0.9,0) -- (0.9,0.3);
\filldraw (0,0) circle (1.1pt);
\filldraw (0.3,0) circle (1.1pt);
\filldraw (0.6,0) circle (1.1pt);
\filldraw (0.9,0) circle (1.1pt);
\filldraw (1.2,0) circle (1.1pt);
\filldraw (1.5,0) circle (1.1pt);
\filldraw (0.9,0.3) circle (1.1pt);
\end{tikzpicture}
}

\def \Ehatseven{
\begin{tikzpicture}[baseline=-0.6ex]
\draw (0,0) -- (0.3,0) -- (0.6,0) -- (0.9,0) -- (1.2,0) -- (1.5,0) -- (1.8,0);
\draw (0.9,0) -- (0.9,0.3);
\filldraw (0,0) circle (1.1pt);
\filldraw (0.3,0) circle (1.1pt);
\filldraw (0.6,0) circle (1.1pt);
\filldraw (0.9,0) circle (1.1pt);
\filldraw (1.2,0) circle (1.1pt);
\filldraw (1.5,0) circle (1.1pt);
\filldraw (1.8,0) circle (1.1pt);
\filldraw (0.9,0.3) circle (1.1pt);
\end{tikzpicture}
}

\def \Eeight{
\begin{tikzpicture}[baseline=-0.6ex]
\draw (0,0) -- (0.3,0) -- (0.6,0) -- (0.9,0) -- (1.2,0) -- (1.5,0)-- (1.8,0);
\draw (0.6,0) -- (0.6,0.3);
\filldraw (0,0) circle (1.1pt);
\filldraw (0.3,0) circle (1.1pt);
\filldraw (0.6,0) circle (1.1pt);
\filldraw (0.9,0) circle (1.1pt);
\filldraw (1.2,0) circle (1.1pt);
\filldraw (1.5,0) circle (1.1pt);
\filldraw (1.8,0) circle (1.1pt);
\filldraw (0.6,0.3) circle (1.1pt);
\end{tikzpicture}
}

\def \Ehateight{
\begin{tikzpicture}[baseline=-0.6ex]
\draw (0,0) -- (0.3,0) -- (0.6,0) -- (0.9,0) -- (1.2,0) -- (1.5,0)-- (1.8,0) -- (2.1,0);
\draw (0.6,0) -- (0.6,0.3);
\filldraw (0,0) circle (1.1pt);
\filldraw (0.3,0) circle (1.1pt);
\filldraw (0.6,0) circle (1.1pt);
\filldraw (0.9,0) circle (1.1pt);
\filldraw (1.2,0) circle (1.1pt);
\filldraw (1.5,0) circle (1.1pt);
\filldraw (1.8,0) circle (1.1pt);
\filldraw (2.1,0) circle (1.1pt);
\filldraw (0.6,0.3) circle (1.1pt);
\end{tikzpicture}
}

\def \Pthree{
\begin{tikzpicture}[baseline=-0.6ex]
\draw (0,0) -- (0.3,0) -- (0.6,0);
\filldraw (0,0) circle (1.1pt);
\filldraw (0.3,0) circle (1.1pt);
\filldraw (0.6,0) circle (1.1pt);
\end{tikzpicture}
}

\def\KKKKPPr{
\begin{tikzpicture}[baseline=-0.6ex]
\draw (0,0) -- (0.3,0.3) -- (0.6,0) -- (0.9,0) -- (1.2,0);
\draw (0,0) -- (0.3,-0.3) -- (0.6,0);
\draw (0.3,0.3) -- (0.3,-0.3);
\draw (0,0) -- (0.6,0);
\filldraw (0,0) circle (1.1pt);
\filldraw (0.3,0.3) circle (1.1pt);
\filldraw (0.3,-0.3) circle (1.1pt);
\node[root] at (0.6,0) {};
\filldraw (0.9,0) circle (1.1pt);
\filldraw (1.2,0) circle (1.1pt);
\end{tikzpicture}
}

\def\KKKPPPr{
\begin{tikzpicture}[baseline=-0.6ex]
\draw (0.3,0) -- (0.6,0) -- (0.9,0) -- (1.2,0);
\draw (0.3,0) -- (0,0.3) -- (0,-0.3) -- (0.3,0);
\filldraw (0,0.3) circle (1.1pt);
\filldraw (0,-0.3) circle (1.1pt);
\node [root] at (0.3,0) {};
\filldraw (0.6,0) circle (1.1pt);
\filldraw (0.9,0) circle (1.1pt);
\filldraw (1.2,0) circle (1.1pt);
\end{tikzpicture}
}

\def\KKKPPPrl{
\begin{tikzpicture}[vertex/.style = {circle, fill, inner sep=2pt, outer sep=0pt},
every edge quotes/.style = {auto=left, sloped, font=\scriptsize, inner sep=1pt}]
\node[vertex,red] (o) at (1,0)     [label=above: $o$] {};
\node[vertex] (v1) at (0,1)     [label=left: $v_1$]  {};
\node[vertex] (v2) at (0,-1)     [label=left: $v_2$]  {};
\node[vertex] (u1) at (2,0)       [label=above: $u_1$] {};
\node[vertex] (u2) at (3,0)   [label=above: $u_2$] {};
\node[vertex] (u3) at (4,0)       [label=above: $u_3$] {};

\draw (o) -- (v1) -- (v2) -- (o) -- (u1) -- (u2) -- (u3);
\end{tikzpicture}
}

\def\KKKPPPP{
\begin{tikzpicture}[baseline=-0.6ex]
\draw (0.3,0) -- (0.6,0) -- (0.9,0) -- (1.2,0) -- (1.5,0);
\draw (0.3,0) -- (0,0.3) -- (0,-0.3) -- (0.3,0);
\filldraw (0,0.3) circle (1.1pt);
\filldraw (0,-0.3) circle (1.1pt);
\filldraw (0.3,0) circle (1.1pt);
\filldraw (0.6,0) circle (1.1pt);
\filldraw (0.9,0) circle (1.1pt);
\filldraw (1.2,0) circle (1.1pt);
\filldraw (1.5,0) circle (1.1pt);
\end{tikzpicture}
}

\def\KKKPPFF{
\begin{tikzpicture}[baseline=-0.6ex]
\draw (0.3,0) -- (0.6,0) -- (0.9,0) -- (1.2,0.3); 
\draw (0.9,0) -- (1.2,-0.3);
\draw (0.3,0) -- (0,0.3) -- (0,-0.3) -- (0.3,0);
\filldraw (0,0.3) circle (1.1pt);
\filldraw (0,-0.3) circle (1.1pt);
\filldraw (0.3,0) circle (1.1pt);
\filldraw (0.6,0) circle (1.1pt);
\filldraw (0.9,0) circle (1.1pt);
\filldraw (1.2,0.3) circle (1.1pt);
\filldraw (1.2,-0.3) circle (1.1pt);
\end{tikzpicture}
}

\def\KKKPKKK{
\begin{tikzpicture}[baseline=-0.6ex]
\draw (0.3,0) -- (0.6,0) -- (0.9,0) -- (1.2,0.3)  -- (1.2,-0.3) --(0.9,0);
\draw (0.3,0) -- (0,0.3) -- (0,-0.3) -- (0.3,0);
\filldraw (0,0.3) circle (1.1pt);
\filldraw (0,-0.3) circle (1.1pt);
\filldraw (0.3,0) circle (1.1pt);
\filldraw (0.6,0) circle (1.1pt);
\filldraw (0.9,0) circle (1.1pt);
\filldraw (1.2,0.3) circle (1.1pt);
\filldraw (1.2,-0.3) circle (1.1pt);
\end{tikzpicture}
}

\def\DsPPrs{
\begin{tikzpicture}[baseline=-0.6ex]
\draw (0,0) -- (0.3,0.3) -- (0.6,0) -- (0.9,0) -- (1.2,0);
\draw (0,0) -- (0.3,-0.3) -- (0.6,0);
\draw (0.3,0.3) -- (0.3,-0.3);
\filldraw (0,0) circle (1.1pt);
\filldraw (0.3,0.3) circle (1.1pt);
\filldraw (0.3,-0.3) circle (1.1pt);
\node [root] at (0.6,0) {};
\filldraw (0.9,0) circle (1.1pt);
\filldraw (1.2,0) circle (1.1pt);
\end{tikzpicture}
}

\def\DsPPrt{
\begin{tikzpicture}[baseline=-0.6ex]
\draw (0,0) -- (0.3,0.3) -- (0.6,0) -- (0.9,0) -- (1.2,0);
\draw (0,0) -- (0.3,-0.3) -- (0.6,0);
\draw (0.3,0.3) -- (0.3,-0.3);
\filldraw (0,0) circle (1.1pt);
\node [root] at (0.3,0.3) {};
\filldraw (0.3,-0.3) circle (1.1pt);
\filldraw (0.6,0) circle (1.1pt);
\filldraw (0.9,0) circle (1.1pt);
\filldraw (1.2,0) circle (1.1pt);
\end{tikzpicture}
}

\def\DsPPP{
\begin{tikzpicture}[baseline=-0.6ex]
\draw (0,0) -- (0.3,0.3) -- (0.6,0) -- (0.9,0) -- (1.2,0) -- (1.5,0);
\draw (0,0) -- (0.3,-0.3) -- (0.6,0);
\draw (0.3,0.3) -- (0.3,-0.3);
\filldraw (0,0) circle (1.1pt);
\filldraw (0.3,0.3) circle (1.1pt);
\filldraw (0.3,-0.3) circle (1.1pt);
\filldraw (0.6,0) circle (1.1pt);
\filldraw (0.9,0) circle (1.1pt);
\filldraw (1.2,0) circle (1.1pt);
\filldraw (1.5,0) circle (1.1pt);
\end{tikzpicture}
}

\def\DtPPrt{
\begin{tikzpicture}[baseline=-0.6ex]
\draw (0,0) -- (0.3,0.3) -- (0.6,0) -- (0.9,0) -- (1.2,0);
\draw (0,0) -- (0.3,-0.3) -- (0.6,0);
\draw (0,0) -- (0.6,0);
\filldraw (0,0) circle (1.1pt);
\filldraw (0.3,0.3) circle (1.1pt);
\filldraw (0.3,-0.3) circle (1.1pt);
\node [root] at (0.6,0) {};
\filldraw (0.9,0) circle (1.1pt);
\filldraw (1.2,0) circle (1.1pt);
\end{tikzpicture}
}

\def\DtPPP{
\begin{tikzpicture}[baseline=-0.6ex]
\draw (0,0) -- (0.3,0.3) -- (0.6,0) -- (0.9,0) -- (1.2,0) -- (1.5,0);
\draw (0,0) -- (0.3,-0.3) -- (0.6,0);
\draw (0,0) -- (0.6,0);
\filldraw (0,0) circle (1.1pt);
\filldraw (0.3,0.3) circle (1.1pt);
\filldraw (0.3,-0.3) circle (1.1pt);
\filldraw (0.6,0) circle (1.1pt);
\filldraw (0.9,0) circle (1.1pt);
\filldraw (1.2,0) circle (1.1pt);
\filldraw (1.5,0) circle (1.1pt);
\end{tikzpicture}
}

\def\KKKKP{
\begin{tikzpicture}[baseline=-0.6ex]
\draw (0,0) -- (0.3,0.3) -- (0.6,0) -- (0.9,0);
\draw (0,0) -- (0.3,-0.3) -- (0.6,0);
\draw (0.3,0.3) -- (0.3,-0.3);
\draw (0,0) -- (0.6,0);
\filldraw (0,0) circle (1.1pt);
\filldraw (0.3,0.3) circle (1.1pt);
\filldraw (0.3,-0.3) circle (1.1pt);
\filldraw (0.6,0) circle (1.1pt);
\filldraw (0.9,0) circle (1.1pt);
\end{tikzpicture}
}

\def\KKKKPr{
\begin{tikzpicture}[baseline=-0.6ex]
\draw (0,0) -- (0.3,0.3) -- (0.6,0) -- (0.9,0);
\draw (0,0) -- (0.3,-0.3) -- (0.6,0);
\draw (0.3,0.3) -- (0.3,-0.3);
\draw (0,0) -- (0.6,0);
\filldraw (0,0) circle (1.1pt);
\filldraw (0.3,0.3) circle (1.1pt);
\filldraw (0.3,-0.3) circle (1.1pt);
\node [root] at (0.6,0) {};
\filldraw (0.9,0) circle (1.1pt);
\end{tikzpicture}
}

\def\leftfork{
\begin{tikzpicture}[baseline=-0.6ex]
\filldraw (0,0.3) circle (1.1pt);
\filldraw (0,-0.3) circle (1.1pt);
\filldraw (0.3,0) circle (1.1pt);
\filldraw (0.6,0) circle (1.1pt);
\filldraw (0.9,0) circle (1.1pt);
\draw (0,0.3) -- (0.3,0) -- (0.6,0) -- (0.9,0);
\draw (0,-0.3) -- (0.3,0);
\end{tikzpicture}
}

\def\rightfork{
\begin{tikzpicture}[baseline=-0.6ex]
\filldraw (0,0) circle (1.1pt);
\filldraw (0.3,0) circle (1.1pt);
\filldraw (0.6,0) circle (1.1pt);
\filldraw (0.9,0.3) circle (1.1pt);
\filldraw (0.9,-0.3) circle (1.1pt);
\draw (0,0) -- (0.3,0) -- (0.6,0) -- (0.9,0.3);
\draw (0.6,0) -- (0.9,-0.3);
\end{tikzpicture}
}

\def\cherry{
\begin{tikzpicture}[baseline=-0.6ex]
\filldraw (0,0) circle (1.1pt);
\filldraw (0.3,0.3) circle (1.1pt);
\filldraw (0.3,-0.3) circle (1.1pt);
\draw (0,0) -- (0.3,0.3);
\draw (0,0) -- (0.3,-0.3);
\end{tikzpicture}
}

\def\tri{
\begin{tikzpicture}[baseline=-0.6ex]
\filldraw (0,0) circle (1.1pt);
\filldraw (0.3,0.3) circle (1.1pt);
\filldraw (0.3,-0.3) circle (1.1pt);
\draw (0,0) -- (0.3,0.3) -- (0.3,-0.3) -- (0,0);
\end{tikzpicture}
}

\def\PPPPPFF{
\begin{tikzpicture}[baseline=-0.6ex]
\draw (0,0) -- (0.3,0) -- (0.6,0) -- (0.9,0) -- (1.2,0) -- (1.5,0.3);
\draw (1.2,0) -- (1.5,-0.3);
\filldraw (0,0) circle (1.1pt);
\filldraw (0.3,0) circle (1.1pt);
\filldraw (0.6,0) circle (1.1pt);
\filldraw (0.9,0) circle (1.1pt);
\filldraw (1.2,0) circle (1.1pt);
\filldraw (1.5,0.3) circle (1.1pt);
\filldraw (1.5,-0.3) circle (1.1pt);
\end{tikzpicture}
}

\def\diamond{
\begin{tikzpicture}
\draw (0,0) -- (1,1) -- (2,0) -- (1,-1) -- (0,0);
\draw (1,1) -- (1,-1);
\filldraw (0,0) circle (2pt);
\filldraw (1,1) circle (2pt);
\filldraw (1,-1) circle (2pt);
\filldraw (2,0) circle (2pt);
\node at (1.3,1.1) {$t$};
\node at (2.3,0) {$s$};
\end{tikzpicture}
}

\def\figKKKKP{
\begin{tikzpicture}[scale=0.5]
\draw (0,0) -- (1,1) -- (2,0) -- (1,-1) -- (0,0);
\draw (1,1) -- (1,-1);
\draw (0,0) -- (2,0) -- (3.5,0);
\filldraw (0,0) circle (3pt);
\filldraw (1,1) circle (3pt);
\filldraw (1,-1) circle (3pt);
\filldraw (2,0) circle (3pt);
\filldraw (3.5,0) circle (3pt);
\node at (2,-0.5) {$o$};
\node at (3.5,-0.5) {$v$};
\end{tikzpicture}
}

\def\figSP{
\begin{tikzpicture}[scale=0.5]
\draw (1,0) -- (0.2,1);
\draw (1,0) -- (0.2,-1);
\draw (1,0) -- (1.5,1);
\draw (1,0) -- (1.5,-1);
\draw (1,0) -- (2.3,0) -- (3.3,0) -- (4.3,0) -- (5.3,0);
\filldraw (1,0) circle (3pt);
\filldraw (0.2,1) circle (3pt);
\filldraw (0.2,-1) circle (3pt);
\filldraw (1.5,1) circle (3pt);
\filldraw (1.5,-1) circle (3pt);
\filldraw (2.3,0) circle (3pt);
\filldraw (3.3,0) circle (3pt);
\filldraw (4.3,0) circle (3pt);
\filldraw (5.3,0) circle (3pt);
\node at (0.4,0) {$o$};
\node at (5.9,0) {$v$};
\end{tikzpicture}
}

\def\KKKKoGG{
\begin{tikzpicture}
\draw (5,0) ellipse (2.5cm and 1cm);
\draw (0,0) -- (1,1) -- (2,0) -- (1,-1) -- (0,0);
\draw (1,1) -- (1,-1);
\draw (0,0) -- (2,0) -- (3,0);
\draw (4,0.5) -- (4.5,-0.4);
\draw (5.5,-0.7) -- (6,0.1);
\filldraw (0,0) circle (2pt);
\filldraw (1,1) circle (2pt);
\filldraw (1,-1) circle (2pt);
\filldraw (2,0) circle (2pt);
\filldraw (3,0) circle (2pt);
\filldraw (4,0.5) circle (2pt);
\filldraw (4.5,-0.4) circle (2pt);
\filldraw (5.5,-0.7) circle (2pt);
\filldraw (6,0.1) circle (2pt);
\node at (3,0.3) {$u$};
\node at (3.5,0) {$\cdots$};
\node at (5,0) {$\cdots$};
\node at (2.1,0.3) {$o$};
\node at (8,0) {$G'$};
\end{tikzpicture}
}

\def\PP{
\begin{tikzpicture}[baseline=-0.6ex]
\draw (0,0) -- (0.3,0);
\filldraw (0,0) circle (1.1pt);
\filldraw (0.3,0) circle (1.1pt);
\end{tikzpicture}
}

\def\PPP{
\begin{tikzpicture}[baseline=-0.6ex]
\draw (0,0) -- (0.3,0) -- (0.6,0);
\filldraw (0,0) circle (1.1pt);
\filldraw (0.3,0) circle (1.1pt);
\filldraw (0.6,0) circle (1.1pt);
\end{tikzpicture}
}

\def\KKK{
\begin{tikzpicture}[baseline=-0.6ex]
\draw (0.3,0) -- (0,0.3) -- (0,-0.3) -- (0.3,0);
\filldraw (0,0.3) circle (1.1pt);
\filldraw (0,-0.3) circle (1.1pt);
\filldraw (0.3,0) circle (1.1pt);
\end{tikzpicture}
}

\def\PFF{
\begin{tikzpicture}[baseline=-0.6ex]
\draw (0.3,0) -- (0.6,0);
\draw (0.3,0) -- (0,0.3);
\draw (0,-0.3) -- (0.3,0);
\filldraw (0,0.3) circle (1.1pt);
\filldraw (0,-0.3) circle (1.1pt);
\filldraw (0.3,0) circle (1.1pt);
\filldraw (0.6,0) circle (1.1pt);
\end{tikzpicture}
}

\def\KKKP{
\begin{tikzpicture}[baseline=-0.6ex]
\draw (0.3,0) -- (0.6,0);
\draw (0.3,0) -- (0,0.3) -- (0,-0.3) -- (0.3,0);
\filldraw (0,0.3) circle (1.1pt);
\filldraw (0,-0.3) circle (1.1pt);
\filldraw (0.3,0) circle (1.1pt);
\filldraw (0.6,0) circle (1.1pt);
\end{tikzpicture}
}

\def\PPPP{
\begin{tikzpicture}[baseline=-0.6ex]
\draw (0,0) -- (0.3,0) -- (0.6,0) -- (0.9,0);
\filldraw (0,0) circle (1.1pt);
\filldraw (0.3,0) circle (1.1pt);
\filldraw (0.6,0) circle (1.1pt);
\filldraw (0.9,0) circle (1.1pt);
\end{tikzpicture}
}

\def\PPPPP{
\begin{tikzpicture}[baseline=-0.6ex]
\draw (0,0) -- (0.3,0) -- (0.6,0) -- (0.9,0) -- (1.2,0);
\filldraw (0,0) circle (1.1pt);
\filldraw (0.3,0) circle (1.1pt);
\filldraw (0.6,0) circle (1.1pt);
\filldraw (0.9,0) circle (1.1pt);
\filldraw (1.2,0) circle (1.1pt);
\end{tikzpicture}
}

\def\Dsm{
\begin{tikzpicture}[baseline=-0.6ex]
\draw (0,0) -- (0.3,0.3) -- (0.6,0);
\draw (0,0) -- (0.3,-0.3) -- (0.6,0);
\draw (0,0) -- (0.6,0);
\filldraw (0,0) circle (1.1pt);
\filldraw (0.3,0.3) circle (1.1pt);
\filldraw (0.3,-0.3) circle (1.1pt);
\filldraw (0.6,0) circle (1.1pt);
\end{tikzpicture}
}

\def\KKKPP{
\begin{tikzpicture}[baseline=-0.6ex]
\draw (0.3,0) -- (0.6,0) -- (0.9,0);
\draw (0.3,0) -- (0,0.3) -- (0,-0.3) -- (0.3,0);
\filldraw (0,0.3) circle (1.1pt);
\filldraw (0,-0.3) circle (1.1pt);
\filldraw (0.3,0) circle (1.1pt);
\filldraw (0.6,0) circle (1.1pt);
\filldraw (0.9,0) circle (1.1pt);
\end{tikzpicture}
}

\def\PKKKP{
\begin{tikzpicture}[baseline=-0.6ex]
\draw (0,0) -- (0.3,0.3) -- (0.6,0.3);
\draw (0,0) -- (0.3,-0.3) -- (0.6,-0.3);
\draw (0.3,0.3) -- (0.3,-0.3);
\filldraw (0,0) circle (1.1pt);
\filldraw (0.3,0.3) circle (1.1pt);
\filldraw (0.3,-0.3) circle (1.1pt);
\filldraw (0.6,0.3) circle (1.1pt);
\filldraw (0.6,-0.3) circle (1.1pt);
\end{tikzpicture}
}

\def\KKKFF{
\begin{tikzpicture}[baseline=-0.6ex]
\draw (0.3,0) -- (0.6,0.3);
\draw (0.3,0) -- (0.6,-0.3);
\draw (0.3,0) -- (0,0.3) -- (0,-0.3) -- (0.3,0);
\filldraw (0,0.3) circle (1.1pt);
\filldraw (0,-0.3) circle (1.1pt);
\filldraw (0.3,0) circle (1.1pt);
\filldraw (0.6,0.3) circle (1.1pt);
\filldraw (0.6,-0.3) circle (1.1pt);
\end{tikzpicture}
}

\def\SSSSS{
\begin{tikzpicture}[baseline=-0.6ex]
\draw (0.3,0) -- (0.6,0.3);
\draw (0.3,0) -- (0.6,-0.3);
\draw (0.3,0) -- (0,0.3);
\draw (0.3,0) -- (0,-0.3);
\filldraw (0,0.3) circle (1.1pt);
\filldraw (0,-0.3) circle (1.1pt);
\filldraw (0.3,0) circle (1.1pt);
\filldraw (0.6,0.3) circle (1.1pt);
\filldraw (0.6,-0.3) circle (1.1pt);
\end{tikzpicture}
}

\def\SSSSSsm{
\begin{tikzpicture}[baseline=-0.6ex]
\draw (0.2,0) -- (0.4,0.2);
\draw (0.2,0) -- (0.4,-0.2);
\draw (0.2,0) -- (0,0.2);
\draw (0.2,0) -- (0,-0.2);
\filldraw (0,0.2) circle (1.1pt);
\filldraw (0,-0.2) circle (1.1pt);
\filldraw (0.2,0) circle (1.1pt);
\filldraw (0.4,0.2) circle (1.1pt);
\filldraw (0.4,-0.2) circle (1.1pt);
\end{tikzpicture}
}

\def\SSSSSPP{
\begin{tikzpicture}[baseline=-0.6ex]
\draw (0.2,0) -- (0.4,0.3);
\draw (0.2,0) -- (0.4,-0.3);
\draw (0.2,0) -- (0,0.3);
\draw (0,-0.3) -- (0.2,0) -- (0.5,0) -- (0.8,0);
\filldraw (0,0.3) circle (1.1pt);
\filldraw (0,-0.3) circle (1.1pt);
\filldraw (0.2,0) circle (1.1pt);
\filldraw (0.4,0.3) circle (1.1pt);
\filldraw (0.4,-0.3) circle (1.1pt);
\filldraw (0.5,0) circle (1.1pt);
\filldraw (0.8,0) circle (1.1pt);
\end{tikzpicture}
}

\def\SSSSSPPr{
\begin{tikzpicture}[baseline=-0.6ex]
\draw (0.2,0) -- (0.4,0.3);
\draw (0.2,0) -- (0.4,-0.3);
\draw (0.2,0) -- (0,0.3);
\draw (0,-0.3) -- (0.2,0) -- (0.5,0) -- (0.8,0);
\filldraw (0,0.3) circle (1.1pt);
\filldraw (0,-0.3) circle (1.1pt);
\node [root] at (0.2,0) {};
\filldraw (0.4,0.3) circle (1.1pt);
\filldraw (0.4,-0.3) circle (1.1pt);
\filldraw (0.5,0) circle (1.1pt);
\filldraw (0.8,0) circle (1.1pt);
\end{tikzpicture}
}

\def\treekernelA{
\begin{tikzpicture}[baseline=-0.6ex]
\draw (0.2,0) -- (0.4,0.3);
\draw (0.2,0) -- (0.4,-0.3);
\draw (0.2,0) -- (0,0.3);
\draw (0,-0.3) -- (0.2,0) -- (0.5,0) -- (0.8,0) -- (1.1,0) -- (1.4,0) -- (1.7,0);
\filldraw (0,0.3) circle (1.1pt);
\filldraw (0,-0.3) circle (1.1pt);
\node [root] at (0.2,0) {};
\filldraw (0.4,0.3) circle (1.1pt);
\filldraw (0.4,-0.3) circle (1.1pt);
\filldraw (0.5,0) circle (1.1pt);
\filldraw (0.8,0) circle (1.1pt);
\filldraw (1.1,0) circle (1.1pt);
\filldraw (1.4,0) circle (1.1pt);
\filldraw (1.7,0) circle (1.1pt);
\end{tikzpicture}
}

\def\treekernelB{
\begin{tikzpicture}[baseline=-0.6ex]
\draw (0.3,0) -- (0.5,0.3);
\draw (0.3,0) -- (0.5,-0.3);
\draw (0.3,0) -- (0.1,0.3);
\draw (0.3,0) -- (0.1,-0.3);
\draw (0,0) -- (0.3,0) -- (0.6,0) -- (0.9,0) -- (1.2,0) -- (1.5,0);
\filldraw (0.1,0.3) circle (1.1pt);
\filldraw (0.1,-0.3) circle (1.1pt);
\filldraw (0,0) circle (1.1pt);
\node [root] at (0.3,0) {};
\filldraw (0.5,0.3) circle (1.1pt);
\filldraw (0.5,-0.3) circle (1.1pt);
\filldraw (0.6,0) circle (1.1pt);
\filldraw (0.9,0) circle (1.1pt);
\filldraw (1.2,0) circle (1.1pt);
\filldraw (1.5,0) circle (1.1pt);
\end{tikzpicture}
}

\def\treekernelC{
\begin{tikzpicture}[baseline=-0.6ex]
\draw (0.2,0) -- (0.4,0.3) -- (0.7,0.3);
\draw (0.2,0) -- (0.4,-0.3);
\draw (0.2,0) -- (0,0.3);
\draw (0,-0.3) -- (0.2,0) -- (0.5,0) -- (0.8,0) -- (1.1,0) -- (1.4,0);
\filldraw (0,0.3) circle (1.1pt);
\filldraw (0,-0.3) circle (1.1pt);
\node [root] at (0.2,0) {};
\filldraw (0.4,0.3) circle (1.1pt);
\filldraw (0.7,0.3) circle (1.1pt);
\filldraw (0.4,-0.3) circle (1.1pt);
\filldraw (0.5,0) circle (1.1pt);
\filldraw (0.8,0) circle (1.1pt);
\filldraw (1.1,0) circle (1.1pt);
\filldraw (1.4,0) circle (1.1pt);
\end{tikzpicture}
}

\def\FFPPP{
\begin{tikzpicture}[baseline=-0.6ex]
\draw (0,0.3) -- (0.3,0) -- (0.6,0) -- (0.9,0);
\draw (0,-0.3) -- (0.3,0);
\filldraw (0,0.3) circle (1.1pt);
\filldraw (0,-0.3) circle (1.1pt);
\filldraw (0.3,0) circle (1.1pt);
\filldraw (0.6,0) circle (1.1pt);
\filldraw (0.9,0) circle (1.1pt);
\end{tikzpicture}
}

\def\DsP{
\begin{tikzpicture}[baseline=-0.6ex]
\draw (0,0) -- (0.3,0.3) -- (0.6,0) -- (0.9,0);
\draw (0,0) -- (0.3,-0.3) -- (0.6,0);
\draw (0.3,0.3) -- (0.3,-0.3);
\filldraw (0,0) circle (1.1pt);
\filldraw (0.3,0.3) circle (1.1pt);
\filldraw (0.3,-0.3) circle (1.1pt);
\filldraw (0.6,0) circle (1.1pt);
\filldraw (0.9,0) circle (1.1pt);
\end{tikzpicture}
}

\def\DtP{
\begin{tikzpicture}[baseline=-0.6ex]
\draw (0,0) -- (0.3,0.3) -- (0.6,0) -- (0.9,0);
\draw (0,0) -- (0.3,-0.3) -- (0.6,0);
\draw (0,0) -- (0.6,0);
\filldraw (0,0) circle (1.1pt);
\filldraw (0.3,0.3) circle (1.1pt);
\filldraw (0.3,-0.3) circle (1.1pt);
\filldraw (0.6,0) circle (1.1pt);
\filldraw (0.9,0) circle (1.1pt);
\end{tikzpicture}
}

\def\CCCCP{
\begin{tikzpicture}[baseline=-0.6ex]
\draw (0,0) -- (0.3,0.3) -- (0.6,0) -- (0.9,0);
\draw (0,0) -- (0.3,-0.3) -- (0.6,0);
\filldraw (0,0) circle (1.1pt);
\filldraw (0.3,0.3) circle (1.1pt);
\filldraw (0.3,-0.3) circle (1.1pt);
\filldraw (0.6,0) circle (1.1pt);
\filldraw (0.9,0) circle (1.1pt);
\end{tikzpicture}
}

\def\KKKKPP{
\begin{tikzpicture}[baseline=-0.6ex]
\draw (0,0) -- (0.3,0.3) -- (0.6,0) -- (0.9,0) -- (1.2,0);
\draw (0,0) -- (0.3,-0.3) -- (0.6,0);
\draw (0.3,0.3) -- (0.3,-0.3);
\draw (0,0) -- (0.6,0);
\filldraw (0,0) circle (1.1pt);
\filldraw (0.3,0.3) circle (1.1pt);
\filldraw (0.3,-0.3) circle (1.1pt);
\filldraw (0.6,0) circle (1.1pt);
\filldraw (0.9,0) circle (1.1pt);
\filldraw (1.2,0) circle (1.1pt);
\end{tikzpicture}
}

\def\KKKPPP{
\begin{tikzpicture}[baseline=-0.6ex]
\draw (0.3,0) -- (0.6,0) -- (0.9,0) -- (1.2,0);
\draw (0.3,0) -- (0,0.3) -- (0,-0.3) -- (0.3,0);
\filldraw (0,0.3) circle (1.1pt);
\filldraw (0,-0.3) circle (1.1pt);
\filldraw (0.3,0) circle (1.1pt);
\filldraw (0.6,0) circle (1.1pt);
\filldraw (0.9,0) circle (1.1pt);
\filldraw (1.2,0) circle (1.1pt);
\end{tikzpicture}
}

\def\DsPP{
\begin{tikzpicture}[baseline=-0.6ex]
\draw (0,0) -- (0.3,0.3) -- (0.6,0) -- (0.9,0) -- (1.2,0);
\draw (0,0) -- (0.3,-0.3) -- (0.6,0);
\draw (0.3,0.3) -- (0.3,-0.3);
\filldraw (0,0) circle (1.1pt);
\filldraw (0.3,0.3) circle (1.1pt);
\filldraw (0.3,-0.3) circle (1.1pt);
\filldraw (0.6,0) circle (1.1pt);
\filldraw (0.9,0) circle (1.1pt);
\filldraw (1.2,0) circle (1.1pt);
\end{tikzpicture}
}

\def\PKKKPP{
\begin{tikzpicture}[baseline=-0.6ex]
\draw (0,0) -- (0.3,0.3) -- (0.6,0.3) -- (0.9,0.3);
\draw (0,0) -- (0.3,-0.3) -- (0.6,-0.3);
\draw (0.3,0.3) -- (0.3,-0.3);
\filldraw (0,0) circle (1.1pt);
\filldraw (0.3,0.3) circle (1.1pt);
\filldraw (0.3,-0.3) circle (1.1pt);
\filldraw (0.6,0.3) circle (1.1pt);
\filldraw (0.6,-0.3) circle (1.1pt);
\filldraw (0.9,0.3) circle (1.1pt);
\end{tikzpicture}
}

\def\DtPP{
\begin{tikzpicture}[baseline=-0.6ex]
\draw (0,0) -- (0.3,0.3) -- (0.6,0) -- (0.9,0) -- (1.2,0);
\draw (0,0) -- (0.3,-0.3) -- (0.6,0);
\draw (0,0) -- (0.6,0);
\filldraw (0,0) circle (1.1pt);
\filldraw (0.3,0.3) circle (1.1pt);
\filldraw (0.3,-0.3) circle (1.1pt);
\filldraw (0.6,0) circle (1.1pt);
\filldraw (0.9,0) circle (1.1pt);
\filldraw (1.2,0) circle (1.1pt);
\end{tikzpicture}
}

\def\KKKFPF{
\begin{tikzpicture}[baseline=-0.6ex]
\draw (0.3,0) -- (0.6,0.3) -- (0.9,0.3);
\draw (0.3,0) -- (0.6,-0.3);
\draw (0.3,0) -- (0,0.3) -- (0,-0.3) -- (0.3,0);
\filldraw (0,0.3) circle (1.1pt);
\filldraw (0,-0.3) circle (1.1pt);
\filldraw (0.3,0) circle (1.1pt);
\filldraw (0.6,0.3) circle (1.1pt);
\filldraw (0.6,-0.3) circle (1.1pt);
\filldraw (0.9,0.3) circle (1.1pt);
\end{tikzpicture}
}

\def\deltastar{
\begin{tikzpicture}[baseline=-0.6ex]
\draw (0,0) -- (0.3,0) -- (0.6,0.3) -- (0.9,0.3);
\draw (0.3,0) -- (0.6,-0.3) -- (0.9,-0.3);
\draw (0.6,0.3) -- (0.6,-0.3);
\filldraw (0,0) circle (1.1pt);
\filldraw (0.3,0) circle (1.1pt);
\filldraw (0.6,0.3) circle (1.1pt);
\filldraw (0.6,-0.3) circle (1.1pt);
\filldraw (0.9,0.3) circle (1.1pt);
\filldraw (0.9,-0.3) circle (1.1pt);
\end{tikzpicture}
}

\def\PKKKKP{
\begin{tikzpicture}[baseline=-0.6ex]
\draw (0,0) -- (0.3,0) -- (0.6,0.3) -- (0.9,0) -- (1.2,0);
\draw (0.3,0) -- (0.6,-0.3) -- (0.9,0);
\draw (0.6,0.3) -- (0.6,-0.3);
\draw (0.3,0) -- (0.9,0);
\filldraw (0,0) circle (1.1pt);
\filldraw (0.3,0) circle (1.1pt);
\filldraw (0.6,0.3) circle (1.1pt);
\filldraw (0.6,-0.3) circle (1.1pt);
\filldraw (0.9,0) circle (1.1pt);
\filldraw (1.2,0) circle (1.1pt);
\end{tikzpicture}
}

\begin{document}

\title{The least balanced graphs and trees}

\author{P\'{e}ter Csikv\'{a}ri}
\thanks{This work was supported by the MTA-R\'enyi Counting in Sparse Graphs ``Momentum'' Research Group. The first author was also supported by the Dynasnet ERC Synergy project (ERC-2018-SYG 810115) and the Hungarian National Research, Development and Innovation Office (Advanced grant 153378).}
\address{HUN-REN Alfr\'ed R\'enyi Institute of Mathematics, Budapest, Hungary and ELTE: E\"otv\"os Lor\'and University, Mathematics Institute, Department of Computer Science}
\email{peter.csikvari@gmail.com}

\author{Viktor Harangi}
\thanks{The second author was also supported by NRDI 
grant KKP 138270, and the Hungarian Academy of Sciences (J\'anos Bolyai Scholarship).}
\address{HUN-REN Alfr\'ed R\'enyi Institute of Mathematics, Budapest, Hungary} 
\email{harangi@renyi.hu}


\begin{abstract}
Given a connected graph, the principal eigenvector of the adjacency matrix (often called the Perron vector) can be used to assign positive weights to the vertices. A natural way to measure the homogeneousness of this vector is by considering the ratio of its $\ell^1$ and $\ell^2$ norms. 

It is easy to see that the most balanced graphs in this sense (i.e., the ones with the largest ratio) are the regular graphs. What can we say about the least balanced (or most centralized) graphs with the smallest ratio? It was conjectured by R\"ucker, R\"ucker and Gutman that, for any given $n \geq 6$, among $n$-vertex connected graphs the smallest ratio is achieved by the complete graph $K_4$ with a single path $P_{n-4}$ attached to one of its vertices. In this paper we confirm this conjecture.

We also verify the analogous conjecture for trees: for any given $n \geq 8$, among $n$-vertex trees the smallest ratio is achieved by the star graph $S_5$ with a path $P_{n-5}$ attached to its central vertex.
\end{abstract}

\maketitle

\section{Introduction} \label{sec:intro}

Let $G$ be a finite simple connected graph with vertex set $V(G)$. Its adjacency matrix $A_G$ is symmetric, and hence by the Frobenius--Perron theorem we know that $x_v>0$ for the principal eigenvector $\x=\big( x_v \big)_{v \in V(G)}$ corresponding to the largest eigenvalue. 

It is standard to use the entries of the principal eigenvector as a global ranking of the vertices expressing their importance or centrality. Depending on the context, the rank may express the amount that a vertex contributes to the graph's ability to spread a disease or information, or propagate influence. With this in mind, it is often helpful to know how large the relevant part of the graph is. One of the most standard ways to measure the size of the ``effective support'' of a vector (and hence quantifying how sparse/concentrated versus how spread-out/homogeneous the vector is) is by considering (the square of) the ratio of the $\ell^1$ and $\ell^2$ norms. This motivates the introduction of the following quantity.
\begin{Def} \label{def:effective_order}
We define the \emph{effective order} of a graph $G$ as 
\begin{equation} \label{eq:G_ratio_def}
\Gamma_G := \frac{\| \x\|_1^2}{\| \x \|_2^2} 
= \frac{\big( \sum_{v \in V(G)} x_v \big)^2}{\sum_{v \in V(G)} x_v^2} ,
\end{equation}
where $\x$ is the principal eigenvector of the adjacency matrix of $G$. 
\end{Def}
We propose the name \emph{effective order} because, in many respects, this quantity more accurately reflects the true size of the graph than merely considering the order of $G$ (i.e., the number of vertices). Loosely speaking, $\Ga_G$ counts the number of influential vertices in $G$; smaller $\Ga_G$ means that $G$ is more centralized, while larger $\Ga_G$ indicates a more balanced graph. Hence the effective order is a robust indicator of networks showing how many nodes have true relevance. This is a single number, yet it reveals a great deal about the nature of the network at hand. This quantity (in slightly different forms) has been proposed and studied in a number of papers; see Section~\ref{sec:background} for details. 

In this paper we work out methods for estimating the effective order $\Ga_G$. Our most notable tool uses the resolvent matrix. We present our results in a general form with the primary goal to settle two conjectures regarding the smallest possible values of $\Ga_G$ among connected graphs and trees. 

First of all, note that $\Gamma_G \leq |V(G)|$ by the Cauchy--Schwarz inequality. Therefore, the effective order is always upper bounded by the order, and equality holds if and only if the graph is regular (i.e., each degree is the same). This means that, unsurprisingly, regular graphs are the most balanced ones. What can we say about the minimal possible value of $\Gamma_G$ among $n$-vertex connected graphs? \textbf{Which connected graphs are the most centralized (or the least balanced)?} 
It was conjectured by R\"ucker, R\"ucker and Gutman \cite{rucker2002kites} that for any given $n \geq 6$, the minimum is given by the graph consisting of $K_4$ and a pendant path. 
The same conjecture was also proposed by David Gregory independently (S.~Cioab\u{a}, personal communication), and it gained some attention recently, see \cite{cioaba2021principal,clark2024comparing,lele2023unsolved}. 

We use the standard notation $P_k$ for the path graph with $k$ vertices and $P_\infty$ for the infinite (one-ended) path. Given a graph $H$ with a vertex $v \in V(H)$ we use the notation $H +_v P_k$ for the graph obtained from the disjoint union of $H$ and $P_k$ by adding an edge connecting $v$ and an endpoint of $P_k$. Note that $H +_v P_k$ has $|V(H)|+k$ vertices. We define $H +_v P_\infty$ similarly. If $H$ is a complete graph, then we may omit $v$ from the notation. Also, if $H$ has a unique largest-degree vertex $v$, then $H+ P_k$ will simply mean $H +_v P_k$.

\begin{Conj}[R\"ucker, R\"ucker and Gutman \cite{rucker2002kites}] \label{conj:minimal_Ga}
If $n\geq 6$, then among $n$-vertex connected graphs, $\Gamma_G$ is minimized by 
\[ G = \, \KKKKPP \cdots \PPP \,= K_4 + P_{n-4}. \]
\end{Conj}
We may define $\Gamma_G$ to be $\| \x \|_1^2 \big/ \| \x \|_2^2$ even for infinite graphs provided that the adjacency operator has an $\ell^1$ eigenvector $\x$ corresponding to the top of the spectrum. This can be done for $G = K_4 + P_{\infty}$ and it turns out that
\[ \Gamma_{K_4 + P_\infty} 
= \frac{5+3\sqrt{3}}{2} \approx 5.098076 
\quad \text{and} \quad 
\Gamma_{K_4 + P_k}  \nearrow \Gamma_{K_4 + P_\infty} 
\text{ as } k \to \infty .\]

In this paper we settle Conjecture \ref{conj:minimal_Ga} by proving the following.
\begin{Thm} \label{thm:only_K4Pk}
Let $n\geq 7$ and $G$ be an $n$-vertex connected graph. Then 
\[ \Gamma_G<\Gamma_{K_4 + P_\infty} = \frac{5+3\sqrt{3}}{2} 
\quad \Longleftrightarrow \quad 
G \cong \, \KKKKPP \cdots \PPP \,= K_4 + P_{n-4} . \]
In particular, among $n$-vertex connected graphs, $\Gamma_G$ is minimized by $G = K_4 + P_{n-4}$. (This latter statement holds for $n\geq 6$ already.)
\end{Thm}

We also settle another conjecture of R\"ucker, R\"ucker and Gutman \cite{rucker2002kites}
that, for any given $n\geq 8$, among $n$-vertex trees the graph $S_5+P_{n-5}$ minimizes $\Gamma_G$, where $S_5 = \SSSSSsm$ is the star on $5$ vertices. Again, one can compute the limiting ratio:
\[ \Gamma_{S_5 + P_{\infty}} 
= 4+2\sqrt{3} \approx 7.4641 
\quad \text{and we have} \quad 
\Gamma_{S_5 + P_k}  \nearrow \Gamma_{S_5 + P_\infty} 
\text{ as } k \to \infty .\]
As before, we prove the conjecture by showing that, for large enough $n$, there is only one graph below the limiting ratio.
\begin{Thm} \label{thm:only_S5Pk}
Let $n\geq 14$ and $T$ be a tree on $n$ vertices. Then
\[ \Gamma_T<\Gamma_{S_5+P_{\infty}}=4+2\sqrt{3} 
\quad \Longleftrightarrow \quad 
T \cong \, \SSSSSPP \cdots \PP = S_5+P_{n-5} . \]
In particular, among $n$-vertex trees, $\Gamma_T$ is minimized by $T= S_5 + P_{n-5}$. (This latter statement holds for $n\geq 8$ already.)
\end{Thm}

\subsection{Background} \label{sec:background}

When we want to define a quantity for measuring how unbalanced or how centralized a network is, we basically need to make two choices.
\begin{enumerate}
\item How to rank the vertices? That is, how to assign weights to the vertices expressing their importance?
\item Once a ranking (i.e., a vector of nonnegative weights) is given, we need a way to measure how concentrated this vector is. 
\end{enumerate}

\subsubsection*{The ranking.} 
When one tries to identify influential nodes in a network, a natural idea is to consider nodes with large degrees. However, having many connections is not quite the same as being important or influential. The ranking system should reflect the fact that connections to more influential nodes matter more. For instance, we may demand the weight of a vertex to be proportional to the sum of the weights of its neighbors. In a connected graph the unique\footnote{up to scaling/normalization} vector (consisting of nonnegative weights) with this property is the principal eigenvector $\x$ of the adjacency matrix. Also note that, unless $G$ is bipartite, we have $A^k / \la^k \to \x \x^\top$ as $k \to \infty$. This actually shows that repeated propagation of importance along edges converges exactly to the principal eigenvector ranking. 

The above ranking is often called \emph{eigenvector centrality}, and is broadly used in various fields. It is a powerful tool for understanding network structure and dynamics, and has countless applications. Below, we list a few typical areas of usage, along with selected references.
\begin{itemize}
\item \textbf{Sociology:} for analyzing social networks (e.g., identifying key influencers in online social platforms, studying leadership or prestige in organizational networks, modeling information diffusion). 
\cite{bonacich1972factoring,bonacich2007some,maharani2014degree}
\item \textbf{Web science:} for ranking web pages. For instance, Google's PageRank system is a directed variant of eigenvector centrality. 
\cite{brin1998anatomy,langville2006google} 
\item \textbf{Biology and bioinformatics:} for identifying essential or influential biological entities in various networks such as protein--protein interaction (PPI) networks or gene regulatory networks. 
\cite{ozgur2008identifying,li2012new,ashtiani2018systematic}
\item \textbf{Epidemiology:} for identifying ``super-spreaders'' or high-impact individuals (or locations) in disease transmission networks as well as for working out vaccination strategies. For instance, the epidemic threshold in SIS models is closely related to the largest eigenvalue, while the entries of the principal eigenvector correspond to steady-state infection probabilities. 
\cite{van2012epidemic,chaharborj2022controlling}
\item \textbf{Transportation:} for urban (transportation) planning or road/rail network vulnerability analysis. 
\cite{cheng2015measuring,agryzkov2019centrality}
\item \textbf{Economics:} for modeling interdependencies among institutions, markets, firms, countries, and sectors. 
\cite{engel2021network,fagiolo2023centrality}
\item \textbf{Neuroscience:} for identifying key brain regions and studying neurological disorders. 
\cite{lohmann2010eigenvector,zuo2012network}
\item \textbf{Physics, complex systems, network science:} for describing various dynamical processes such as \emph{diffusion processes} or \emph{synchronization} (Kuramoto-like models), for understanding phase transitions (percolation and structural transitions can depend on eigenvector localization in a number of models), and for analyzing network robustness and resilience. \cite{newman2010networks,martin2014localization,pradhan2017optimized,young2024dynamical}
\end{itemize}

\subsubsection*{Measuring centralization} 
Now that we have a ranking (i.e., weights assigned to the vertices expressing their relevance), how can we use it to get to know our network? Is our network balanced or centralized? Is there a good way to measure this?

Given a vector $\y = (y_i)$, a widely used proxy for sparsity is the following quantity:
\begin{equation} \label{eq:Ga_y}
\Ga(\y) := \frac{\| \y \|_1^2}{\| \y \|_2^2} 
= \frac{\big( \sum |y_i| \big)^2}{\sum y_i^2}
\end{equation}
It can be viewed as a continuous relaxation of the size of the support. Indeed, if all nonzero entries of $\y$ are equal, then $\Ga(\y)$ is simply the support size, that is, the number of nonzero entries. In general, it measures the ``effective support size'': the quantity of entries that ``meaningfully contribute'' to $\y$. Accordingly, the effective order $\Ga_G = \Ga(\x)$ can be viewed as the number of vertices that meaningfully contribute to $G$. 

There are several slightly different but equivalent variants of this quantity used in the literature. The so-called \emph{coefficient of variation} is a standard notion in various fields from economic
inequality \cite{atkinson1970measurement} to measurement reliability in clinical epidemiology \cite{shechtman2013coefficient} and machine learning \cite{bindu2019coefficient}.

The \emph{Hoyer sparsity measure}, introduced in \cite{hoyer2004non}, is a shifted and scaled version of $\|y\|_1 \big/ \|y\|_2$, and hence also equivalent to the above quantities. Yet another related notion is the \emph{signal-to-noise ratio (SNR)}, which plays an important role in information theory \cite{shannon1949communication} and image processing \cite{johari4developing}. In certain situations, it can be defined as the reciprocal of the coefficient of variation. 

The use of $\Ga(\y)$ or equivalent quantities is very natural but one may measure the sparsity of a vector $\y$ in other ways; see \cite{hurley2009comparing} for an overview of different sparsity measures. Here we mention one further natural candidate: the \emph{Shannon entropy} $\sum_i -p_i \log(p_i)$ of the probability distribution corresponding to the normalized entries $p_i := |y_i| / \|y\|_1$. If $\y$ is constant on its support $S$, then the entropy is $\log(|S|)$. Therefore, we may view it as a continuous relaxation of the logarithm of the support size. As far as we know, this quantity has not been investigated in the context of eigenvector centrality. It looks plausible that it is minimized by the same connected graph $K_4 + P_{n-4}$.

\subsubsection*{Measures of irregularity}
We now turn to results regarding the principal eigenvector $\x$ of the adjacency matrix of a graph. 

Clark used the coefficient of variation to measure the irregularity of $\x$. As he points out in \cite[Lemma~2.1]{clark2024comparing}, the square of this coefficient (denoted by $c_e$) can be expressed in terms of the $\ell^1$ and $\ell^2$ norms of $\x$ as follows:
\[ c_e = \big| V(G) \big| \frac{\| \x \|_2^2}{\| \x \|_1^2} - 1 
\text{, which is equal to }
\frac{\big| V(G) \big|}{\Ga_G} - 1 
\text{ with our notations.}\]
Therefore, maximizing $c_e$ corresponds to minimizing $\Ga_G$. Consequently, \cite[Conjecture 4.1]{clark2024comparing} is equivalent to Conjecture~\ref{conj:minimal_Ga}. Actually, Conjecture~\ref{conj:minimal_Ga} was originally proposed in \cite{rucker2002kites} in the context of molecular graphs, with the aim of introducing graph invariants that quantify certain structural properties of molecules. The conjecture was popularized in \cite[Conjecture 24]{lele2023unsolved} along with other unsolved problems in spectral graph theory. Note, however, that in those papers the principal eigenvector $\x$ is normalized to satisfy $\| \x \|_2=1$, and hence $\Ga_G$ is simply $\| \x \|_1^2$, and that is the reason why the formulation of the conjecture seems different at first glance.

Several other measures of irregularity have been studied in connection to the largest eigenvalue and the principal eigenvector. Most notably, the \emph{principal ratio} $\ga(G)$, defined as the ratio of the largest and smallest entries of $\x$, has been thoroughly examined. Note that $\ga(G)=1$ if and only if $G$ is regular. Cioab\u{a} and Gregory conjectured \cite{cioaba2007irregular} that the most irregular graph (i.e., the one having the largest principal ratio) among connected $n$-vertex graphs is $K_m+P_{n-m}$ for some $m$. This conjecture was confirmed by Tait and Tobin for sufficiently large $n$ \cite{tait2018maxpr}. In \cite{zhang2024stability} Zhang studied the stability of the principal ratio around regular graphs. In \cite{clark2024comparing} Clark argues that the principal ratio $\ga(G)$ is more sensitive/less robust than $c_e$ (and hence $\Ga_G$). He also proves an upper bound for $c_e$ in terms of $\ga(G)$ \cite[Lemma 2.2]{clark2024comparing}, which may be expressed as the following lower bound on $\Ga_G$:
\[ \Ga_G \geq \frac{|V(G)|}{\ga^2(G)} .\]

Another measure of irregularity is the difference of $\la_G$ and the average degree of $G$, proposed by Bell \cite{bell1992note}. Here the extremal graph is different: it is conjectured to be a complete graph $K_m$ with $n-m$ leaves attached to one of the vertices \cite{aouchiche2008variable}.

Many irregularity measures based on degrees have been considered as well. For instance, Albertson proposed \cite{albertson1997irregularity} the quantity 
\[ \sum_{uv \in E(G)} | \deg(u) - \deg(v) | .\]
The corresponding extremal graphs were characterized in \cite{hansen2005variable}. Nikiforov proved several inequalities between degree-based irregularity measures in \cite{nikiforov2006eigenvalues}.

\subsection{Our approach}

The main difficulty in bounding $\Gamma_G$ and proving these conjectures is that $\Gamma_G$ is far from being a monotone graph parameter: we may have $\Gamma_H \gg \Gamma_G$ for a subgraph $H$ of $G$. For instance, $\Gamma_{K_4+P_k}$ is bounded, while $\Gamma_{P_k}$ grows linearly in $k$.

The situation, however, becomes manageable if we consider ``heavy'' subgraphs $H$ that have few vertices but a fairly large weight with respect to the Perron vector $\x$ of $G$. This way we will be able to give lower bounds on $\Ga_G$ in terms of $H$ and the top eigenvalue of $G$.

More specifically, we start from the vertex $o$ with the largest $x_o$ and explore $G$ around this master vertex by a breadth-first search from $o$. We refer to this small part of the graph around $o$ as the \emph{kernel}. In order to get sufficiently strong bounds on $\Ga_G$, we will need to consider kernels containing $6$ vertices (or $10$ vertices in the case of trees). In both cases, there are over a hundred possibilities for the kernel. First we aim to show that only one of these possibilities is viable for the minimizing graphs/trees. Given any other kernel, we need to show that $\Ga_G$ exceeds the limiting ratio for every extension $G$ of that kernel. 

Some of the kernels could be handled using basic estimates but one quickly realizes that there are many kernels for which very precise bounds will be necessary and some amount of computer assistance seems inevitable in these problems. Our approach was to develop a bounding method that can treat all the kernels on a common theoretical basis, and use a computer to verify that our estimates are sufficiently accurate to prove the conjecture. It has been a challenge to find a method that works for (essentially) all extensions of a kernel and gives strong enough bounds. We will outline this method in Section \ref{sec:key_method}. 

Once we know what the kernel must be, we need to show that what lies outside the kernel (the ``tail'') must be a single path. The difficulty here is that one can modify the tail slightly, with very little change in $\Ga_G$. Delicate estimates are required to prove that any modification will cause $\Ga_G$ to exceed the limiting ratio $\Gamma_{K_4 + P_\infty}$ (or $\Gamma_{S_5 + P_\infty}$ in the case of trees).

\subsection{Notations.} \label{sec:notations}
By $K_r$, $P_r$, and $S_r$ we denote the complete graph, the path, and the star on $r$ vertices, respectively.\footnote{Note that the central vertex of $S_r$ has degree $r-1$.} For a complete graph with a pendant path we write $K_r + P_k$. Similarly, $S_r + P_k$ is a star graph with a pendant path. (Note that both have $r+k$ vertices and maximum degree $r$.)

Throughout the paper $\la_G$ denotes the largest eigenvalue of the adjacency matrix $A_G$ of a connected graph $G$, and let $\x_G=\big( x_v \big)_{v \in V(G)}$ be the corresponding eigenvector (often called the \emph{principal eigenvector} or simply the \emph{Perron vector}). The entries $x_v$ of $\x_G$ are known to have the same sign, so we will always assume that they are all positive. We will often omit the subscript $G$ and simply write $\la$ and $\x$. For a subgraph $H$ of $G$, let $\x\big|_H$ be the restriction of $\x$ to the vertices of $H$. (Note that $\x\big|_H$ and $\x_H$ are two different vectors.) We will often refer to $x_v$ as the \emph{weight} of $v$, and define the \emph{master (vertex)} of $G$ as the vertex with the largest weight.

The degree of a vertex $v$ is denoted by $\deg(v)$, while $\dist(u,v)$ is the distance between $u$ and $v$. As usual, $N_G(v)$ and $N_G[v]$ denote the \emph{open and closed neighborhood} of a vertex, respectively.\footnote{We use the notation $N_G(v)$ both for the set of neighbors and for the induced subgraph on that set. The context always makes it clear which of the two is meant. The same applies to $N_G[v]$ and $N_G[H]$.} Also, for a subgraph $H$, $N_G[H]$ is the closed neighborhood of $H$, that is, the induced subgraph containing the vertices of $H$ and their neighbors.

Furthermore, recall \eqref{eq:Ga_y}, where we defined $\Ga(\y)$ as $\big( \sum |y_i| \big)^2 \big/ \big( \sum y_i^2 \big)$ for a finite (or countably infinite) vector $\y=(y_i)$. With this notation, we clearly have $\Ga_G = \Ga\big( \x_G \big)$. 

Finally, we will use the following shorthand notations for the limiting ratios:
\[ \bestar:= \Gamma_{K_4 + P_\infty} = \frac{5+3\sqrt{3}}{2} 
\quad \text{and} \quad
\betr:= \Gamma_{S_5 + P_\infty} = 4+2\sqrt{3}.\]

\subsection{Our key bounding method} \label{sec:key_method}
Suppose that $H$ is a connected induced subgraph of $G$. We will think of $H$ as the known part of $G$ and treat the rest as ``unknown variables''. In particular, we will always treat the top eigenvalue $\la=\la_G$ of $G$ as a variable. We will refer to vertices of $H$ as \emph{inside vertices}. For an \emph{outside vertex} $w \in V(G) \sm V(H)$, we use the shorthand notation 
\[ \partial w := N_G(w) \cap V(H) \]
for the set of inside neighbors of $w$. Furthermore, let $\Ht := N_G[H]$ denote the closed neighborhood of $H$. For the sake of simplicity, let $\xt := \x\big|_{\Ht}$ denote the restriction of the Perron vector $\x$ to $V(\Ht)$. Our goal is to bound $\Ga(\xt)$. (Later Lemma \ref{lem:intro_2la+3} will show how this can be turned into a bound on $\Ga_G$.) We start by introducing the following \emph{(resolvent) matrix}:
\begin{equation} \label{eq:B_intro}
B :=(\la I - A_H)^{-1} .    
\end{equation}
This is a symmetric matrix of size $|V(H)| \times |V(H)|$ whose entries are (rational) functions of $\la$. So if we think of $H$ as a fixed (small) graph, then for any $u,v \in V(H)$, $B_{u,v}$ is a fixed rational function\footnote{A rational function is the ratio of two polynomials.}. Moreover, since $G$ contains $H$, we have $\la=\la_G>\la_H$, and it can be shown that each $B_{u,v}$ is positive in that regime. The resolvent matrix $B$ has further nice properties, see Proposition \ref{prop:B_properties}.

For our purposes the key fact is that $B$ can be used to express the weights $x_v$, $v \in V(H)$ inside $H$ in terms of the weights outside $H$. Indeed, by writing up the eigenvalue equations for each inside vertex, we get a linear system for $x_u$, $u \in V(H)$. Solving this system gives  
\[ x_u 
= \sum_{w \in V(\Ht) \sm V(H)} x_w \sum_{v \in \partial w} B_{u,v} \quad \text{for every } u \in V(H) .\]
This formula motivates the introduction of the following extension of the resolvent matrix: for $u \in V(H)$ and $\emptyset \neq V \subseteq V(H)$ let
\begin{equation} \label{eq:Bt_intro}
\Bt_{u,V} := \sum_{v \in V} B_{u,v} .
\end{equation}
With this notation we have $x_u = \sum_w \Bt_{u,\partial w} \, x_w$. Setting $k=|V(H)|$, one may think of $\Bt$ as a $k \times (2^k-1)$ matrix. Its columns consist of all partial sums of the columns of $B$. In particular, the column corresponding to $V=\{v\}$ is a column of $B$, so $\Bt$ extends $B$ in this sense. 

It follows that both $\| \xt \|_1^2$ and $\| \xt \|_2^2$ can be written as quadratic forms of the variables $x_w$, and hence 
\[ \Ga(\xt) = \frac{\| \xt \|_1^2}{\| \xt \|_2^2}  
= \frac{\sum_{w,w'} a_{w,w'} x_w x_{w'}}{\sum_{w,w'} b_{w,w'} x_w x_{w'}} \]
for some coefficients $a_{w,w'}$, $b_{w,w'}$ (that are rational functions of $\la$). Recall that the weights $x_w$ are all positive. Moreover, the coefficients of the quadratic forms are also positive provided that $\la>\la_H$. Therefore
\[ \Ga(\xt) \geq \min_{w,w'} \frac{a_{w,w'}}{b_{w,w'}} .\]

It turns out, however, that this approach gives good bounds only near $\la_H$. This limitation can be overcome by imposing additional requirements on $G$ (other than it contains $H$ as an induced subgraph). This could actually be done in various ways but a basic assumption, which will be sufficient for our purposes, is the following. Suppose that, for some fixed vertex $o$ of $H$, we have $x_w \leq x_o$ for every outside vertex $w$. Under this assumption, the denominator increases if we replace each $x_w^2$ with $x_w x_o$. This gives a different quadratic form in the denominator which leads to good lower bounds for any $\la > \la_H$. The drawback is that it only applies to graphs $G \supset H$ that satisfy our additional assumption. After working out the details, we will obtain the following result.
\begin{Thm} \label{thm:intro_main_bound}
Let $H$ be a connected induced subgraph of $G$, and consider the corresponding resolvent matrix $B$ and its extended version $\Bt$ as defined in \eqref{eq:B_intro} and \eqref{eq:Bt_intro}. We denote the column sums of $\Bt$ by  
\[ s_{V} := \sum_{u \in V(H)} \Bt_{u,V} \quad 
\text{for a nonempty } V \subseteq V(H) ,\]
and the inner products of the column vectors of $\Bt$ by 
\[ c_{U,V} := \sum_{u \in V(H)} \Bt_{u,U} \Bt_{u,V} \quad 
\text{for nonempty } U,V \subseteq V(H) .\]
Furthermore, let $\Ht=N_G[H]$ and $\xt = \x\big|_{\Ht}$, where $\x$ is the Perron vector of $G$. If $\partial w = N_G(w) \cap V(H)$ denotes the set of inside neighbors of a vertex $w \in W := V(\Ht) \sm V(H)$, then we have 
\[ \| \xt \|_1^2 = \sum_{w,w' \in W} \big( s_{\partial w} + 1\big) \big( s_{\partial w'} + 1\big) x_w x_{w'} \]
Moreover, if the master vertex\footnote{Recall that in Section~\ref{sec:notations} the master vertex was defined as the vertex $o$ with the largest $x_o$.} $o$ of $G$ lies in $H$, then 
\[ \| \xt \|_2^2 \leq \sum_{w,w' \in W} \left( c_{\partial w,\partial w'} + \frac{1}{2} \Bt_{o,\partial w} + \frac{1}{2} \Bt_{o,\partial w'}\right) x_w x_{w'} .\]
Let $\Uc^G_H := \big\{ \partial w \, : \, w \in W \big\}$. It follows that 
\begin{equation} \label{eq:the_bound}
\Ga(\xt) = \frac{\| \xt \|_1^2}{\| \xt \|_2^2}  
\geq \min_{U,V \in \Uc^G_H} \; \frac{\big( s_U + 1\big) \big( s_V + 1\big)}{c_{U,V} + \frac{1}{2} \Bt_{o,U} + \frac{1}{2} \Bt_{o,V}} .
\end{equation}
Note that this bound depends on the eigenvalue $\la=\la_G$: indeed, the right-hand side is the minimum of various rational functions of $\la$. 
\end{Thm}

One can apply this theorem as follows. Given a specific $H$ with a vertex $o \in V(H)$ and some target ratio $\beta > 0$, we want to conclude that $\Ga(\xt) \geq \beta$. To achieve this, we need to check that the expression on the right-hand side of \eqref{eq:the_bound} is at least $\beta$ for any pair $\emptyset \neq U,V \subseteq V(H)$ and for any possible $\la$. Note, however, that a set $U$ is included in $\Uc^G_H$ only if $G$ has a vertex $w$ with $\partial w = U$, which implies that $\la_G \geq \la_U$, where $\la_U$ is the top eigenvalue of the graph $G[V(H) \cup \{w\} ]$. Therefore, for a pair $U,V$, we need to perform the check only for $\la \geq \max(\la_U, \la_V)$. In fact, as we will explain in Section \ref{sec:key_bounds} (see Remark \ref{rem:poly}), the computer will simply need to verify the positivity of the coefficients of some low-degree polynomials, which can be done efficiently. 

By default, we would need to do this for all nonempty subsets $U,V$ of $V(H)$. However, in our applications we will often have extra information regarding how outside vertices may be attached to $H$, narrowing down what $\partial w$ might be. For instance, $o$ is (almost) always assumed to have no outside neighbors, meaning $\partial w$ never contains $o$. In the case of trees the situation is even better: we always know that $|\partial w|=1$.

For a concrete example, see the Appendix, where we carried out the computations for the graph $H = \KKKPPP \, = K_3+P_3$.

\subsection{Outline of the proof of Theorem \ref{thm:only_K4Pk}} \label{sec:proof_outline}

The proofs of both conjectures are analogous. Here we provide a synopsis of the first one.

Let $o$ denote the master vertex of $G$ meaning that $x_o \geq x_v$ for each $v \in V(G)$. The following observation will be used several times. We will prove this in Section~\ref{sec:heavy_subgraphs_degree} by combining some basic estimates, see Lemma \ref{subgraph-lower-bound}(b).
\begin{Lemma} \label{lem:intro_2la+3}
For any induced subgraph $H$ of $G$ containing the master $o$ and its neighbors we have 
\[ \Ga_G = \Ga(\x) 
\geq \min\bigg( \Ga\big(\x\big|_H \big) , 2\la_G + 3 \bigg) .\]
\end{Lemma}
The proof of the conjecture consists of two parts. 

\subsubsection*{First part: identifying the kernel, i.e., determining what $G$ looks like around its master vertex.}
Our goal is to describe $G$ around its master vertex under the assumption that $\Ga_G<\bestar \approx 5.098$. 
In fact, we will choose a slightly larger target ratio $\beta=5.25$ and work under the weaker assumption $\Ga_G<5.25$. This way, by adding only a few extra complications to the proof, we get the following stronger result.
\begin{Thm} \label{thm:525}
Suppose that $G$ is a connected graph with $|V(G)| \geq 8$. If $\Gamma_G < 5.25$, then the master $o$ of $G$ is contained in a $4$-clique $C$ such that there is a single edge $e=(o,u)$ between $C$ and $V(G) \setminus C$. In other words, $G$ is the disjoint union of $K_4$ and a connected graph $G'$ joined by a single additional edge $(o,u)$, see Figure \ref{fig:K4_plus_Gpr}. In particular, the closed neighborhood $N_G[o]$ of the master must be isomorphic to $\KKKKP \, \cong K_4 + P_1$.

The same is true for connected graphs of order $7$ with the single exception $G=\DsPPP$, for which $\Gamma_G \approx 5.180545$.
\end{Thm}
\begin{figure}[ht]
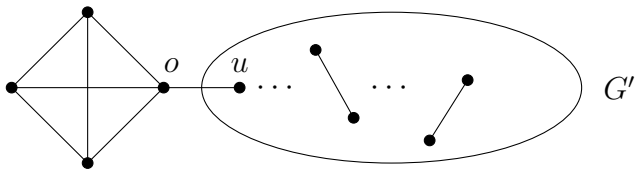

\centering
\KKKKoGG
\caption{If $\Gamma_G<5.25$, then $G$ contains a $K_4$ that is joined to the remaining part $G'$ by a single edge incident to the master vertex $o$.}
\label{fig:K4_plus_Gpr}
\end{figure}
\begin{proof}[Proof outline]
Suppose that $G$ is a connected graph with $|V(G)| \geq 7$ and $\Ga_G < 5.25$.

First of all, we may assume that $\la_G > 2$, and hence $\deg(o) \geq 3$. The reason is that connected graphs with $\la_G \leq 2$ have been characterized and one can easily compute $\Ga_G$ for these graphs and conclude that $\Ga_G>5.25$ for each of them. See Section \ref{sec:lambda_at_most_2} for details.

Moreover, we may also assume that $\deg(o) \leq 5$. This is due to Theorem \ref{degree-6} which provides a simple lower bound on $\Ga\big( \x \big|_{N_G[o] } \big)$. The bound $\be_d$ depends on the degree $d := \deg(o)$ of the master and it grows as $c \sqrt{d}$. The relevant fact is that $\Ga\big( \x \big|_{N_G[o]} \big) \geq \be_d > 5.3$ for $d \geq 6$. Combining this with Lemma \ref{lem:intro_2la+3}, it follows that $\Ga_G > 5.3$ whenever $\deg(o) \geq 6$.

It remains to consider the case when $3 \leq \deg(o) \leq 5$. Then we can choose a $6$-vertex induced subgraph $H \subset G$ containing $o$ and all its neighbors. We call the rooted graph $(H,o)$ the \emph{kernel}. It turns out that there are $155$ possibilities for the kernel. As we mentioned earlier, some of these kernels are difficult to rule out and we wanted to treat all kernels in a unified framework. Our general bounding technique was outlined in Section \ref{sec:key_method} and the bound itself was stated in Theorem \ref{thm:intro_main_bound}.\footnote{The theorem bounds $\Ga(\xt)$ but Lemma \ref{lem:intro_2la+3} can be used again to get a bound on $\Ga_G$, see Remark~\ref{rem:Gaxt_to_GaG} for details.} Note that $o$ cannot have any neighbors outside $H$, and hence we can disregard any $U,V$ containing $o$ when taking the minimum in Theorem \ref{thm:intro_main_bound}. For some kernels we can exclude further sets as well. (The details will be given in Section \ref{sec:kernel}.) Our computer check essentially excludes all but one kernel and we get that $\Ga_G<5.25$ is possible only if $(H,o)$ is isomorphic to $\KKKKPPr \, = \big( K_4 +_o P_2, o \big)$. From here the theorem follows easily.
\end{proof}

\subsubsection*{Second part: showing that the ``tail'' must be a single path.}
Once we know that the kernel of $G$ must be \KKKKPPr, we can start to look further away. Assume that the ``tail'' (i.e., the graph $G'$ in Theorem~\ref{thm:525}) is not a single path: suppose that the tail branches first at the $\ell$-th vertex, that is, $G \supset K_4 + P_\ell$ and the last vertex of $P_\ell$ has at least two further neighbors. We need to show that in this case $\Ga_G$ exceeds the limiting ratio $\bestar = \Gamma_{K_4 + P_\infty}$. In Section~\ref{sec:tail} we will prove a general result (Theorem~\ref{thm:Gamma-upper}) providing a set of conditions guaranteeing that $\Ga_G>\Ga_{H+_v P_\infty}$ for graphs $G$ extending $H$ with a ``branching tail''. Such results are true only with a small margin so we need to be very careful with our estimations. Once we checked the conditions of this general result in the special case $H=K_4$ (see Theorem~\ref{K4_tail}), we conclude that $\Ga_G < \bestar$ is only possible if $G \cong K_4 + P_{n-4}$. Of course, we also need to show that $\Ga_G$ is indeed below the limiting ratio for this particular graph which will be done in Section~\ref{sec:pendant_paths}, see Theorem~\ref{thm:indeed_below}. This part of the proof will also be done in a general setting: Lemma~\ref{lem:Gamma-lower} proves that $\Ga_{H+_v P_k} < \Ga_{H+_v P_\infty}$ holds under mild conditions.

\subsection{Program code}
We have uploaded our program code (Python/SageMath) to the following GitHub repository: \; 
\url{https://github.com/harangi/perron}.\\
Besides the code, there is a Jupyter notebook \cite{JUP} in the repository. It shows how to use the code to perform the necessary computer checks for our proofs. The notebook also contains supplementary material to help the reader better understand the nature of these conjectures and our methods.

\subsection*{Organization of the paper.} 
In the introduction we tried to give a good overview of our main results and methods. When we give the details in the rest of the paper, we try to do it with a somewhat broader scope: at many places we prove more than what is necessary for settling the conjectures. We opted for this approach because we hope that the tools we developed could be useful beyond these conjectures. However, we always point out the parts that may be skipped by a reader solely interested in the conjectures. 

The paper is organized as follows.
\begin{itemize}
\item Section~\ref{sec:basic_results}: basic facts regarding $\la_G$, $\Ga_G$ and the resolvent matrix.
\item Section~\ref{sec:lambda_at_most_2}:  $\Gamma_G$ for graphs with $\la_G \leq 2$.
\item Section~\ref{sec:heavy_subgraphs_degree}: heavy subgraphs are introduced and $\Ga_G$ is bounded solely in terms of the degree of the master vertex.
\item Section~\ref{sec:key_bounds}: the key bounding method.
\item Section~\ref{sec:kernel}: the computer-assisted proof (based on the tools of Section~\ref{sec:key_bounds}) that the minimizing graphs/trees have the conjectured kernel.
\item Section~\ref{sec:pendant_paths}: results about $\la_G$ and $\Ga_G$ for graphs with (finite or infinite) pendant paths.
\item Section~\ref{sec:tail}: estimating $\Ga_G$ when the ``tail branches''.
\item Section~\ref{sec:the_end}: putting the puzzle pieces together to prove the main results (Theorems~\ref{thm:only_K4Pk} and \ref{thm:only_S5Pk}).
\end{itemize}

\section{Basic results and estimates} \label{sec:basic_results}

We use the notations introduced in Section \ref{sec:notations}. In particular, $o$ denotes the master vertex of $G$ (that has the largest entry $x_o$ of the Perron vector $\x$). 

The following claim shows that small $\Gamma_G$ implies that the master vertex has a relatively large weight, and the largest eigenvalue is small.

\begin{Lemma} \label{Gamma-lambda bound}
We have\\
(a) $x_o\geq \frac{1}{\sqrt{\Gamma_G}}\|\x\|_2 =\frac{1}{\Gamma_G}\|\x\|_1$;\\
(b) $\Gamma_G-1\geq \la_G$.
\end{Lemma}

\begin{proof}
(a) The simple bound  
$\sum_{v\in V}x_v^2\leq x_o\sum_{v\in V}x_v$ gives 
\[ x_o\geq \frac{\sum_{v\in V}x_v^2}{\sum_{v\in V}x_v}=\frac{\|\x\|_2}{\|\x\|_1}\|\x\|_2=\frac{1}{\sqrt{\Gamma_G}}\|\x\|_2. \]
(b) Since 
\[ \|\x\|_1=\sum_{v\in V}x_v\geq x_o+\sum_{u\in N_G(o)}x_u=(\la_G+1)x_o ,\]
we get that 
\[ \frac{\|\x\|_1}{\la_G+1}\geq x_o \geq \frac{1}{\sqrt{\Gamma_G}}\|\x\|_2 ,\]
and hence
\[ \Gamma_G=\frac{\|\x\|_1}{\|\x\|_2}\sqrt{\Gamma_G}\geq \la_G+1. \]
\end{proof}

\begin{Rem}
It follows from (b) that if we are to search for graphs with $\Gamma_G<\beta$, then it suffices to do so among graphs with $\la_G<\beta-1$.
\end{Rem}

\begin{Lemma}
Let $\x_1$ and $\x_2$ be two positive vectors in $\mathbb{R}^n$. Then for any $\alpha \in (0,1)$ we have
\[\Gamma(\alpha \x_1+(1-\alpha)\x_2)\geq \min(\Gamma(\x_1),\Gamma(\x_2)).\]
\end{Lemma}

\begin{proof}
Let $\beta=\min(\Gamma(\x_1),\Gamma(\x_2))$.
Then we have
\begin{align*}
\|\alpha \x_1+(1-\alpha)\x_2\|_1^2&=(\alpha \|\x_1\|_1+(1-\alpha)\|\x_2\|_1)^2\\
&\geq (\alpha \sqrt{\beta}
\|\x_1\|_2+(1-\alpha)\sqrt{\beta}\|\x_2\|_2)^2\\
&=\beta (\alpha \|\x_1\|_2+(1-\alpha)\|\x_2\|_2)^2\\
&\geq \beta \|\alpha \x_1+(1-\alpha)\x_2\|_2^2.
\end{align*}
\end{proof}

\begin{Lemma} \label{reverse-AM-QM}
Suppose that $x_1,\dots ,x_k$ satisfy $m\leq x_1,\dots ,x_k\leq M$ and $\sum_{i=1}^k x_i=S$. Then
\[\frac{\left(\sum_{i=1}^k x_i\right)^2}{\sum_{i=1}^k x_i^2}\geq \frac{S^2}{(m+M)S-kmM}.\]
\end{Lemma}

\begin{proof}
We clearly have $\sum_{i=1}^k(M-x_i)(x_i-m)\geq 0$. It follows that 
\[(m+M)\sum_{i=1}^k x_i-kmM \geq \sum_{i=1}^k x_i^2 .\]
Therefore 
\[\frac{\left(\sum_{i=1}^k x_i\right)^2}{\sum_{i=1}^k x_i^2}\geq \frac{S^2}{(m+M)S-kmM}.\]
\end{proof}

\begin{Lemma}[Perturbation lemma] \label{perturbation}
Suppose that the vector $\x=(x_0,x_1,\dots)$ has nonnegative entries. For some $0 < \varepsilon \leq x_0$ let $\x'=(x_0-\varepsilon,x_1,x_2,\dots)$. Then 
\[ \Ga(\x') < \Ga(\x) 
\quad \text{provided that} \quad 
x_0 < \frac{\sum_{i=0}^\infty x_i^2}{\sum_{i=0}^\infty x_i} .\]
\end{Lemma}
\begin{proof}
Let $S:=\sum_{i=0}^\infty x_i$ and $T:=\sum_{i=0}^\infty x_i^2$.
Then 
\[ \Ga(\x) = \frac{S^2}{T} 
\quad \text{and} \quad 
\Ga(\x') = \frac{(S-\eps)^2}{T-2x_0\eps+\eps^2} .\]
So $\Ga(\x) > \Ga(\x')$ is equivalent to 
\[ 2\eps S(T-S x_0)+(S^2-T)\eps^2 > 0.\]
The first term is positive by our assumption $x_0<T/S$, while the second term is always non-negative since $S^2 \geq T$.
\end{proof}

\begin{Lemma} \label{prop:la&deg} 
Let $o$ be the master vertex of a graph $G$ and let $d :=\deg(o)$. Then 
\[\sqrt{d} \leq \la_G \leq d .\]
\end{Lemma}
\begin{proof}
On the one hand, $G$ contains the star graph $S_{d+1}$ as a subgraph, therefore we have $\la_G \geq \la_{S_{d+1}}=\sqrt{d}$.

On the other hand, the eigenvalue equation at the master vertex tells us that 
\[ \la_G x_o = \sum_{v \in N_G(o)} x_v \leq d x_o, 
\; \text{and hence} \; \la_G \leq d .\]
\end{proof}

\subsection{Resolvent matrix} \label{sec:resolvent}

Next we introduce a key object for our approach, a matrix associated to any (connected) graph $H$ and real number $\la > \la_H$, which enables us to study the Perron vector of a graph through a subgraph.
\begin{Def} \label{def:B}
For a given finite graph $H$ and $\la > \la_H$ let
\begin{equation} \label{eq:B_def}
B = B_H(\la) := (\la \cdot I - A_H)^{-1} ,
\end{equation}
where $A_H$ is the adjacency matrix of $H$ and $I$ is the identity matrix. Note that the inverse exists because $\la$ is not an eigenvalue of $A_H$ since it is larger than the top eigenvalue $\la_H$. 
\end{Def}

\begin{Rem}
Note that $-B$ is the so-called \emph{resolvent} of the matrix $A_H$. It is a well-studied object in spectral graph theory; see e.g.~\cite[Chapter 4]{godsil2017algebraic}.
\end{Rem}

\begin{Prop} \label{prop:B_properties}
The matrix $B=B_H(\la)$ has the following properties for $\la>
\la_H$.
\begin{itemize}
\item[(a)] $B$ can be expressed by the following power series formula:
\begin{align} \label{eq:B_power_series}
B = \la^{-1} \left(I - \frac{A_H}{\la} \right)^{-1} 
= \la^{-1} \sum_{k=0}^\infty \left( \frac{A_H}{\la} \right)^k 
= \sum_{k=0}^\infty \frac{A_H^k}{\la^{k+1}}.
\end{align}
\item[(b)] $B$ is a symmetric matrix with positive entries.
\item[(c)] As a function of $\la$, each entry of $B$ is a strictly monotone decreasing function.
\item[(d)] Let $P_H(\la) :=\det(\la I - A_H)$ be the characteristic polynomial of $A_H$. It has integer coefficients and degree $|V(H)|$. We have 
\begin{equation} \label{eq:adj}
P_H(\la) B_H(\la) = \mathrm{adj}(\la I - A_H) .
\end{equation}
Each entry of the adjugate matrix on the right is a polynomial in $\la$ with integer coefficients and degree at most $|V(H)|-1$. Consequently, each entry of $B$ is a rational function of $\la$.
\item[(e)] Asymptotics as $\la \searrow \la_H$: we have $\lim_{\la\to \la_H}(\la-\la_H)B_{H}(\la)=\x_H\x_H^T$, where $\x_H$ is the Perron vector of $H$ normalized in such a way that $\|\x_H\|_2=1$.
\item[(f)] Asymptotics as $\la \to \infty$: we have $\lim_{\la\to \infty} \la B_{H}(\la)=I$.
\end{itemize}
\end{Prop}
\begin{proof}
(a) Note that $A_H$ is a symmetric matrix so $\| A_H \|=\la_H$. Since $\la>\la_H$, $\| A_H\big/ \la \|<1$, and hence the formal power series in \eqref{eq:B_power_series} is indeed convergent and equal to $B$.

Parts (b), (c) and (f) are immediate consequences of the power series formula, while (d) follows from the adjugate formula for the inverse of a nonsingular matrix.

Finally, (e) follows from the spectral decomposition formula. Indeed, let $\v_1,\dots ,\v_n$ be an orthonormal basis consisting of eigenvectors of $A_H$ so that $A_H\v_k=\la_k\v_k$. We may choose $\la_1=\la_H$ and $\v_1=\x_H$. Then, by spectral decomposition, we have the following formula for $B_H$:
\[B_H(\la)=\sum_{k=1}^n\frac{1}{\la-\la_k}\v_k\v_k^T ,\]
and (e) clearly follows.
\end{proof}
Next we explain how one can make use of this matrix to study the principal eigenvector $\x$ of a graph $G$ extending $H$. 
\begin{Lemma} \label{lem:inside_weights}
Let $G$ be a connected graph and $H$ an induced subgraph of $G$. Set $\la$ to be the top eigenvalue $\la_G$ of the large graph and consider the matrix $B=B_H(\la)$.\footnote{Note that $\la = \la_G > \la_H$ as required.} 

For each $v \in V(H)$ we define $y_v$ to be the sum of the weights of the neighbors of $v$ outside $H$:
\[ y_v :=\sum_{w \in N_G(v) \setminus V(H)} x_w,\]
and we write $\y$ for the vector $\big( y_v \big)_{v \in V(H)}$. Then 
\[ \x\big|_{H} = B \y ,\]
where 
$\displaystyle \x\big|_{H} = \big( x_v \big)_{v \in V(H)} $
denotes the restriction of $G$'s principal eigenvector $\x$ to $V(H)$.

Therefore, for a fixed $\la$, we can intuitively think of $B_{u,v}$ as the change in $x_u$ caused by an ``additional unit outside weight'' at $v$. 
\end{Lemma}
\begin{proof}
The eigenvalue equation at a given $v \in V(H)$ gives 
\begin{equation} \label{eq:ev_GH}
\la x_v = \sum_{u \in N_H(v)} x_u 
+ \sum_{w \in N_G(v) \setminus V(H)} x_w
= \big( A_H \x\big|_{H} \big)_v + y_v .
\end{equation}
It follows that 
\[ \la \x\big|_{H} = A_H \x\big|_{H} + \y 
\text{, and thus } 
\x\big|_{H} = (\la I - A_H)^{-1} \y = B \y .\]
\end{proof}

\section{Graphs with largest eigenvalue at most 2}
\label{sec:lambda_at_most_2}

Graphs with largest eigenvalue at most $2$ are completely characterized, so we only need to compute $\Gamma_G$ for the graphs in the characterization.

\begin{Thm}[Smith \cite{smith1970some}; Lemmens and Seidel \cite{lemmens1973equiangular}] Connected graphs with $\la_G \leq 2$ consist of four families and six sporadic graphs as follows.
\begin{enumerate}[(i)]
\item the path graphs $P_n$
\item the graphs $D_n=S_3+P_{n-3} = \, \leftfork \cdots \Pthree$
\item the cycle graphs $C_n$
\item the ``bi-fork'' graphs $\hat{D}_n = \, \leftfork \cdots \rightfork$ including $S_5 = \, \SSSSS$
\item the following six sporadic graphs:
\begin{center}
\begin{tabular}{lll}
$E_6 = \, \Esix$  & $E_7 = \, \Eseven$ & $E_8 = \, \Eeight$ \\
$\hat{E}_6 = \, \Ehatsix$  & $\hat{E}_7 = \, \Ehatseven$ & $\hat{E}_8 = \, \Ehateight$
\end{tabular}
\end{center}
\end{enumerate}
Moreover, for the graphs $C_n,\hat{D}_n, \hat{E}_6, \hat{E}_7, \hat{E}_8$ the largest eigenvalue is exactly $2$, and any connected $G$ with $\la_G \geq 2$ is known to contain at least one of them as a subgraph.
\end{Thm}

Next we compute $\Gamma_G$ for each of the graphs above.

\begin{Thm} \label{lambda2gamma}
We have 
\begin{enumerate}[(i)]
\item $\displaystyle \Gamma_{P_n}=\frac{2}{n+1}\left(\frac{\sin\left(\frac{\pi}{n+1}\right)}{1-\cos\left(\frac{\pi}{n+1}\right)}\right)^2$;
\item $\displaystyle \Gamma_{D_n}=\frac{1}{2(n-1)}\left(1+\frac{\sin\left(\frac{\pi}{2(n-1)}\right)}{1-\cos\left(\frac{\pi}{2(n-1)}\right)}\right)^2$;
\item $\displaystyle \Gamma_{C_n}=n$;
\item $\displaystyle \Gamma_{\hat{D}_n}=\frac{(n-2)^2}{n-3}$;
\item $\displaystyle \Gamma_{E_6}\approx 5.293,\ \Gamma_{E_7}\approx 6.043,\ \Gamma_{E_8}\approx 6.781,\ \Gamma_{\hat{E}_6}=6,\  \Gamma_{\hat{E}_7}=6.75,\ \Gamma_{\hat{E}_8}=7.5$.
\end{enumerate}
It is easy to see that $\Gamma_{P_n}$, $\Gamma_{D_n}$, $\Gamma_{C_n}$, $\Gamma_{\hat{D}_n}$ are all monotone increasing in $n$ and has linear growth (with constants $8/\pi^2$, $8/\pi^2$, $1$, $1$, respectively). For $n \geq 7$, $\Gamma_{D_n}$ is the smallest of the four with 
\[  \Gamma_{D_7} \approx 6.157; \;
\Gamma_{D_8} \approx 6.966; \;
\Gamma_{D_9} \approx 7.775; \;
\Gamma_{D_{10}} \approx 8.584; \;
\Gamma_{D_{11}} \approx 9.393 .\]
\end{Thm}

\begin{proof}
Parts (iii),(iv) and (v) are straightforward to get.

\noindent\textbf{(i)} It is well known (and easy to check) that the vector $\v=(v_1,\ldots,v_n)$ with $v_k=\sin\left(\frac{k\pi}{n+1}\right)$ is an eigenvector of $P_n$ belonging to the largest eigenvalue $2\cos\left(\frac{\pi}{n+1}\right)$. Then the claim follows from the identities
\begin{equation*}
\sum_{k=1}^n\sin\left(\frac{k\pi}{n+1}\right) 
= \frac{\sin\left(\frac{\pi}{n+1}\right)}{1-\cos\left(\frac{\pi}{n+1}\right)} 
\quad \text{and} \quad
\sum_{k=1}^n\sin^2\left(\frac{k\pi}{n+1}\right) 
= \frac{n+1}{2} ,
\end{equation*}
which can be proved using $\sin(x)=\frac{e^{ix}-e^{-ix}}{2i}$ and summing the resulting geometric series:
\begin{multline*}
\sum_{k=1}^n\sin\left(\frac{k\pi}{n+1}\right)=\sum_{k=0}^n\sin\left(\frac{k\pi}{n+1}\right)
=\sum_{k=0}^n\frac{1}{2i}\left(e^{\frac{k\pi i}{n+1}}-e^{\frac{-k\pi i}{n+1}}\right)\\
=\frac{1}{2i}\left(\frac{1-(-1)}{1-e^{\frac{\pi i}{n+1}}}-\frac{1-(-1)}{1-e^{-\frac{\pi i}{n+1}}}\right)
=\frac{2}{2i} \, \frac{(1-e^{\frac{-\pi i}{n+1}})-(1-e^{\frac{\pi i}{n+1}})}{2-e^{\frac{\pi i}{n+1}}-e^{-\frac{\pi i}{n+1}}}
=\frac{2\sin\left(\frac{\pi}{n+1}\right)}{2-2\cos\left(\frac{\pi}{n+1}\right)} ;
\end{multline*}
\begin{multline*}
\sum_{k=1}^n\sin^2\left(\frac{k\pi}{n+1}\right)
=\sum_{k=0}^n\sin^2\left(\frac{k\pi}{n+1}\right)
=\sum_{k=0}^n-\frac{1}{4}\left(e^{\frac{k\pi i}{n+1}}-e^{\frac{-k\pi i}{n+1}}\right)^2\\
=-\frac{1}{4}\sum_{k=0}^n\left(e^{\frac{2k\pi i}{n+1}}-2+e^{\frac{-2k\pi i}{n+1}}\right)
=-\frac{1}{4}(-2)(n+1)=\frac{n+1}{2}.
\end{multline*}

\noindent\textbf{(ii)} It is again well known (and again easy to check) that the vector $\v=(v_1,\ldots,v_n)$ with $v_k=\sin\left(\frac{k\pi}{2(n-1)}\right)$ for $1\leq k\leq n-2$ and $v_{n-1}=v_n=\frac{1}{2}$ is an eigenvector of $D_n$ belonging to the largest eigenvalue $2\cos\left(\frac{\pi}{2(n-1)}\right)$.
The claim then follows from the following identities, which can be proved similarly to (i): 
\begin{align*}
\sum_{k=1}^{n-2}\sin\left(\frac{k\pi}{2(n-1)}\right) 
&= \frac{\cos\left(\frac{\pi}{2(n-1)}\right)+\sin\left(\frac{\pi}{2(n-1)}\right)-1}{2\left(1-\sin\left(\frac{\pi}{2(n-1)}\right)\right)}; \\
\sum_{k=1}^{n-2}\sin^2\left(\frac{k\pi}{2(n-1)}\right) 
&= \frac{n}{2}-1.
\end{align*}
\end{proof}

\section{Heavy subgraphs and the degree of the master vertex}
\label{sec:heavy_subgraphs_degree}

In this section we introduce the concept of heavy subgraphs and make some observations about the degree of the master vertex.

\begin{Def}
Let $G$ be a graph with Perron vector $\x$ and $H$ an induced subgraph. We call $H$ $\gamma$-\emph{heavy} if it holds for all $v\notin V(H)$ that
\[ x_v\leq \gamma \sum_{u\in V(H)}x_u. \]
Furthermore, if the master $o$ lies in $H$, then we call $H$ $\gat$-\emph{heavy with respect to the master} if it holds for all $v\notin V(H)$ that
\begin{equation} \label{eq:heavy_wrt_master}
x_v \leq \gat x_o .
\end{equation}
For the conjectures in this paper it will be sufficient to use this second notion with $\gat=1$, in which case \eqref{eq:heavy_wrt_master} automatically holds. If stronger bounds are needed, one can often use some $\gat$ less than $1$; see the approach suggested in Section~\ref{sec:better_gamma}.
\end{Def}
We start with the following simple observation.
\begin{Prop} \label{prop:heavy_obs}
If $H \supseteq N_G[o]$, that is, if $H$ contains the master $o$ and its neighbors, then $H$ is $\frac{1}{\la_G+1}$-heavy.
\end{Prop}
\begin{proof}
Using the eigenvalue equation at $o$ we get 
\[ \sum_{u\in V(H)} x_u \geq x_o + \sum_{u \in N_G(o)} x_u 
= (\la_G+1) x_o .\]
Since $o$ has the largest weight in $G$, it follows that for all $v \in V(G)$ 
\[ x_v \leq x_o \leq  \frac{1}{\la_G+1}  \sum_{u\in V(H)} x_u ,\]
and the proof is complete.
\end{proof}
\begin{Lemma} \label{subgraph-lower-bound} 
(a) Let $\beta>1$. Suppose that for some induced subgraph $H$ of $G$ we have
$\Gamma(\x\big|_H)\geq \beta$ and $H$ is 
 $\frac{2}{\beta-1}$-heavy, then $\Gamma_G\geq \beta$.\\
(b) For any $H$ containing $N_G[o]$ we have $\Gamma_G\geq \min(\Gamma(\x\big|_H),2\la_G+3).$\\
(c) For any $H$ containing $N_G[o]$ and a vertex $v$ of distance $2$ from $o$ we have
\[\Gamma_G\geq \min\left(\Gamma(\x\big|_H),2\la_G+\frac{2}{\la_G^2-1}+3\right).\]
\end{Lemma}

\begin{proof}
To see (a), we set 
\[ S = \sum_{u \in V(H)} x_u ; \,
T = \sum_{u \in V(H)} x_u^2 ; \quad
\St = \sum_{v \notin V(H)} x_v ; \,
\Tt = \sum_{v \notin V(H)} x_v^2 .\]
Since $H$ is $\frac{2}{\beta-1}$-heavy, for every $v \notin V(H)$ we have 
\[ x_v \leq \frac{2}{\beta-1} S 
\text{, implying that }
x_v^2 \leq \frac{2}{\beta-1} S x_v .\]
Summing these inequalities for all $v \notin V(H)$ yields 
\[ \Tt \leq \frac{2}{\beta-1} S \St .\]
Using this and the trivial bound $\St^2 \geq \Tt$ we get 
\begin{align*}
(S+\St)^2-\be(T+\Tt) &= (S^2-\be T) + 2S\St + \St^2 - \be \Tt \\
&\geq (S^2-\be T) +  (\be-1) \Tt + \Tt - \be \Tt 
= S^2-\be T .
\end{align*}
The right-hand side is nonnegative since $\Gamma(\x\big|_H)\geq \beta$. Therefore the left-hand side is nonnegative as well, which, in turn, is equivalent to $\Gamma_G\geq \beta$.

We obtain (b) as an immediate consequence of (a) and Proposition \ref{prop:heavy_obs}.

For (c) we note that $x_v\geq \frac{1}{\la_G^2-1}x_o$ for any vertex $v$ with $\dist(o,v)=2$. Indeed, if $w$ is a common neighbor of $o$ and $v$, then the eigenvalue equations imply that $\la_G x_w\geq x_v+x_o$ and $\la_G x_v\geq x_w$. Then 
$\la_G^2 x_v\geq \la_G x_w\geq x_v+x_o$, thus $x_v\geq \frac{1}{\la_G^2-1}x_o$. Hence 
\[\sum_{u\in V(H)}x_u\geq \left(1+\la_G+\frac{1}{\la_G^2-1}\right)x_o.\]
This shows that $H$ is $\frac{2}{\beta-1}$-heavy for $\beta=2\la_G+\frac{2}{\la_G^2-1}+3$, and hence it is $\frac{2}{\beta'-1}$-heavy for $\beta'=\min\big(\Gamma(\x\big|_H),\beta \big)$ as well, and the proof is complete by (a). 
\end{proof}

See Remark~\ref{rem:Gaxt_to_GaG} for how this lemma will be used in the proofs of the conjectures.

\subsection{Degree of the master vertex.}
Our next goal is to give a lower bound on $\Gamma_G$ in terms of the degree of the master vertex. Note that combining Lemma~\ref{Gamma-lambda bound}(b) and Lemma~\ref{prop:la&deg} already gives such a bound:
\[\deg(o)\leq \la_G^2\leq (\Gamma_G-1)^2 .\]
We need a better bound, however, because we want to conclude that if $\Gamma_G\leq \bestar$, then $\deg(o)\leq 5$. (See the beginning of Section~\ref{subsec:comp_one} for how this will be used in the proof of the first conjecture.)

We will prove the following (significantly stronger) bound: if $\deg(o)=d$, then 
\[ \Gamma_G\geq \min(\beta_d,2\sqrt{d}+3) ,\]
where $\beta_d$ is defined as
\[\beta_d=\min_{\la \in [\sqrt{d},d]}\frac{\la}{1-\frac{d+1}{(\la+1)^2}}.\]
The approximate values for small $d$ are the following:
\medskip

\begin{center}
\begin{tabular}{|l||c|c|c|c|c|c|c|c|c|c|c|} \hline
$d$ & $3$ & $4$ & $5$ & $6$ & $7$ & $8$ & $9$ & $10$ & $11$ & $12$ \\ \hline
$\beta_d \approx$ & $3.596$ & $4.223$ & $4.788$ & $5.305$ & $5.785$ & $6.235$ & $6.660$ & $7.064$ & $7.450$ & $7.820$ \\ \hline
\end{tabular}
\end{center}
\bigskip

\begin{Lemma} The sequence $\beta_d$ is monotone increasing in $d$. Furthermore, if the function $F_d(\la)=\frac{\la}{1-\frac{d+1}{(1+\la)^2}}$
achieves its minimum at $\la_d$, then
$\beta_d=\frac{3\la_d+1}{2}.$
\end{Lemma}
\begin{proof}
The derivative of $F_d(\la)$ is 
\begin{equation} \label{eq:Fd_der}
\frac{d}{d\la}F_d(\la)=\frac{(1+\la)^3-(d+1)(3\la+1)}{(1+\la)^2\left(1-\frac{d+1}{(1+\la)^2}\right)^2}.
\end{equation}
Here the denominator is always positive. The function $f(\la):=\frac{(1+\la)^3}{3\la+1}$ is strictly monotone increasing if $\la\geq 0$ as $f'(\la)=\frac{6(\la+1)^2\la}{(3\la+1)^2}$. Furthermore, it is easy to see that $f(\sqrt{d})<d+1<f(d)$. It follows that there is a unique $\la_d \in (\sqrt{d},d)$ for which $f(\la)=d+1$, that is, $(1+\la_d)^3-(d+1)(3\la_d+1)=0$. So the numerator of \eqref{eq:Fd_der} is negative for $\sqrt{d} < \la < \la_d$ and positive for $\la_d < \la < d$, and hence $F_d(\la)$ attains its minimum on $[\sqrt{d},d]$ at $\la_d$. Furthermore,
\[ \beta_d:=F_d(\la_d)=\frac{\la_d}{1-\frac{d+1}{(1+\la_d)^2}}=\frac{\la_d}{1-\frac{d+1}{\frac{(d+1)(3\la_d+1)}{\la_d+1}}}=\frac{\la_d}{1-\frac{\la_d+1}{3\la_d+1}}=\frac{3\la_d+1}{2}. \]
Since $f$ is strictly monotone increasing on $(0,\infty)$, it follows that $\la_d$ is a monotone increasing sequence, and hence so is $\beta_d$.
\end{proof}
\begin{Rem}
It can be shown that
\[\beta_d=\frac{3\sqrt{3}}{2}\sqrt{d}-\frac{3}{2}+\frac{2}{\sqrt{3d}}+O\left(\frac{1}{d}\right).\]
\end{Rem}
\begin{Thm} \label{degree-6}
Let $d=\deg(o)$ and let $H=N_G[o]$. Then 
\[ \Gamma(\x\big|_H)\geq \beta_d \quad \text{and} \quad 
\Gamma_G\geq \min(\beta_d,2\sqrt{d}+3) = 
\begin{cases}
\be_d & \text{if } d \leq 52;\\
2\sqrt{d}+3 & \text{if } d \geq 53.
\end{cases}. \]
\end{Thm}

\begin{proof}
On the one hand, $\sum_{v\in V(H)}x_v=(\la+1)x_o$ from the eigenvalue equation at $o$. On the other hand, for $v\in V(H)$ we have $x_o\geq x_v\geq \frac{1}{\la}x_o$. (The first inequality follows from the fact that $o$ is the master, while the second one is a simple consequence of the eigenvalue equation at $v$.)

Thus we can apply Lemma~\ref{reverse-AM-QM} to the $k:=d+1$ numbers $x_v$, $v \in V(H)$ with $m=\frac{1}{\la}x_o$, $M=x_o$ and $S=(1+\la)x_o$. We get that
\[\Gamma(\x\big|_H)\geq \frac{((\la+1)x_o)^2}{\left(1+\frac{1}{\la}\right)(\la+1)x_o^2-\frac{k}{\la}x_o^2}=\frac{\la}{1-\frac{d+1}{(1+\la)^2}}\geq \beta_d.\]
Furthermore, by Lemma~\ref{subgraph-lower-bound}(b) we have
\[\Gamma_G\geq \min\left(\Gamma(\x\big|_H),2\la_G+3\right)\geq \min\left(\beta_d,2\sqrt{d}+3\right) \]
using that $\la_G\geq \sqrt{d}$ by Lemma~\ref{prop:la&deg}.
\end{proof}

The rest of this section is not needed for the conjectures. First we include a result that implies Theorem~\ref{degree-6} and provides a stronger bound for many graphs.
\begin{Lemma} \label{many-large-vertices} 
Let $d=\deg(o)$ and $k\geq d+1$. Suppose that $G$ has at least $k$ vertices $v$ (including $o$) with $x_v\geq \frac{x_o}{\la_G}$. Then 
\[ \Gamma(\x\big|_H)\geq \min(\beta_{k-1},2\sqrt{2k}-1) 
\quad \text{and} \quad \Gamma_G\geq  \min(\beta_{k-1},2\sqrt{2k}-1,2\sqrt{d}+3) .\]
\end{Lemma}

\begin{proof}
Suppose that $x_o=x_1\geq \dots \geq x_k\geq \frac{x_o}{\la}$. We have that
\[S:=\sum_{j=1}^kx_j\geq x_o+\sum_{v\in N_G(o)}x_v=(\la_G+1)x_o.\]
For $M=x_o$ and $m=\frac{1}{\la_G}x_o$ we have 
\[\Gamma(\x\big|_{[k]})\geq \frac{S^2}{(m+M)S-kmM}.\]
The partial derivative (w.r.t.~$S$) of the right-hand side is
\begin{align*}
\frac{\mathrm{d}}{\mathrm{d}S}\left(\frac{S^2}{(m+M)S-kmM}\right)
&=\frac{2S((m+M)S-kmM)-S^2(m+M)}{((m+M)S-kmM)^2}\\
&=\frac{S^2(m+M)-2kmMS}{((m+M)S-kmM)^2} .
\end{align*}
This is positive if $S\geq \frac{2kmM}{m+M}$. Since $S\geq (\la_G+1)x_o$ and $\frac{2kmM}{m+M}=\frac{2k}{\la+1}x_o$, this is automatically satisfied if $\la\geq \sqrt{2k}-1$. So if $\la\geq \sqrt{2k}-1$, then
\[\frac{S^2}{(m+M)S-kmM}\geq \frac{((\la_G+1)x_o)^2}{(m+M)(\la_G+1)x_o-kmM}=\frac{\la_G}{1-\frac{k}{(\la_G+1)^2}}\geq \beta_{k-1}.\]
So from now on we can assume that $\la\leq \sqrt{2k}-1$. Observe that
\[\frac{S^2}{(m+M)S-kmM}\geq \frac{4mMk}{(m+M)^2}\]
Indeed, after multiplying by the denominators which are positive we get that 
\[(m+M)^2S^2-4mMk((m+M)S-kmM)=((m+M)S-2mMk)^2\geq 0.\]
The function 
$\frac{4mMk}{(m+M)^2}=\frac{4k\la}{(\la+1)^2}$
is monotone decreasing if $\la\geq 1$ as its derivative is $\frac{4(\la+1)(1-\la)k}{(\la+1)^2}$, so if $\la\leq \sqrt{2k}-1$, then
\[\frac{4k\la}{(\la+1)^2}\geq \frac{4k(\sqrt{2k}-1)}{((\sqrt{2k}-1)+1)^2}=2(\sqrt{2k}-1).\]
So if $\la\leq \sqrt{2k}-1$, then
\[\Gamma(\x\big|_H)\geq \frac{S^2}{(m+M)S-kmM}\geq \frac{4mMk}{(m+M)^2}=\frac{4k\la}{(\la+1)^2}\geq 2(\sqrt{2k}-1).\]
This completes the proof.
\end{proof}

\subsection{Towards better heaviness constants} \label{sec:better_gamma}

The next lemma may be used to get better bounds for certain graphs. The point is to bound the weights $x_u$ for vertices with $\dist(o,u) \geq 2$ so that we may use a $\gat$ less than $1$, which, in turn, would make our bounds in Section~\ref{sec:key_bounds} stronger (see Theorem~\ref{thm:Ga_xt_bound}). In this paper $\gat=1$ is sufficient to prove the conjectures, but the following lemma may be handy if stronger bounds are needed.

Suppose that we do a breadth-first search on $G$ started at the master $o$ in such a way that we choose the vertices inside a layer\footnote{A \emph{layer} is the set of vertices with the same distance from $o$.} in a decreasing order of their weights $x_u$. By $H_k$ we denote the induced subgraph containing the first $k$ vertices.
\begin{Lemma} \label{gamma-weights}
Let $d=\deg(o)$ and $k\geq d+1$. Suppose that $H=H_k$ is the connected subgraph of $G$ as above.
\begin{itemize}
\item If $u\notin V(H)$ with $\dist(o,u)=2$, then
\begin{align*}
x_u &\leq \frac{\Gamma_G-\la_G-1}{k-d}x_o \text{, and}\\
x_u &\leq \frac{1}{\la_G+k-d}\sum_{v\in V(H)}x_v.
\end{align*}
\item If $u\notin V(H)$ with $\dist(o,u)\geq 3$, then
\[x_u\leq \frac{\Gamma_G-\la_G-1}{\la_G+1}x_o.\]
\end{itemize}
\end{Lemma}
\begin{proof}
Let $u\in V(H)$. If $\mathrm{dist}(o,u)=2$, then it means that $H_k$ has $k-1-d$ vertices of distance $2$ from $o$, all these vertices having a larger weight than $u$. Hence
\[\|\x\|_1\geq x_o+\sum_{v\in N_G(o)}x_v+\sum_{\substack{v\in V(H) \\ \dist(o,v)=2}}x_v+x_u\geq (\la_G+1)x_o+(k-d)x_u.\]
Since $\|\x\|_1\leq \Gamma_Gx_o$, we get that
\[\Gamma_Gx_o\geq (\la_G+1)x_o+(k-d)x_u\]
which is equivalent with the first claim.

To prove the second statement we argue similarly:
\begin{align*}
\sum_{v\in V(H)}x_v
&\geq x_o+\sum_{v\in N_G(o)}x_v 
+\sum_{\substack{v\in V(H) \\ \dist(o,v)=2}}x_v \\
&\geq (\la_G+1)x_o+(k-d-1)x_u\geq (\la_G+k-d)x_u ,
\end{align*}
which is equivalent with the second claim.

If $\mathrm{dist}(o,u)\geq 3$, then 
\[\Gamma_Gx_o\geq \|\x\|_1\geq \sum_{v\in V}x_v\geq x_o+x_u+\sum_{v\in N_G(o)}x_v+\sum_{v\in N_G(u)}x_v=(\la_G+1)(x_o+x_u).\]
Hence
\[x_u\leq \frac{\Gamma_G-\la_G-1}{\la_G+1}x_o.\]
\end{proof}

\section{Bounding via heavy subgraphs}
\label{sec:key_bounds}

This section contains the key ingredient of our approach: the method for bounding $\Ga_G$ by making use of the resolvent matrix corresponding to a heavy subgraph of $G$. In Section \ref{sec:key_method} of the introduction we already gave a fairly detailed account of these results. Here we provide a more general treatment and prove everything rigorously.

Recall the following notations and definitions from Sections~\ref{sec:key_method} and \ref{sec:resolvent}.
\begin{Def} \label{def:Bsc}
Let $H$ be a fixed connected graph and $\la$ a variable (greater than $\la_H$).

As before, set $B=(\la I-A_H)^{-1}$ and consider the extended version $\Bt$ consisting of all partial sums of the column vectors of $B$. With formula:
\[ \Bt_{u,V} := \sum_{v \in V} B_{u,v} \quad 
\text{for a nonempty } V \subseteq V(H) .\]
For column sums of $\Bt$ we write   
\[ s_{V} := \sum_{u \in V(H)} \Bt_{u,V} \quad 
\text{for a nonempty } V \subseteq V(H) ,\]
while inner products of the column vectors of $\Bt$ are denoted by 
\[ c_{U,V} := \sum_{u \in V(H)} \Bt_{u,U} \Bt_{u,V} \quad 
\text{for nonempty } U,V \subseteq V(H) .\]
Note that by Proposition \ref{prop:B_properties} (b)\&(d) each entry $B_{u,v}$ of $B$ is a rational function of $\la$ taking positive values for all $\la > \la_H$. Consequently, the same is true for each $\Bt_{u,V}$, $s_V$, $c_{U,V}$.
\end{Def}
First we state our key lemma and theorem before proving them both.
\begin{Lemma} \label{lem:xt_norms}
Let $H$ be a connected induced subgraph of $G$. Let $\Ht = N_G[H]$ be the closed neighborhood of $H$. By $\xt:=\x \big|_{\Ht}$ we denote the restriction of the Perron vector $\x$ of $G$. Furthermore, for $w \in W := V(\Ht) \sm V(H)$ let $\partial w = N_G(w) \cap V(H)$ denote the set of inside neighbors.

Then $\| \xt \|_1^2$ and $\| \xt \|_2^2$ can be expressed/bounded by quadratic forms as follows. Setting $\la=\la_G$ for the rational functions defined above we have:
\begin{enumerate}
\item [(a)] $\displaystyle \| \xt \|_1 = \sum_{w \in W} \big( s_{\partial w} + 1\big) x_w$, and hence\\
$\displaystyle \| \xt \|_1^2 = \sum_{w,w' \in W} \big( s_{\partial w} + 1\big) \big( s_{\partial w'} + 1\big) x_w x_{w'}$;
\item [(b1)] $\displaystyle \| \xt \|_2^2 = \sum_{w,w' \in W} c_{\partial w,\partial w'} \, x_w x_{w'} + \sum_{w \in W}  x_w^2 \leq \sum_{w,w' \in W} c_{\partial w,\partial w'} \, x_w x_{w'} + \sum_{\substack{w,w' \in W \\ \partial w = \partial w'}} x_w x_{w'}$;
\item [(b2)] if $H$ is $\gamma$-heavy, then\\
$\displaystyle \| \xt \|_2^2 \leq \sum_{w,w' \in W} \left( c_{\partial w,\partial w'} + \frac{\ga}{2} s_{\partial w} + \frac{\ga}{2} s_{\partial w'}\right) x_w x_{w'}$;
\item [(b3)] if $H$ is $\gat$-heavy w.r.t.~$o \in V(H)$, then\\ 
$\displaystyle \| \xt \|_2^2 \leq \sum_{w,w' \in W} \left( c_{\partial w,\partial w'} + \frac{\gat}{2} \Bt_{o,\partial w} + \frac{\gat}{2} \Bt_{o,\partial w'}\right) x_w x_{w'} .$
\end{enumerate}
\end{Lemma}
\begin{Def}
In each quadratic form, the coefficient of $x_w x_{w'}$ depends on $w$ and $w'$ only through $\partial w$ and $\partial w'$. This motivates the introduction of the following notations corresponding to the coefficients of the quadratic forms in the lemma: for $\emptyset \neq U,V \subseteq V(H)$ let 
\begin{align} \label{eq:coeff_def}
\begin{split}
a_{U,V} &:= \big( s_U + 1\big) \big( s_V + 1\big) ;\\
b^{(1)}_{U,V} &:= c_{U,V} + \delta_{U,V} ;\\
b^{(2)}_{U,V} &:= c_{U,V} + \ga s_U/2 + \ga s_V/2 ;\\
b^{(3)}_{U,V} &:= c_{U,V} + \gat\Bt_{o,U}/2+ \gat\Bt_{o,V}/2 ,
\end{split}
\end{align}
where $\delta_{U,V}$ is $1$ when $U=V$ and $0$ otherwise.
\end{Def}
\begin{Thm} \label{thm:Ga_xt_bound}
Let $H \subset G$ be connected graphs as above, and let the rational functions $a_{U,V}$ and $b^{(i)}_{U,V}$ ($i=1,2,3$) be defined as in \eqref{eq:coeff_def}. Furthermore, let 
\[ \Uc^G_H := \big\{ \partial w \, : \, w \in W \big\} .\]
Then we have the following lower bounds on $\Ga(\xt)$:
\begin{enumerate}[(i)]
\item $\displaystyle \Ga(\xt) \geq \min_{U,V \in \Uc^G_H} \,\, 
\frac{a_{U,V}\big( \la_G \big)}{b^{(1)}_{U,V}\big( \la_G \big)}$;
\item $\displaystyle \Ga(\xt) \geq \min_{U,V \in \Uc^G_H} \,\, 
\frac{a_{U,V}\big( \la_G \big)}{b^{(2)}_{U,V}\big( \la_G \big)}$ 
provided that $H$ is $\gamma$-heavy;
\item $\displaystyle \Ga(\xt) \geq \min_{U,V \in \Uc^G_H} \,\, 
\frac{a_{U,V}\big( \la_G \big)}{b^{(3)}_{U,V}\big( \la_G \big)}$ 
provided that $H$ is $\gat$-heavy w.r.t.~$o \in V(H)$.
\end{enumerate}
\end{Thm}

\begin{proof}[Proof of Lemma \ref{lem:xt_norms}]
We saw in Lemma \ref{lem:inside_weights} that after setting $\la=\la_G$ we have $\x\big|_{H} = B \y$, where $\y$ was defined by 
\[ y_v =\sum_{w \in N_G(v) \setminus V(H)} x_w .\] 
Consequently, for each $u \in V(H)$ we have 
\[ x_u = \sum_{v \in V(H)} B_{u,v} \, y_v 
= \sum_{v \in V(H)} B_{u,v} \sum_{w \in N_G(v) \setminus V(H)} x_w 
= \sum_{w \in W} x_w \sum_{v \in \partial w} B_{u,v} 
= \sum_{w \in W} \Bt_{u,\partial w} \, x_w .\]
Therefore 
\[ 
\sum_{u \in V(H)} x_u 
= \sum_{u \in V(H)} \sum_{w \in W} \Bt_{u,\partial w} \, x_w
= \sum_{w \in W} s_{\partial w} \, x_w ,\]
and (a) follows by adding $\sum_{w \in W} x_w$ to both sides. 
Similarly, we have  
\begin{equation} \label{eq:xH_l2}
\sum_{u \in V(H)} x_u^2 
= \sum_{u \in V(H)} 
\bigg( \sum_{w \in W} \Bt_{u,\partial w} \, x_w \bigg)^2 
= \sum_{w,w' \in W} c_{\partial w,\partial w'} \, x_w x_{w'} .
\end{equation}
Note that (b1) follows immediately by adding $\sum_{w \in W} x_w^2$ to both sides. For (b2) and (b3), we need to bound the term $\sum_{w \in W} x_w^2$ under the given assumptions.

If $H$ is $\ga$-heavy, then for each $w \in W$  
\[ x_w^2 \leq x_w \bigg( \ga \sum_{u \in V(H)} x_u \bigg)
= \ga x_w \sum_{w' \in W} s_{\partial w'} \, x_{w'} ,\]
and hence 
\[ \sum_{w \in W} x_w^2 
\leq \ga \sum_{w,w' \in W} s_{\partial w'} \, x_w x_{w'} 
= \ga \sum_{w,w' \in W} \frac{s_{\partial w}+s_{\partial w'}}{2} \, x_w x_{w'} ,\]
and (b2) follows by adding this to \eqref{eq:xH_l2}.

Similarly, if $H$ is $\gat$-heavy w.r.t.~$o$, then for each $w \in W$  
\[ x_w^2 \leq x_w ( \gat x_o ) 
= \gat x_w \sum_{w' \in W} B_{o,\partial w'} \, x_{w'} ,\]
and hence 
\[ \sum_{w \in W} x_w^2 
\leq \gat \sum_{w,w' \in W} B_{o,\partial w'} \, x_w x_{w'} 
= \gat \sum_{w,w' \in W} \frac{B_{o,\partial w}+B_{o,\partial w'}}{2} \, x_w x_{w'} ,\]
and (b3) follows by adding this to \eqref{eq:xH_l2}.
\end{proof}
\begin{proof}[Proof of Theorem \ref{thm:Ga_xt_bound}]
We simply need to use the inequality 
\[ \frac{\sum a_k}{\sum b_k} \geq \min_k \frac{a_k}{b_k} \]
that clearly holds for any positive numbers $a_k,b_k$. 

Setting $\la=\la_G$ again, by Lemma \ref{lem:xt_norms} we have
\[ \Ga(\xt) \geq \frac{ \sum_{w,w' \in W} a_{\partial w, \partial w'} \, x_w x_{w'}}{\sum_{w,w' \in W} b^{(i)}_{\partial w, \partial w'} \, x_w x_{w'}} 
\geq \min_{w,w' \in W} \frac{a_{\partial w, \partial w'}}{b^{(i)}_{\partial w, \partial w'}} 
= \min_{U,V \in \Uc^G_H} \frac{a_{U,V}}{b^{(i)}_{U,V}} \]
under the given assumptions for $i=1,2,3$. We used that all the weights $x_w$ and all the coefficients $a_{U,V}$ and $b^{(i)}_{U,V}$ are positive.
\end{proof}

Next, we state the exact form of the bound that we will use when proving the conjectures. We start with a couple of definitions.
\begin{Def} \label{U-extension}
Let $H$ be a finite connected graph, $o\in V(H)$ a distinguished vertex, and $\Uc$ a set of nonempty subsets of $V(H)$. We say that a graph $G$ is a $\Uc$-extension of the rooted graph $(H,o)$ if it satisfies the following properties:
\begin{itemize}
\item $G$ is connected and contains $H$ as a proper induced subgraph;
\item $o$ is the master of $G$, that is, $x_o$ is the largest entry in the Perron vector $\x$ of $G$;
\item $\Uc_H^G \subseteq \Uc$, that is, for any vertex $w\in V(G) \setminus V(H)$ the set $\partial w = N_G(w) \cap V(H)$ is either the empty set or lies in $\Uc$.
\end{itemize}
\end{Def}

The motivation behind Definition~\ref{U-extension} is that when we have extra information regarding the ``outside connections'' of $H$, then we can often narrow down what $\partial w$ might be, potentially leading to stronger bounds. For instance, we often know that $o$ has no neighbors outside $H$, meaning that we can immediately exclude all sets containing $o$ from $\Uc$. The situation is even better for trees: if $H$ is a subtree, then each $\partial w$ has one element.

\begin{Def} \label{def:laU}
Let $H$ be fixed. Given a set $U\subseteq V(H)$, let $\la_U$ denote the largest eigenvalue of the graph obtained from $H$ by adding a new vertex $w$ to $H$ and connecting $w$ to the vertices in $U$. 
\end{Def}
Note that if $H$ is an induced subgraph of $G$, then for every $w \in W$ it must hold that  
\[ \la_G \geq \la_{\partial w} > \la_H .\]
In other words, $\la_G \geq \la_U$ for every $U \in \Uc_H^G$.

\begin{Thm} \label{thm:Ga_xt_bound_spec}
Suppose that the rooted graph $(H,o)$ is connected and consider the corresponding rational functions $\Bt_{u,V}, s_V, c_{U,V}$ as in Definition \ref{def:Bsc} (in the variable $\la$) and the corresponding eigenvalues $\la_U$ as in Definition \ref{def:laU}. 

If $G$ is a $\Uc$-extension of $(H,o)$, then for $\xt = \x\big|_{N_G[H]}$ we have
\[ \Ga(\xt) \geq \min_{U,V \in \Uc} \;
\min_{\la \geq \max(\la_U,\la_V)} 
\frac{\big( s_U + 1\big) \big( s_V + 1\big)}{c_{U,V} + \Bt_{o,U}/2 + \Bt_{o,V}/2} .\]
\end{Thm}
\begin{proof}
Since $o$ is the master of $G$ (second condition of being a $\Uc$-extension), we can use part (iii) of Theorem \ref{thm:Ga_xt_bound} with $\gat=1$:
\[ \Ga(\xt) \geq \min_{U,V \in \Uc^G_H} \; 
\frac{a_{U,V}\big( \la_G \big)}{b^{(3)}_{U,V}\big( \la_G \big)} ,\]
where 
\[   a_{U,V}=\big( s_U + 1\big) \big( s_V + 1\big) 
\text{ and } b^{(3)}_{U,V} = c_{U,V} + \Bt_{o,U}/2 + \Bt_{o,V}/2 .\]

As we pointed out, $\la_G \geq \la_U$ for every $U \in \Uc_H^G$. Therefore, for any pair $U,V \in \Uc_H^G$ it holds that $\la_G \geq \max(\la_U,\la_V)$, and hence 
\[ \Ga(\xt) \geq \min_{U,V \in \Uc^G_H} \; 
\frac{a_{U,V}\big( \la_G \big)}{b^{(3)}_{U,V}\big( \la_G \big)} 
\geq \min_{U,V \in \Uc^G_H} \; \min_{\la \geq \max(\la_U,\la_V)}
\frac{a_{U,V}}{b^{(3)}_{U,V}} 
\geq \min_{U,V \in \Uc} \; \min_{\la \geq \max(\la_U,\la_V)}
\frac{a_{U,V}}{b^{(3)}_{U,V}} ,\]
where the last inequality holds because $\Uc_H^G \subseteq \Uc$ (third condition of being a $\Uc$-extension).
\end{proof}

The previous theorem is the cornerstone of our proofs: in Section~\ref{sec:kernel} we will use it to determine what the minimizing graphs look like around their master vertices. Next we make some remarks regarding how we can efficiently check that the bound given by this theorem is good enough (i.e., it is above some target ratio).
\begin{Rem} \label{rem:poly}
Suppose that, given some target ratio $\be$, we want to show that $\Ga(\xt) \geq \be$ for every $\Uc$-extension of $(H,o)$. According to the theorem, it suffices to check that 
\begin{equation} \label{eq:coeff_check}
\text{for every } U,V \in \Uc: \quad
\frac{\big( s_U + 1\big) \big( s_V + 1\big)}{c_{U,V} + \Bt_{o,U}/2 + \Bt_{o,V}/2} \geq \be 
\quad (\forall \la \geq \max(\la_U,\la_V)) .
\end{equation}
How can we do this efficiently for a given $H,o,U,V$? After rearranging:
\begin{equation} \label{eq:coeff_check_rearranged}
\big( s_U + 1\big) \big( s_V + 1\big) 
- \be \big( c_{U,V} + \Bt_{o,U}/2 + \Bt_{o,V}/2 \big) \geq 0 
\quad (\forall \la \geq \max(\la_U,\la_V)) .
\end{equation}
Note that on the left-hand side of \eqref{eq:coeff_check_rearranged} we have a  rational function of $\la$. Indeed, according to \eqref{eq:adj} in Proposition \ref{prop:B_properties}(d), $B$ multiplied by the characteristic polynomial $P(\la)$ has polynomial entries. It follows that $P \Bt_{u,V}, P s_V, P^2 c_{U,V}$ are all polynomials of $\la$ (with integer coefficients, in fact). Consequently, if we multiply \eqref{eq:coeff_check_rearranged} by $P^2$, then we get a polynomial in $\la$:
\begin{equation} \label{eq:coeff_check_poly}
Q = Q_{U,V}(\la) := 
\big( P s_U + P\big) \big( P s_V + P\big) 
- \be \big( P^2 c_{U,V} + P^2 \Bt_{o,U}/2 + P^2 \Bt_{o,V}/2 \big) ,
\end{equation}
and we need to verify that $Q(\la) \geq 0$ for all $\la \geq \widetilde{\la}:= \max(\la_U,\la_V)$. In fact, we will check the following simple sufficient condition. We compute the coefficients of the polynomial $\widetilde{Q}(\ka) := Q(\widetilde{\la}+\ka)$, and if each coefficient is positive, then we can conclude that $\widetilde{Q}(\ka) \geq 0$ for any $\ka \geq 0$, which means that $Q(\la) \geq 0$ for any $\la \geq \widetilde{\la}$, as required.

It is easy to see that the degree of $\widetilde{Q}$ is $2|V(H)|$, and therefore for each pair $U,V$ it suffices to check the positivity of $2|V(H)|+1$ coefficients.
\end{Rem}

We finish this section by commenting on how one can get a bound on $\Ga_G$ from the bound on $\Ga(\xt)$.
\begin{Rem} \label{rem:Gaxt_to_GaG}
If the master $o$ of $G$ lies in $H$, then  Lemma~\ref{subgraph-lower-bound}(b) can be always applied for the subgraph $N_G[H]$ (because it always contains $N_G[o]$). Therefore, we have 
\[ \Ga_G \geq \min\big( \Ga(\xt) , 2\la_G+3 \big) .\]
Moreover, if we know that $N_G[H]$ contains at least one vertex at distance $2$ of $o$, then we have the following stronger bound by part (c) of the same lemma:
\[ \Ga_G \geq \min\left( \Ga(\xt) , 2\la_G+\frac{2}{\la_G^2-1}+3 \right) .\]
In our applications we will have $\la_G>2$ so 
\[ 
2\la_G+3 > 7 > 5.25 > \bestar 
\quad \text{and} \quad 
2\la_G+\frac{2}{\la_G^2-1}+3 > 7+\frac{2}{3} > \betr.
\]
So if Theorem~\ref{thm:Ga_xt_bound_spec} gives us a bound $\Ga(\xt) \geq \beta$ for some $\beta$ less than $7$ (or $7+2/3$ in the stronger version), then the same bound holds for $\Ga_G$ as well: $\Ga_G \geq \beta$.
\end{Rem}

\section{Determining the optimal kernel with a computer}
\label{sec:kernel}

In this section we show how the bounds of the previous section can be used to deduce what $G$ must look like around its master vertex for the graphs/trees minimizing $\Ga_G$. We will use a computer to go through all possible ``kernels'' and rule out all but one. For a SageMath code and the results of the run, see the Jupyter notebook \cite{JUP}.

\begin{Def}
Given a connected graph $G$ with master vertex $o$ and a positive integer $k<|V(G)|$, we say that the rooted graph $(H,o)$ is a \emph{$k$-kernel} of $G$ if $V(H)$ consists of $k$ vertices closest to $o$ (more precisely, if $V(H)$ can be obtained as the first $k$ vertices of a breadth-first search started at $o$).
\end{Def}
\begin{Rem} \label{layers}
Given a connected graph $G$ with master vertex $o$, for an integer $r \geq 0$ let 
\[ V_r := \big\{ v \in V(G) \, : \, \dist(o,v)=r \big\} \]
be the set of vertices at distance $r$ from $o$.\footnote{Note that $V_r$ ($r=0,1,\ldots$) are the layers of a breadth-first search started at $o$.} In particular, $V_0=\{o\}$ and $V_1=N_G(o)$. Furthermore, let $V_{\leq r} := \bigcup_{i \leq r} V_i$. 

Suppose that $\big| V_{\leq r} \big| \leq k < \big| V_{\leq r+1} \big|$. Note that $(H,o)$ is a $k$-kernel of $G$ if and only if $V_{\leq r} \subseteq V(H) \subsetneq V_{\leq r+1}$ and $|V(H)|=k$. 

We remark that only vertices in 
\[ A := V_r \cup \big( V(H) \cap V_{r+1} \big) \]
may have neighbors outside $H$. It follows that $G$ is a $\Uc$-extension of $(H,o)$, where $\Uc = \Pc_+(A)$ is the set of all nonempty subsets of $A$. (See Definition~\ref{U-extension} for the definition of $\Uc$-extension.)
\end{Rem}
The next simple claim says that no vertex in a $k$-kernel can dominate the master (in terms of their neighbor sets) provided that the kernel contains all neighbors of $o$. 
\begin{Prop} \label{prop:no_dom}
Let $H$ be an induced subgraph of $G$ containing the master vertex $o$ and all its neighbors. Then
\begin{equation} \label{eq:no_dom}
\nexists v \in V(H) \sm \{o\} \text{ such that } 
N_H[v] \supsetneq N_H[o] .
\end{equation}
\end{Prop}
\begin{proof}
We prove by contradiction. Assume that $N_H[v] \supsetneq N_H[o]$ for a vertex $v \in V(H)$. Then 
\[ N_G[o]=N_H[o] \subsetneq N_H[v] \subseteq N_G[v] ,\]
and hence
\[ \la_G x_o = \sum_{u \in N_G(o)} x_u 
< \sum_{u \in N_G(v)} x_u = \la_G x_v \]
implying that $x_o<x_v$, which contradicts the assumption that $o$ is the master.
\end{proof}

\subsection{Connected graphs} \label{subsec:comp_one}

This section contains \textbf{the proof of Theorem \ref{thm:525}}. Let $G$ be a connected graph on $n\geq 7$ vertices with master $o$. The main part of the proof is to show that if $\Ga_G<5.25$, then the $6$-kernel of $G$ must be isomorphic to $\big( K_4 +_o P_2, o \big)$, with the single exception of the $7$-vertex graph $\DsPPP$.

First of all, we may assume that $\la_G>2$ because Theorem~\ref{lambda2gamma} lists $\Ga_G$ for all connected graphs $G$ with $\la_G \leq 2$, and $\Ga_G \geq 6>5.25$ for every such $G$ with $|V(G)| \geq 7$.

In the proof we will repeatedly apply Theorem~\ref{thm:Ga_xt_bound_spec} and obtain the bound $\Ga(\xt) \geq 5.25$. According to Remark~\ref{rem:Gaxt_to_GaG}, this always automatically implies $\Ga_G \geq 5.25$ as well because we have $2 \la_G + 3 > 2\cdot 2 + 3 = 7 > 5.25$. 

Set $d := \deg(o)$. On the one hand, $d \geq \la_G$ by Lemma~\ref{prop:la&deg}. Since $\la_G>2$, we get $d \geq 3$. On the other hand, if $d \geq 6$, then Theorem~\ref{degree-6} shows that $\Ga_G \geq 5.3$. 

From this point on, we will assume that $d \in \{3,4,5\}$. Let $(H,o)$ be a $6$-kernel of $G$. Note that $H \supseteq N_G[o]$, hence Proposition \ref{prop:no_dom} applies and \eqref{eq:no_dom} holds. We summarize what we know of the rooted graph $(H,o)$:
\begin{itemize}
\item $H$ is connected;
\item $|V(H)|=6$;
\item $\deg_H(o) \geq 3$;
\item $\nexists v \in V(H) \sm \{o\}$ such that $N_H[v] \supsetneq N_H[o]$ .
\end{itemize}
A quick computer search shows that, up to (rooted) isomorphism, there exist $155$ rooted graphs with these properties.\footnote{There are $112$ connected graphs on $6$ vertices. For many of them there is only one possible root satisfying the last two conditions.} We will simply refer to them as \emph{(possible) kernels}.

For a given kernel $(H,o)$, we define the set $\Vact$ of active vertices (that may have outside neighbors in some $G$) as follows: 
\begin{equation} \label{eq:Vact}
\Vact := 
\begin{cases}
V(H) \sm \{o\} & \text{if } \dist(o,v) \leq 2 \text{ for each } v \in V(H) ; \\
V(H) \sm N_G[o] & \text{if } \dist(o,v) =3 \text{ for some } v \in V(H) .
\end{cases}
\end{equation}
It is easy to see that only these active vertices of $H$ may have neighbors outside $H$. In other words, $G$ is a $\Uc$-extension of $(H,o)$ for $\Uc=\Pc_+(\Vact)$, see Remark \ref{layers}.

So, given a kernel $(H,o)$ with active vertices $\Vact$ defined as in \eqref{eq:Vact}, we need to examine whether it is true that 
\[ (\ast) \quad \Ga_G \geq 5.25 
\text{ for every $\Pc_+(\Vact)$-extension $G$ of $(H,o)$.}
\]
There is one kernel for which $(\ast)$ is not true at all: $\KKKKPPr \, \cong \big( K_4 +_o P_2, o \big)$. For most other kernels $(\ast)$ is an immediate consequence of Theorem~\ref{thm:Ga_xt_bound_spec} once condition \eqref{eq:coeff_check} is checked by a computer. There are four kernels, where $(\ast)$ is essentially true (except for one single graph $G$) but we will need to be a bit more careful when we apply the theorem. 
\begin{Def} \label{def:exceptions}
Let $D$ be the diamond graph with $s$ and $t$ being vertices of degree $2$ and $3$, respectively:

\begin{center}
\diamond
\end{center}

\noindent Then the four \emph{exceptional kernels} are the following:\footnote{The root is indicated by a red square.}
\begin{align} \label{eq:exceptions}
\begin{split}
\KKKPPPr \, &\cong \big( K_3 +_o P_3, o \big) ; \\
\DsPPrs \, &\cong \big( D +_s P_2, s \big) ;\\
\DsPPrt \, &\cong \big( D +_s P_2, t \big) ;\\
\DtPPrt \, &\cong \big( D +_t P_2, t \big) .\\
\end{split}
\end{align}
\end{Def}

For every kernel other than $\big( K_4 +_o P_2, o \big)$ and the 4 exceptional kernels listed in \eqref{eq:exceptions}, we can simply apply Theorem~\ref{thm:Ga_xt_bound_spec}.\footnote{For a concrete example, see the Appendix, where we demonstrated on a specific kernel what calculations are needed for the computer verification.}
\bigskip

\noindent\fbox{\begin{minipage}{0.98\textwidth}
\textbf{Computer check for 150 kernels:} condition \eqref{eq:coeff_check} verified for 
\[ (H,o); \quad 
\Uc=\Pc_+(\Vact); \quad
\be=5.25. \]
\end{minipage}}
\bigskip

Now we can turn our attention to the exceptional kernels, for which we will do the following \texttt{two-step verification}. Notice that for each of them $H$ has exactly one leaf that we will denote by $v$. It turns out that condition \eqref{eq:coeff_check} only fails when $U=V=\{v\}$. It means that we may apply the theorem with $\Uc' := \Pc_+(\Vact) \sm \big\{ \{v\} \big\}$. 
\bigskip

\noindent\fbox{\begin{minipage}{0.98\textwidth}
\textbf{Computer check for 4 exceptional kernels:} condition \eqref{eq:coeff_check} verified for
\[ (H,o); \quad 
\Uc'=\Pc_+(\Vact) \sm \big\{ \{v\} \big\}; \quad
\be=5.25. \]
\end{minipage}}
\bigskip

After this verification we can conclude that $\Ga_G \geq 5.25$ for any $G$ that has no outside vertex $w$ with $\partial w = \{v\}$.

So it remains to consider the case when $G$ has a vertex $w$ with $\partial w = \{v\}$. Then we may add $w$ to $H$ to get a $7$-vertex induced subgraph $H^+ := G\big[ V(H) \cup \{w\} \big] \cong H +_v P_1$. We can apply Theorem~\ref{thm:Ga_xt_bound_spec} again, this time with $(H^+,o)$ and $\Uc^+ := \Pc_+\big( \Vact \cup \{w\} \big)$. 
\bigskip

\noindent\fbox{\begin{minipage}{0.98\textwidth}
\textbf{Computer check for 4 exceptional kernels:} condition \eqref{eq:coeff_check} verified for
\[ (H^+,o); \quad 
\Uc^+=\Pc_+\big( \Vact \cup \{w\} \big); \quad
\be=5.25. \]
\end{minipage}}
\bigskip

This \texttt{two-step verification} proves $(\ast)$ except for a single $G$: the graph $H^+ = H +_v P_1$ itself is not covered by either application of Theorem~\ref{thm:Ga_xt_bound_spec}. So it remains to check $\Ga_G$ for these graphs:
\begin{align*}
\KKKPPPP: \quad \Ga_{K_3+P_4} &\approx 5.28092 ;\\ 
\DsPPP: \quad \Ga_{D +_s P_3} &\approx 5.180545 < 5.25 ;\\ 
\DtPPP: \quad \Ga_{D +_t P_3} &\approx 5.287096 .
\end{align*}
Consequently, if we disregard $G=\DsPPP$, then the only way $\Ga_G$ may be below $\beta=5.25$ is if the $6$-kernel is $\big( K_4 +_o P_2, o \big)$. From here we can finish the proof easily. Let $C$ denote the $4$-clique in this $6$-kernel, and let $u \in V(H) \sm C$ be the additional neighbor of $o$, see Figure \ref{fig:K4_plus_Gpr}. It remains to show that the only edge between $C$ and $V(G) \sm C$ is $(o,u)$. Assume that there is another edge connecting $C$ and $w \notin C$. We know that there is no such edge in $H \cong K_4+P_2$ so $w \notin V(H)$. Consider the induced subgraph $H' := G\big[ C \cup \{u,w\} \big]$. Clearly, $(H',o)$ is also a $6$-kernel of $G$. Since $H'$ has more than one vertex with degree $4$ (or higher), it follows that $(H',o) \not\cong (K_4 +_o P_2,o)$, meaning that $G$ has a $6$-kernel that we have already checked. So the only way $\Ga_G$ could be below $5.25$ is if $G \cong \DsPPP$, which is clearly not the case, either. The proof is complete.

\begin{Rem}
Each condition that the computer needs to check during this proof is a simple polynomial inequality \eqref{eq:coeff_check_poly}, and in our setting we can always verify it by simply checking the positivity of the coefficients of a corresponding $\widetilde{Q}$ (obtained after a change of variables), see Remark \ref{rem:poly} for details.

How many polynomials to check? For each $6$-vertex kernel, $|\Vact| \leq 5$ and hence $|\Uc|\leq 31$. So there are at most $\binom{32}{2} = 496$ pairs $U,V$ to consider. In total, we need to check around $100,000$ polynomial inequalities, which can be done with a personal computer in half a minute. 
\end{Rem}

\subsection{Trees}

For trees we use the same approach to show that the minimizing trees have the conjectured kernel. We aim to show that, apart from a small number of exceptions, $\Ga_T<\betr$ happens only if the $10$-kernel of $T$ is $S_5+P_5$. The exact result is the following.
\begin{Thm} \label{thm:10kernel}
Suppose that $T$ is a tree on $n \geq 11$ vertices with master $o$. Then $\Ga_T<\betr$ implies that the $10$-kernel of $T$ is isomorphic to $\big( S_5 +_o P_5, o \big)$ unless $T$ is one of the following three exceptions: 
\[ T=S_6 + P_k, \quad k=5,6,7 . \]
\end{Thm}
\begin{proof}
As in the case of connected graphs, we may assume that $\la_T>2$ because Theorem~\ref{lambda2gamma} lists $\Ga_G$ for all connected graphs $G$ with $\la_G \leq 2$, and $\Ga_G > 9.393 > \betr$ for every such $G$ with $|V(G)| \geq 11$.

By Lemma~\ref{prop:la&deg} it follows that $d := \deg_T(o) \geq 3$.

Now let $(H,o)$ denote a $10$-kernel of $T$. We need to iterate through all possible kernels $(H,o)$. This time they have the following properties:
\begin{itemize}
\item $H$ is a tree;
\item $|V(H)|=10$;
\item $\deg_H(o) \geq 3$.
\end{itemize}
Up to (rooted) isomorphism, there are $194$ such rooted trees.

For a given kernel $(H,o)$, we define the set of active vertices as follows:
\[ \Vact := \big\{ v \in V(H) \, : \, \dist(o,v) \geq \ell - 1 \big\}, \quad
\text{where } \ell := \max_{v\in V(H)} \dist(o,v) .\] 
It is easy to see that only active vertices may have neighbors outside $H$, see Remark \ref{layers} for details.\footnote{We point out that here it may happen that $H$ does not contain all neighbors of $o$. If this happens, then $H$ must be the star graph $S_{10}$. Then $\Vact = V(H)$, and the root is active and may have further neighbors outside $H$. (Had we considered $13$-kernels, we could assume $o \notin \Vact$, because $\Ga_T<\bestar$ implies that $\deg(o)\leq 12$ by Theorem~\ref{degree-6}.)} Moreover, every $\partial w$ must be a one-element set. Indeed, if $v_1, v_2 \in V(H)$ were both adjacent to the same outside vertex $w \notin V(H)$, then the path in $H$ connecting $v_1$ and $v_2$ could be extended into a cycle, contradicting that $T$ is a tree. Consequently, $T$ is a $\Uc$-extension of $(H,o)$ for 
\[ \Uc := \big\{ \{v\} \, : \, v \in \Vact \big\} . \]
Compared to the first conjecture, here we need to consider larger kernels (of order $10$ versus $6$). So it is crucial for us that we do not need to consider all nonempty subsets of $\Vact$, only the singletons. This way the total number of polynomial inequalities to check remains well under $10,000$ in this proof. A personal computer can do this within seconds.

Now we are ready to start applying Theorem~\ref{thm:Ga_xt_bound_spec} for the different kernels $(H,o)$. First we argue that we may assume that $N_T[H]$ contains a vertex at distance $2$ from $o$. If $T$ has no such vertex at all, then $T$ is the star graph $S_n$ for which $\Ga_T=\big(\sqrt{n-1}+1\big)^2/2 \geq 8$ if $n\geq 10$. 
For any other tree $T$, we can choose the $10$-kernel $(H,o)$ in such a way that $N_T[H]$ contains at least one vertex at distance $2$ from $o$. Therefore, according to Remark~\ref{rem:Gaxt_to_GaG}, the bounds $\Ga(\xt) \geq \betr$ obtained by Theorem~\ref{thm:Ga_xt_bound_spec} will automatically imply $\Ga_T \geq \betr$ as well.

We can rule out most of the kernels by a direct application of Theorem~\ref{thm:Ga_xt_bound_spec} as follows. 

\bigskip

\noindent\fbox{\begin{minipage}{0.98\textwidth}
\textbf{Computer check for $191$ kernels:} condition \eqref{eq:coeff_check} verified for 
\[ (H,o); \quad 
\Uc=\big\{ \{v\} \, : \, v \in \Vact \big\}; \quad
\be=\betr. \]
\end{minipage}}
\bigskip

There are three kernels $(H,o)$ for which the above simple method would fail and hence more careful considerations are required:\footnote{The root is indicated by a red square.}
\begin{itemize}
\item Case 1: $\quad \treekernelC \, \cong \big( (S_4 +_o P_2) +_o P_4 , o \big)$
\item Case 2: $\quad \treekernelB \, \cong \big( S_6 +_o P_4 , o \big)$
\item Case 3: $\quad \treekernelA \, \cong \big( S_5 +_o P_5 , o \big)$
\end{itemize}

We call the following process \texttt{active-vertex-elimination}. We pick an active vertex $u \in \Vact$ and add a leaf $w$ to $H$ at $u$, that is, we take $H' := H +_u P_1$ and define its active vertices as $\Vact' := \Vact \cup \{ w \}$. Then we perform the computer check for this new system:
\[ (H',o); \quad 
\Uc=\big\{ \{v\} \, : \, v \in \Vact' \big\}; \quad
\be=\betr . \]
If the conditions are verified for this system, then we may remove $u$ from the set of active vertices $\Vact$ and proceed with this smaller $\Vact$ for the original kernel $(H,o)$. (Additionally, we need to check $\Ga_{H'}$ for the single tree $H'$.) If we manage to remove, one by one, all active vertices so that $\Vact$ becomes $\emptyset$, then we can rule this kernel out. This happens in Case 1. In that case, not even the potential single trees $H'$ provide examples for $\Ga_T < \betr$.

In Case 2, after the \texttt{active-vertex-elimination} process, $\Vact$ does not get empty, however. Instead, one element remains in $\Vact$, namely the endpoint of the pendant path. So the only way this kernel may be extended (with any chance of resulting in a tree $T$ with $\Ga_T < \betr$) is through that endpoint. Consequently, we can replace this kernel with $\big( S_6 +_o P_5 , o \big)$ and try the \texttt{active-vertex-elimination} for this larger kernel. The same thing happens: only one vertex remains in $\Vact$, and hence we can move on to $\big( S_6 +_o P_6 , o \big)$. Once again, we conclude that the only way to extend this $12$-kernel is through the endpoint. Finally, this procedure stops at the $13$-kernel $\big( S_6 +_o P_7 , o \big)$, for which $\Vact$ gets empty through our \texttt{active-vertex-elimination} process. Of course, we have to check the potential single trees along the way. In this case three of them happen to have $\Ga_T$ below $\betr$:
\begin{itemize}
\item $T=S_6 + P_5: \quad \Ga_T \approx 7.3371$;
\item $T=S_6 + P_6: \quad \Ga_T \approx 7.4158$;
\item $T=S_6 + P_7: \quad \Ga_T \approx 7.4571$.
\end{itemize}
Therefore, unless $T$ is one of these trees, it follows that $\Ga_T<\betr$ implies that we must be in Case 3, meaning that the $10$-kernel of $T$ is $\big( S_5 +_o P_5 , o \big)$, as claimed. 
\end{proof}

As in Theorem \ref{thm:525} (in the case of connected graphs), 
we could have used some target ratio $\beta$ somewhat above the limiting ratio $\betr$ and the same approach would yield similar results but with further exceptional kernels and single trees.

\section{Graphs with long pendant paths}
\label{sec:pendant_paths}

In this section we study $\la_G$ and $\Ga_G$ for graphs with long pendant paths, both in the finite setting $G=H +_v P_k$ and in the infinite setting $G=H +_v P_{\infty}$, where $H$ is a finite connected graph and $v$ a vertex of $H$. The main goal is to prove that $\Ga_{H +_v P_k} < \Ga_{H +_v P_\infty}$ under mild conditions, see Lemma~\ref{lem:Gamma-lower}. Then we will check these conditions for the relevant special cases ($H=K_4$ and $H=S_5$), see Theorem~\ref{thm:indeed_below}. 

\subsection{Eigenvectors on pendant paths} 
\label{sec:ev_on_path} 
We start with the following simple observation that we will repeatedly use in the remainder of the paper.

Suppose that a connected graph $G$ contains a path $w_0 w_1 \ldots w_{k-1}$ in such a way that $\deg_G(w_i)=2$ for $1 \leq i \leq k-2$, that is, the inner vertices of the path have no further neighbors in $G$. Consider an eigenvector of the graph with eigenvalue $\la>2$ and let $x_i$ denote the coordinate corresponding to $w_i$. For $2 \leq i \leq k-1$, we have $x_i - \la x_{i-1} + x_{i-2} = 0$ by the eigenvalue equation at $w_{i-1}$. Therefore, this is a linear recurrence of order $2$ (with constant coefficients). So one needs to solve the characteristic equation $q^2-\la q + 1 = 0$. For $\la>2$ it has two distinct roots:
\begin{equation} \label{eq:q}
r(\la):=\frac{\la+\sqrt{\la^2-4}}{2}>1 
\quad \text{and} \quad 
\frac{1}{r(\la)}=\frac{\la-\sqrt{\la^2-4}}{2}<1 .
\end{equation}
Since $r(\la)$ is the larger root, we have 
\begin{equation} \label{eq:q_bounds}
\la=r(\la)+\frac{1}{r(\la)}<2r(\la)
\text{, and hence }
\frac{1}{r(\la)}<\frac{2}{\la} .
\end{equation}
We know that the linear recurrence is satisfied by sequences of the following form: $x_i = \alpha r(\la)^i + \beta r(\la)^{-i}$. It is convenient to normalize the eigenvector such that $x_0=1$, in which case $\be=1-\al$ and we have 
\[ x_i= \alpha r(\la)^i + (1-\alpha) r(\la)^{-i} 
\quad (i=0,1,\ldots,k-1) \]
for some constant $\alpha$.

Depending on what the rest of the graph looks like, we may have additional equations that give us more information about  $\la$ and $\alpha$. For instance, if $w_{k-1}$ is a leaf of the graph, that is, $w_0 w_1 \ldots w_{k-1}$ is a \textbf{pendant path}, then 
\[ 0 = \al r(\la)^{k}+(1-\al) r(\la)^{-k} , \]
because we can view this as if there were a vertex $w_k$ with $x_k=0$. We conclude that 
\begin{equation} \label{eq:alpha}
\alpha=\frac{-r(\la)^{-k}}{r(\la)^k-r(\la)^{-k}}. 
\end{equation}
Furthermore, if we have an \textbf{infinite pendant path}, then the only way $x_i$ remains bounded is if $\al=0$, in which case $x_i=r(\la)^{-i}$ for every $i \geq 0$.

According to Lemma~\ref{lem:inside_weights}, if an induced subgraph $H$ of $G$ has only one vertex $v$ with outside neighbors, then for the Perron vector $\x$ of $G$ we have 
\[ x_u = B_{u,v}(\la_G) \]
provided that we have $y_v=1$ for the sum of the weights of the outside neighbors of $v$. For instance, in the setting of pendant paths (i.e., $G=H+_v P$), $v$ has one outside neighbor and we always assume that its weight is $1$. We will use the following form of the observations above. 
\begin{Prop} \label{prop:HB}
Suppose that $H$ is a connected graph and let $B=(\la I - H)^{-1}$. For a fixed $v \in V(H)$ and for any $\la>\la_H$ the vector $x_u:=B_{u,v}(\la)$, $u\in V(H)$, satisfies the eigenvalue equation with $\la$ at each vertex of $H$ except for $v$, where we have 
\[ \la x_v = 1 +\sum_{u \in N_H(v)} x_u .\]
\end{Prop}

\subsection{Infinite pendant path}
For an infinite graph $G$, the adjacency operator $A_G$ can be considered as a self-adjoint $\ell^2\big( V(G) \big) \to \ell^2\big( V(G) \big)$ operator. If $G$ has bounded degrees, then $A_G$ is a bounded operator and $\| A_G \| = \sup \, \sigma(A_G)$, where $\sigma(A_G) \subset \mathbb{R}$ is the spectrum of $A_G$. Next we show that for $G = H +_v P_\infty$ with $\la_H\geq 2$, $\| A_G \|$ actually lies in the point spectrum (i.e., it is an eigenvalue) and we describe the corresponding eigenvector. In these cases it is therefore justified to use the notation $\la_G := \| A_G \|$ for this top eigenvalue.
\begin{Lemma} \label{lem:la_infty}
Let $H$ be a connected finite graph with $\la_H \geq 2$, and $v \in V(H)$ one of its vertices.\footnote{If $\la_H < 2$, then one may replace $H$ with $H+_v P_\ell$. Unless $H \cong P_r$ or $H \cong D_r$,  $\la_{H+_v P_\ell}$ becomes larger than $2$ for some $\ell$ (see Section \ref{sec:lambda_at_most_2}), and the lemma may be applied to that graph instead.} Let $B=B_H(\la)$ be the matrix defined in Section~\ref{sec:resolvent}.

Consider $H +_v P_\infty$ and denote the vertices of the pendant path $P_\infty$ by $w_i$, $i \geq 0$. Let $r(\la)$ be the function defined in \eqref{eq:q}. Viewing the diagonal entry $B_{v,v}(\la)$ also as a function of $\la > \la_H$, let $\la_{\infty}$ be the (unique) solution of the equation 
\begin{equation} \label{eq:la_inf_eq}
 B_{v,v}(\la) = r(\la) .
\end{equation}
Define the vector $\x$ over the vertices as follows:
\begin{equation} \label{eq:x_inf}
x_u := 
\begin{cases}
B_{u,v}(\la_{\infty}) & \text{ if } u \in V(H); \\
r(\la_{\infty})^{-i}  & \text{ if } u = w_i. 
\end{cases}
\end{equation}
Then $\x \in \ell^1 \subset \ell^2$, and $\x$ is an eigenvector of $H +_v P_\infty$ with eigenvalue $\la_{\infty}$. Since each $x_u$ is positive, $\la_{\infty}$ must be the top of the spectrum, and we may denote it by $\la_{H +_v P_{\infty}}$.
\end{Lemma}
\begin{proof}
By Proposition~\ref{prop:B_properties}, $B_{v,v}(\la)$ is a strictly monotone decreasing continuous function on $(\la_H,\infty)$ with   
\[ \lim_{\la \searrow \la_H} B_{v,v}(\la) = \infty 
\text{ and }
\lim_{\la \to \infty} B_{v,v}(\la) = 0 .\]
Since $r$ is a strictly monotone increasing continuous function on the same interval with $r(\la_H)\geq r(2) = 1$, it follows that the equation \eqref{eq:la_inf_eq} has a unique solution that we denote by $\la_{\infty} > \la_H \geq 2$. Since $r(\la_{\infty})>r(2)=1$, the vector $\x$ defined in \eqref{eq:x_inf} decays exponentially along the pendant path, and hence it is indeed in $\ell^1$. By design it satisfies the eigenvalue equation at each vertex other than $w_0$. (See Section~\ref{sec:ev_on_path}, and in particular Proposition~\ref{prop:HB}.) As for $w_0$, the eigenvalue equation holds precisely because $\la_{\infty}$ is a root of \eqref{eq:la_inf_eq}.
\end{proof}

\subsection{Eigenvalue sandwich}
\begin{Def}
Consider the path graph $P_k$ with vertices $w_0,w_1,\ldots,w_{k-1}$. Add $\ell \geq 2$ new vertices, each connected to $w_{k-1}$. We will denote the resulting ``fork'' graph by $F_{k,\ell}$. For instance, $F_{5,2} = \, \PPPPPFF$. We will use the notation $H +_v F_{k,\ell}$ when $H$ and $F_{k,\ell}$ are joined by a single edge $e=(v,w_0)$. 
\end{Def}
Next we prove that $\la_{H +_v P_\infty}$ is always ``sandwiched'' between $\la_{H +_v P_k}$ and $\la_{H +_v F_{k,2}}$.
\begin{Lemma} \label{lem:lambda_sandwich}
We have
\[ \la_{H +_v P_k} < \la_{H +_v P_\infty} < \la_{H +_v F_{k,2}} .\]
\end{Lemma}
\begin{proof}
Let $f:(\la_H,\infty) \to (0,\infty)$ be a monotone decreasing continuous function and consider the equation $B_{v,v}(\la) f(\la)=1$. 
Note that by Proposition \ref{prop:B_properties} 
\[ \lim_{\la \searrow \la_H} B_{v,v}(\la) f(\la) = \infty 
\text{ and }
\lim_{\la \to \infty} B_{v,v}(\la) f(\la) = 0 .\]
Since $B_{v,v}f$ is strictly monotone decreasing and continuous, it follows that there is a unique solution of the equation that we will denote by $\la_f$. 

Note that if $f(\la) < g(\la)$ for all $\la > \la_H$, then we clearly have $\la_f < \la_g$.

Furthermore, we have seen in Lemma~\ref{lem:la_infty} that for $f=1/r$ we get $\la_{1/r}=\la_{H +_v P_\infty}$. 

Two other functions that correspond to eigenvalues of certain graphs are $f(\la):=1/\la$ and $g(\la):=2/\la$. Indeed, if $H'$ (and $H''$) denote the graphs obtained from $H$ by adding a leaf (two leaves, respectively) at $v$, then it is easy to see that 
\[ \la_f = \la_{H'} 
\quad \text{and} \quad 
\la_g = \la_{H''} .\]
By \eqref{eq:q_bounds} we have the following inequalities for each $\la > 2$:
\[ f(\la) < 1 \big/ r(\la) < g(\la) .\]
It follows that 
\[ \la_f < \la_{1/r} < \la_g, \]
that is, 
\[ \la_{H'} < \la_{H +_v P_\infty} < \la_{H''}. \]
If we replace $H$ with $H +_v P_k$ and $v$ with the other endpoint of the pendant path $P_k$, then the previous inequality turns into  
\[ \la_{H +_v P_{k+1}} < \la_{H +_v P_{\infty}} < \la_{H +_v F_{k,2}} ,\]
and the proof is complete.
\end{proof}

\subsection{Complete graphs and stars} \label{sec:KS}
Next we compute $\la_G$ and $\Ga_G$ for $G=K_p+P_\infty$ and $G=S_p+P_\infty$.

For the complete graph $H=K_p$ the matrix $B$ has the following entries:
\[ B_{u,v}(\la) = 
\begin{cases}
\frac{\la-(p-2)}{(\la+1)\big(\la-(p-1)\big)} & \text{if } u=v;\\
\frac{1}{(\la+1)\big(\la-(p-1)\big)} & \text{if } u \neq v.
\end{cases}\]
According to Lemma~\ref{lem:la_infty}, by solving the equation $B_{v,v}(\la)=r(\la)$ we get that   
\[ \la_\infty := \la_{K_p+P_\infty} 
= \frac{p-3}{2}+\frac{p-1}{2(p-2)}\sqrt{p^2-4}. \]
Then 
\[ r\big( \la_{K_p+P_\infty} \big) =\frac{p-2+\sqrt{p^2-4}}{2} ,\]
and hence
\[ \Gamma_{K_p+P_{\infty}}=\frac{(p-1)(2p-3)}{2(2p-5)}+\frac{(p-1)(2p+1)}{2(p+2)(2p-5)}\sqrt{p^2-4}. \]
For $p=4$ we get 

\[ \la_{K_4+P_{\infty}}=\frac{1+3\sqrt{3}}{2}\approx 3.098 
\quad \text{and} \quad
\beta_{\star}=\Ga_{K_4+P_{\infty}}=\frac{5+3\sqrt{3}}{2}\approx 5.098.\]

For the star graph $H=S_p$, let $v$ denote the central vertex. Then the relevant entries of $B$ are:
\[ B_{u,v}(\la) = 
\begin{cases}
\frac{\la}{\la^2-(p-1)} & \text{if } u=v;\\
\frac{1}{\la^2-(p-1)} & \text{if } u \neq v.
\end{cases}\]
Again, we need to solve the equation $B_{v,v}(\la)=r(\la)$. This time it yields
\[ \la_{S_p+P_{\infty}}=\frac{p-1}{\sqrt{p-2}}.\]
Then 
\[ r\big( \la_{K_p+P_\infty} \big)=\sqrt{p-2} ,\]
consequently 
\[ \Gamma_{S_p+P_{\infty}}=\frac{p-1}{2(p-3)}(\sqrt{p-2}+1)^2. \]
In particular, for $p=5$ we get
\[ \la_{S_5+P_{\infty}}= 4 \big/ \sqrt{3} \approx 2.3094
\quad \text{and} \quad 
\beta_{\mathrm{tr}}=\Gamma_{S_5+P_{\infty}}=(\sqrt{3}+1)^2=4+2\sqrt{3}\approx 7.464. \]

Let us mention that for $p=6$ we get $\Gamma_{S_6+P_{\infty}}=7.5$, which is very close to $\beta_{\mathrm{tr}}$. In fact, $S_6+P_7$ has the second smallest $\Gamma_T$ among trees on $13$ vertices. As $n$ grows, more and more graphs $S_5+G'$ (i.e., $S_5$ with various tails) ``beat'' $S_6+P_{n-6}$ among $n$-vertex graphs.

\subsection{Finite versus infinite path} 
In this section we give conditions that ensure 
\[ \Gamma_{H+_v P_k}<\Gamma_{H+_v P_{\infty}} .\]
We are not aiming for the best (i.e., weakest) possible conditions. Our goal is to keep the proof as light as possible on the computational side, while making sure that the result can be applied to the relevant cases $H=K_4$ and $H=S_5$, so that we can conclude that 
\[ \Ga_{K_4+ P_k}<\bestar
\quad \text{and} \quad 
\Ga_{S_5+ P_k}<\betr \]
for every $k \geq 1$. 
\begin{Lemma} \label{lem:Gamma-lower}
Fix a connected graph $H$ and a vertex $v \in V(H)$. Suppose that $\la_H \geq 2$ and let $B$ be the matrix corresponding to $H$ (see Section~\ref{sec:resolvent}). We introduce the following functions of $\la \in (\la_H, \infty)$:
\begin{align} 
S(\la) &:= \sum_{u \in V(H)} B_{u,v}(\la); \nonumber \\
T(\la) &:= \sum_{u \in V(H)} \big( B_{u,v}(\la) \big)^2; \nonumber \\
J(\la) &:= \frac{\bigg( S(\la) + \frac{1}{1-r(\la)^{-1}} \bigg)^2}{T(\la)+\frac{1}{1-r(\la)^{-2}}}. \label{eq:J} 
\end{align}
Suppose that the following conditions hold true for an integer $k_0 \geq 1$:
\begin{enumerate}[(i)]
\item $J(\la)$ is monotone increasing on the interval $(\la_{H+_v P_{k_0}},\la_{H+_v P_{\infty}})$,\\
or the weaker assumption: $J(\la) \leq J(\la_{H+_v P_{\infty}})$ on this interval;
\item $\displaystyle f(\la) := \bigg(T(\la)+1\bigg) r(\la) - S(\la) - \frac{r(\la)}{r(\la)-1} \geq 0$ on the same interval.
\end{enumerate}
Then for any $k\geq k_0$ we have 
\[ \Gamma_{H+_v P_k} 
< J\big( \la_{H+_v P_k} \big) 
\leq J\big( \la_{H+_v P_\infty} \big) 
= \Gamma_{H+_v P_{\infty}} .\]
\end{Lemma}
\begin{proof}
For brevity let 
\[ \la_k:= \la_{H+_v P_k}; \quad 
\la_\infty := \la_{H+_v P_\infty} \]
and 
\[ r_k:=r(\la_k)= r\big( \la_{H+_v P_k} \big); \quad 
r_\infty:=r(\la_\infty) = r\big( \la_{H+_v P_\infty} \big). \]
We have 
\[ 2\leq \la_H < \la_{k_0} \leq \la_k<\la_\infty ,\]
and hence
\[ 1< r_{k_0} \leq r_k < r_\infty .\]
The idea is to define a kind of interpolation between the Perron vectors of $H+_v P_k$ and $H+_v P_\infty$. Let $\al=\frac{-r_k^{-k}}{r_k^k-r_k^{-k}}<0$ and consider the following vectors $\x, \y, \z$, each indexed by the set $V(H) \cup \{w_0,w_1,\ldots\}$:
\begin{align*}
x_u&= \begin{cases}
B_{u,v}(\la_k) & \text{if } u \in V(H);\\
r_k^{-i} + \al(r_k^i-r_k^{-i}) & 
\text{if } u=w_i \text{ for } i=0,1,\ldots,k-1;\\
0 & \text{if } u=w_i \text{ for } i \geq k;\\
\end{cases}\\
y_u&= \begin{cases}
B_{u,v}(\la_k) & \text{if } u \in V(H);\\
r_k^{-i} & \text{if } u=w_i \text{ for } i \geq 0;\\
\end{cases}\\
z_u&= \begin{cases}
B_{u,v}(\la_\infty) & \text{if } u \in V(H);\\
r_\infty^{-i} & \text{if } u=w_i \text{ for } i \geq 0 .\\
\end{cases}
\end{align*}
According to Lemma~\ref{lem:la_infty}, $\z$ is the Perron vector of $H +_v P_\infty$, and hence 
\[ \Ga(\z) = \Ga_{H +_v P_\infty} .\]
Furthermore, $J$ was defined in such a way that 
\[ \Ga(\y) = J(\la_k) 
\quad \text{and} \quad 
\Ga(\z) = J(\la_\infty) .\]
Using $\la_k<\la_\infty$, condition (i) implies that 
\[ J(\la_k) \leq J(\la_\infty)=\Ga(\z)=\Ga_{H+_v P_\infty} .\]
Finally, $\x$ corresponds to the Perron vector of $H +_v P_k$ (see Section~\ref{sec:ev_on_path}, and in particular Proposition~\ref{prop:HB}), and hence 
\[ \Ga(\x) = \Ga_{H +_v P_k} .\]
So it remains to show that $\Ga(\x) < \Ga(y) = J(\la_k)$. We notice that $\x$ and $\y$ coincide over $V(H) \cup \{w_0\}$, while $x_u<y_u$ for any other vertex due to the fact that $\al<0$. We use the perturbation lemma (Lemma~\ref{perturbation}): we start from $\y$, then, step by step, we replace $y_u$ with the smaller $x_u$ at each $u=w_i$, $i \geq 1$. We need that, for any intermediate vector $\y'$ through this process, it holds that 
\[ \frac{\|\y'\|_2^2}{\|\y'\|_1}> y_{w_i} = r_k^{-i}. \]
By trivial estimates and condition (ii) we get
\[ \frac{\|\y'\|_2^2}{\|\y'\|_1} 
> \frac{\left\|\x \big|_{V(H)\cup\{w_0\}} \right\|_2^2}{\|\y\|_1}
= \frac{T(\la_k)+1}{S(\la_k)+ \frac{1}{1-r_k^{-1}}} 
\geq r_k^{-1} \geq r_k^{-i} .\] 
\end{proof}
\begin{Rem} \label{rem:var_t}
In the previous sections, every condition that needed to be verified involved only rational functions. Although $r(\la)$ is not a rational function, we can use $t:=r(\la)$ as a variable instead in order to make every expression a rational function again. Indeed, recall that $r(\la)+1/r(\la)=\la$. Therefore, if we make the substitutions $r(\la) \rightsquigarrow t$ and hence $\la \rightsquigarrow t+1/t$, then everything becomes a rational function of $t$:
\begin{align} 
\Jh(t) &:= \frac{\bigg( S(t+1/t) + \frac{1}{1-t^{-1}} \bigg)^2}{T(t+1/t)+\frac{1}{1-t^{-2}}} ; \label{eq:Jt} \\
\hat{f}(t) &:= \bigg(T(t+1/t)+1\bigg) t - S(t+1/t) - \frac{t}{t-1} . \nonumber
\end{align}
Then $\Jh(t)$ and $\hat{f}(t)$ are rational functions such that $\Jh\big( r(\la) \big)=J(\la)$ and $\hat{f}\big( r(\la) \big)=f(\la)$. Verifying that a particular rational function is monotone or positive on an interval is a straightforward task with a computer.
\end{Rem}
\begin{Thm} \label{thm:indeed_below}
For every $k\geq 1$ we have 
\[ \Gamma_{K_4+P_k}<\bestar=\Gamma_{K_4+P_{\infty}}
\quad \text{and} \quad 
\Gamma_{S_5+P_k}<\betr=\Gamma_{S_5+P_{\infty}} .\]
\end{Thm}
\begin{proof}
We apply Lemma~\ref{lem:Gamma-lower} with $k_0=1$ and $H$ being $K_4$ or $S_5$. 

Case $H=K_4$. Based on Section~\ref{sec:KS} we can easily compute that 
\[ S(\la)=\frac{1}{\la-3} 
\quad \text{and} \quad 
T(\la)= \frac{\la^{2} - 4\la + 7}{\la^{4} - 4\la^{3} - 2\la^{2} + 12\la + 9}  .\]
Furthermore,   
\[ \Jh(t) = \frac{t^7 - t^6 - 3t^5 - 3t^4 + 3t^3 + 9t^2 + 8t + 4}{t^7 - 5t^6 + 7t^5 - 7t^4 + 23t^3 - 19t^2 - 6t + 6}. \]
Condition (i) of the lemma requires this rational function to be monotone increasing on the interval $\big( r(\la_{K_4+P_1}) , r(\la_{K_4+P_\infty}) \big)$. Therefore, its derivative (also a rational function) is required to be nonnegative on the given interval. This can be verified easily by a computer, along with condition (ii), see our Jupyter notebook \cite{JUP}.

Case $H=S_5$. Again, based on Section~\ref{sec:KS}: 
\[ S(\la)=\frac{\la + 4}{\la^{2} - 4}
\quad \text{and} \quad 
T(\la)= \frac{\la^{2} + 4}{\la^{4} - 8\la^{2} + 16}  .\]
For $\Jh$ we get  
\[ \Jh(t) = \frac{t^{4} + 4t^{3} + 10t^{2} + 12t + 9}{t^{4} - 2t^{2} + 9} , \]
which we checked to be monotone increasing on the interval $\big( r(\la_{S_5+P_1}) , r(\la_{S_5+P_\infty}) \big)$. Condition (ii) is also satisfied, see \cite{JUP} for details.
\end{proof}

\section{Estimates for branching tails}
\label{sec:tail}

In this section we consider the situation when $G$ contains $H+_v P_{k}$ as an induced subgraph in such a way that only the last vertex of the path may have further neighbors in $G$. Our goal is to find conditions, under which $\Gamma_G \leq \Gamma_{H+P_{\infty}}$ implies that the last vertex has exactly one further neighbor. In other words, if the ``tail'' branches at the $k$-th vertex (for the first time), then $\Gamma_G$ must exceed the limiting ratio $\Gamma_{H+P_{\infty}}$.
\begin{Thm} \label{thm:Gamma-upper}
Let $H$ be a fixed connected graph with $o,v \in V(H)$ not necessarily distinct vertices. Suppose that $\la_H \geq 2$ and let $B=(\la I -A_H)^{-1}$. We define the functions $S(\la),T(\la),J(\la)$ as in Lemma~\ref{lem:Gamma-lower}. We also need the following variants: 
\[ \Sh(t)=S(t+1/t); \quad
\Th(t)=T(t+1/t); \quad 
\Jh(t)=J(t+1/t) .\]
Let $\la_\infty := \la_{H+_v P_{\infty}}$ and $\be_\infty := \Gamma_{H+_v P_{\infty}}$ Suppose that the following conditions hold true for some $\la_\infty<\la'<\la''$, $c>1$, and integer $k \geq 2$.
\begin{enumerate}[(i)]
\item $2\la_\infty+3>\be_\infty$. 
\item $B_{o,v}(\la'') \leq 1$.
\item $B_{v,v}(\la') \geq 1 \big/ r(\la_\infty)$.
\item $\big( S(\la) + 1 \big)^2 \big/ \big( T(\la) + 1 \big) > \be_\infty$ on $[\la',\la'']$.
\item $J(\la)$ is monotone increasing on $(\la_\infty,\la')$,\\
or the weaker assumption: $J(\la) > J(\la_\infty) = \be_\infty$ on $(\la_\infty,\la')$.
\item $\displaystyle \frac{\big(\Sh(t)+1+c\big)^2}{\Th(t)+1+c^2} > \be_\infty$ for $t \in (r_\infty,r')$, where $r_\infty:=r(\la_\infty)$ and $r':=r(\la')$.
\item $\displaystyle \frac{\big(\Sh(t)+1+ct\big)^2}{\Th(t)+1+c^2t^2} > \be_\infty$ for $t \in (r_\infty,r')$.
\item $\displaystyle \frac{\frac{2}{\be_\infty}\left(\Sh(t)+\frac{t}{t-1}\right) \left(1-\frac{t+1}{t-1}t^{-k} \right)-2k t^{-k}}{ \frac{t^{3}}{t^2-1}} \geq c$ for $t \in (r_\infty,r')$.
%
%
\end{enumerate}
\medskip

Then 
\[ \Gamma_G > \be_\infty=\Gamma_{H+_v P_{\infty}} \]
holds for any connected graph $G$ that contains $H+_v P_{k}$ as a proper induced subgraph in such a way that $o$ is the master vertex, the endpoint of $P_k$ has at least two neighbors outside $H+_v P_{k}$, and no other vertex of $H+_v P_{k}$ has outside neighbors.
\end{Thm}
\begin{Rem} \label{rem:cond8}
Note that if $\la_\infty \geq 2+ \frac{1}{k_0(k_0+1)}$, then it is easy to see that for any fixed $t > r_\infty \geq 1+\frac{1}{k_0}$, the sequence $k t^{-k}$, $k \geq k_0$, is monotone decreasing, and hence the left-hand side of (viii) is increasing in $k \geq k_0$. So condition (viii) gets weaker as $k$ grows: if it holds for $k_0$, then it is true for any $k \geq k_0$. 
\end{Rem}

We will apply this theorem to $H=K_4 +_o P_1$ and $H=S_5 +_o P_4$ in Section~\ref{sec:tail_K4_S5}.



\begin{proof}
Let $G$ be as in the statement of the theorem containing $H +_v P_k$ as an induced subgraph. As before, the vertices of $P_k$ are denoted by $w_0,w_1,\dots ,w_{k-1}$. Let $w_{k},\ldots ,w_{k+s-1}$ be the neighbors of $w_{k-1}$ outside $H+_v P_k$. By our assumption, we have $s \geq 2$. Furthermore, let $w_{-1}:=v$.

In the rest of the proof we will use $\la$ for $\la_G$ and $t$ for $r(\la)$, where $r$ is the function defined in Section~\ref{sec:ev_on_path}: 
\[ \la := \la_G 
\quad \text{and} \quad 
t := r(\la_G) = r(\la) = \frac{\la+\sqrt{\la^2-4}}{2}>1 .\]
Note that $\la_{\infty}< \la_{H +_v F_{k,2} } \leq \la_G = \la$ by Lemma~\ref{lem:lambda_sandwich}.  

Let $\x$ be the Perron vector of $G$. For better readability, we will often write $x_i$ for $x_{w_i}$. Let us normalize $\x$ in such a way that $x_0=x_{w_0}=1$. Then, by Proposition~\ref{prop:HB} we have
\begin{equation} \label{eq:xu}
x_u = B_{u,v}(\la) \quad \text{for } u\in V(H), 
\end{equation}
and hence
\begin{equation} \label{eq:STJ}
\sum_{u \in V(H)} x_u = S(\la)=\Sh(t); \quad 
\sum_{u \in V(H)} x_u^2 = T(\la)=\Th(t) .
\end{equation}

First we show that $\la \leq \la''$. Otherwise we can use \eqref{eq:xu} for $u=o$ and the fact that $B_{o,v}$ is strictly monotone decreasing as well as condition (ii) to conclude that 
\[ x_o = B_{o,v}(\la) < B_{o,v}(\la'') \leq 1= x_{w_0} ,\] 
contradicting that $o$ is the master vertex of $G$. 

If $\la \in [\la',\la'']$, then consider the induced subgraph $H' = G\big[ V(H) \cup \{w_0\} \big]$. By \eqref{eq:STJ} and condition (iv) we have 
\[ \Gamma\big( \x\big|_{H'} \big) = \frac{\big( S(\la) + 1 \big)^2}{T(\la) + 1} > \be_\infty .\]
Whenever we have a subgraph $H'$ containing $N_G[o]$ with $\Gamma\big( \x_{H'} \big) > \be_\infty$, we can use Lemma~\ref{subgraph-lower-bound}(b) and condition (i) to conclude that $\Ga_G > \be_\infty$ holds as well, and we are done. (We will use this observation on a number of other occasions in this proof.)

Therefore, \textbf{we may assume that} $\la_\infty< \la < \la'$. Recall that 
\[ r_\infty=r(\la_\infty); \quad 
t=r(\la); \quad 
r'=r(\la'), \quad \text{and hence } 1<r_\infty<t < r' .\]
We also have 
\[ t+\frac{1}{t} = \la 
\quad \text{and} \quad 
r_\infty+\frac{1}{r_\infty} = \la_\infty .\]
Moreover, by Lemma~\ref{lem:la_infty}, $\la_{\infty}$ is the solution of the equation 
\[ B_{v,v}(\la_\infty)=r_{\infty}.\]
Also recall (from Section~\ref{sec:ev_on_path}) that there exists an $\al$ in such a way that
\begin{equation} \label{eq:xj}
x_j=\al t^j+(1-\al)t^{-j} 
\quad \text{for } j=-1,0,1,\dots ,k-1 .
\end{equation}

Next we show that $0<\alpha<1$. Using condition (iii) and the fact that $r$ is monotone increasing, while $B_{v,v}$ is strictly monotone decreasing by Proposition~\ref{prop:B_properties}, we get  
\[ t> r_\infty = B_{v,v}(\la_\infty) > B_{v,v}(\la)
> B_{v,v}(\la') \geq \frac{1}{r(\la_\infty)} > \frac{1}{r(\la)} = \frac{1}{t} .\]
Furthermore, using \eqref{eq:xj} for the vertex $v = w_{-1}$: 
\[ B_{v,v}(\la)=x_v=x_{-1}=\al \frac{1}{t}+(1-\al)t .\]
Thus we have $t>\al \frac{1}{t}+(1-\al)t > \frac{1}{t}$, and $0<\alpha<1$ clearly follows.


Next observe that
\[\sum_{j=k}^{k+s-1}x_j=\la x_{k-1}-x_{k-2}=\alpha t^{k}+(1-\alpha)t^{-k}.\]
Note that $x_j\geq \frac{x_{k-1}}{\la}$ for $j=k,\dots ,k+s-1$, so
\[\alpha t^{k}+(1-\alpha)t^{-k}=\sum_{j=k}^{k+s-1}x_j\geq x_k + x_{k+1} \geq \frac{2}{\la}\bigg( \alpha t^{k-1}+(1-\alpha)t^{1-k} \bigg).\]
Thus
\[ \al t^{k} \left(1-\frac{2}{\la t} \right) 
\geq (1-\al) t^{-k} \left(\frac{2t}{\la} -1 \right) .\]
Recall that $t+1/t=\la$, meaning that $1-\frac{2}{\la t} = \frac{2t}{\la} - 1 > 0$, and we conclude that  
\begin{equation}\label{alpha-bound}
\alpha t^k\geq (1-\alpha)t^{-k}.
\end{equation}

Let $\overline{H}=G\big[ V(H)\cup \{w_0,\ldots, w_{k+s-1}\} \big]$. Using \eqref{eq:STJ}, \eqref{eq:xj}, and then \eqref{alpha-bound}:
\begin{align*}
\left\| \x\big|_{\overline{H}} \right\|_1
&= \sum_{u \in V(H)} x_u + \sum_{j=0}^{k} \big( \alpha t^j+(1-\alpha)t^{-j} \big) \\
&=\Sh(t)+\alpha \frac{t^{k+1}-1}{t-1}+(1-\alpha)\frac{1-t^{-k-1}}{1-t^{-1}} \\
&=\Sh(t)+\frac{t}{t-1}-\alpha \frac{t+1}{t-1}+\frac{1}{t-1}\left(\alpha t^{k+1}-(1-\alpha)t^{-k}\right) \\
&\geq \Sh(t)+\frac{t}{t-1}-\alpha \frac{t+1}{t-1}+\frac{1}{t-1}\left(\alpha t^{k+1}-\alpha t^{k}\right) \\
&= \Sh(t)+\frac{t}{t-1} + \al \left(t^k - \frac{t+1}{t-1} \right) .
\end{align*}
The trivial estimate gives 
\[ \sum_{j=k}^{k+s-1}x_j^2 
\leq \left(\sum_{j=k}^{k+s-1}x_j\right)^2 = 
\big( \alpha t^k+(1-\alpha)t^{-k} \big)^2 .\]
Using this and \eqref{eq:STJ} we get the following for the $\ell^2$ norm:
\begin{align*}
\left\| \x\big|_{\overline{H}} \right\|_2^2
&\leq \sum_{u \in V(H)} x_u^2 + \sum_{j=0}^{k} \big( \alpha t^j+(1-\alpha)t^{-j} \big)^2 \\
&=\Th(t)+\alpha^2\frac{t^{2k+2}-1}{t^2-1}+2(k+1)\alpha(1-\alpha)+(1-\alpha)^2\frac{1-t^{-2k-2}}{1-t^{-2}} \\
&\leq \Th(t)+\alpha^2\frac{t^{2k+2}-1}{t^2-1}+2(k+1)\alpha(1-\alpha)+(1-\alpha)^2\frac{t^2}{t^{2}-1} \\
&= \Th(t)+\frac{t^2}{t^{2}-1}+\alpha^2 \left( \frac{t^{2k+2}}{t^2-1} - 2k-1 \right) + \al \left( 2k+2-\frac{2t^2}{t^{2}-1} \right) \\&\leq \Th(t)+\frac{t^2}{t^{2}-1}+\alpha^2 \, \frac{t^{2k+2}}{t^2-1} + \al \left( 2k -\frac{2}{t^{2}-1} \right) \\
&\leq \Th(t)+\frac{t^2}{t^{2}-1}+ \alpha \left( x_{k-1} \, \frac{t^{k+3}}{t^2-1} + 2k \right) ,
\end{align*}
where we used the simple bound $\al t^{k-1} \leq x_{k-1}$ at the last inequality.

Let us introduce the following notations:
\begin{align*}
Z_{1,\m} &=\Sh(t)+\frac{t}{t-1}; \quad 
Z_{1,\e} =\alpha \left(t^{k}-\frac{t+1}{t-1}\right); \\
Z_{2,\m} &=\Th(t)+\frac{t^2}{t^2-1}; \quad
Z_{2,\e} =\alpha \left( x_{k-1} \, \frac{t^{k+3}}{t^2-1} + 2k \right) .
\end{align*}
Then our estimates on the $\ell^1$ and $\ell^2$ norms imply that 
\[ \Gamma\left(\x\big|_{\overline{H}}\right)
\geq \frac{\left(Z_{1,\m}+Z_{1,\e} \right)^2}{Z_{2,\m}+Z_{2,\e}}. \]
Our goal is to show that  
$\Gamma\left(\x\big|_{\overline{H}}\right)>\be_{\infty}$ by proving the following: 
\[ (Z_{1,\m}+Z_{1,\e})^2-\be_{\infty}(Z_{2,\m}+Z_{2,\e}) > 0 .\]
After expanding the left-hand side we get 
\[ (Z_{1,\m}^2-\be_{\infty}Z_{2,\m})
+(2Z_{1,\m} Z_{1,\e}-\be_{\infty}Z_{2,\e})
+Z_{1,\e}^2 .\]
The third term is clearly nonnegative. So is the first term, since, by condition (v), we have 
\[\frac{Z_{1,\m}^2}{Z_{2,\m}}=\Jh(t)=J(\la) \geq 
\be_{\infty} .\]
Therefore, if the second term is positive, then $\Gamma\left(\x\big|_{\overline{H}}\right)>\be_{\infty}$, implying $\Ga_G>\be_{\infty}$, and we are done. So from this point on we will assume that $2Z_{1,\m}Z_{1,\e}-\be_{\infty}Z_{2,\e} \leq 0$, that is: 
\[ 2\left(\Sh(t)+\frac{t}{t-1}\right)\alpha \left(t^{k}-\frac{t+1}{t-1}\right) \\ 
\leq \be_{\infty}\alpha \left( x_{k-1} \, \frac{t^{k+3}}{t^2-1} + 2k \right) .\]
We divide both sides by $\al t^{k}$ and rearrange to get the following lower bound on $x_{k-1}$:
\[ x_{k-1} \geq 
\frac{\frac{2}{\be_\infty}\left(\Sh(t)+\frac{t}{t-1}\right) \left(1-\frac{t+1}{t-1}t^{-k} \right)-2k t^{-k}}{ \frac{t^{3}}{t^2-1}} \]
By condition (viii) we conclude that $x_{k-1} \geq c$. 

Since $x_0=1$, $x_{k-1} \geq c>1$ and $x_j \leq t x_{j-1}$ for every $1 \leq j \leq k-1$, it follows that $c \leq x_j \leq ct$ must hold for some $1 \leq j \leq k-1$. Consider $H'=G\big[ V(H) \cup \{w_0,w_j\} \big]$. We claim that $\Ga\big( \x \big|_{H'} \big) > \be_\infty$. Consider the function 
\[ f(y) := \frac{\big(\Sh(t)+1+y\big)^2}{\Th(t)+1+y^2} .\]
On the one hand, $\Ga\big( \x \big|_{H'} \big) = f(x_j)$. On the other hand, it is easy to see that $I := \{ y \, : \, f(y) > \be_\infty \}$ is an open interval. From conditions (vi) and (vii) we know that $c \in I$ and $ct \in I$, and hence $x_j \in [c,ct] \subseteq I$, meaning that $\Ga\big( \x \big|_{H'} \big) = f(x_j)>\be_\infty$. Since $H'$ contains $N_G[o]$, it follows that $\Ga_G > \be_\infty$, and the proof is finally complete.
\end{proof}

\subsection{Checking conditions} \label{sec:tail_K4_S5}
Next we show how to check the conditions of Theorem~\ref{thm:Gamma-upper} in the special cases $H=K_4+P_1$ and $H=S_5+P_4$, which will be needed for the proofs of Theorem~\ref{thm:only_K4Pk} and Theorem~\ref{thm:only_S5Pk}, respectively.

\begin{Thm} \label{K4_tail}
Let $\ell \geq 3$. Suppose that a connected graph $G$ contains $H_\ell=K_4 +_o P_\ell$ as an induced subgraph and that $o$ is the master vertex of $G$. If the leaf of $H_\ell$ has at least two neighbors outside $H_\ell$ and no other vertex in $H_\ell$ has an outside neighbor, then $\Ga_G > \bestar$.
\end{Thm}
\begin{proof} We apply Theorem~\ref{thm:Gamma-upper} to $H:=H_1=K_4 +_o P_1$ and $k \geq 2$. We need to check the conditions of the theorem. Let $v$ be the leaf of $H$ (i.e., the single vertex of $P_1$):

\begin{center}
\figKKKKP
\end{center}

Then $H +_v P_\infty \cong K_4 + P_\infty$, so we have 
\[ \la_\infty = \frac{3\sqrt{3}+1}{2} \approx 3. 098
\quad \text{and} \quad 
\be_\infty = \bestar \approx 5.098 \]
clearly satisfying condition (i). Furthermore, $r_{\infty}=1+\sqrt{3}$. The characteristic polynomial of $A_H$ is 
\[ P(\la) = (\la+1)^2(\la^3-2\la^2-4\la+2) ,\]
and we get the following for the column vector of $B$ corresponding to $v$, that is, the vector consisting of the entries $B_{u,v}(\la)$, $u \in V(H)$:
\[ \frac{1}{P(\la)}\bigg((\la+1)^2,(\la+1)^2,(\la+1)^2,(\la-2)(\la+1)^2, (\la-3)(\la+1)^3 \bigg) ,\]
where the first three entries correspond to the degree-$3$ vertices of $H$, while the fourth and fifth entries correspond to $o$ and $v$, respectively. 

From here $S,T, \Sh, \Th, \Jh$ can be easily computed. Conditions (i)--(iii) are simple inequalities, while each of the remaining conditions (iv)--(viii) is an inequality on an interval involving rational functions and hence can be checked quickly with a computer.

We can set $\la''=3.18$ so that condition (ii) is just satisfied. Then we set $\la'=3.11$ so that condition (iv) holds on $[\la',\la'']$. Then we check condition (iii) and verify that (v) holds on $(\la_\infty,\la')$.

Next we choose $c=1$ and verify (vi), (vii) on the interval $(r_\infty,r')$. Finally, we need to check condition (viii) for any $k \geq 2$. According to Remark~\ref{rem:cond8}, it suffices to check it for $k=2$, which is again a straightforward task with a computer.

See our Jupyter notebook \cite{JUP} for further details.
\end{proof}

Next we consider the other relevant special case of Theorem~\ref{thm:Gamma-upper}.
\begin{Thm} \label{S5_tail}
Let $\ell \geq 8$. Suppose that a connected graph $G$ contains $H_\ell=S_5 +_o P_\ell$ as an induced subgraph, where the central vertex $o$ of the star coincides with the master vertex of $G$. If the leaf of $H_\ell$ (non-adjacent to $o$) has at least two neighbors outside $H_\ell$ and no other vertex in $H_\ell$ has an outside neighbor, then $\Ga_G \geq \betr$.
\end{Thm}

\begin{proof} This time we apply Theorem~\ref{thm:Gamma-upper} to $H:=H_4=S_5 +_o P_4$ with $k \geq 4$, $o$ being the central vertex of $S_5$ and $v$ being the last vertex of $P_4$:

\begin{center}
\figSP
\end{center}

In this case $\la_{\infty}=\sqrt{3}+1 \big/ {\sqrt{3}}$ and $r_{\infty}=\sqrt{3}$, while the characteristic polynomial is $P(\la)=\la^9-8\la^7+15\la^5-4\la^3$. For the column vector of $v$ we get 
\[\frac{1}{P(\la)}\left(\la^3,\la^3,\la^3,\la^3,\la^4,\la^5-4\la^3,\la^6-5\la^4,\la^7-6\la^5+4\la^3,\la^8-7\la^6+9\la^4\right) .\]
We checked conditions (i)--(viii) for 
$\la'=2.312; \;
\la'' =2.34; \;  
c= 1.5 $.
Due to Remark~\ref{rem:cond8}, it is enough to verify condition (viii) for $k=4$. See \cite{JUP} for the details.
\end{proof}

\section{Putting the pieces together}
\label{sec:the_end}

We will now piece together the different parts of the paper and prove the main theorems.

\begin{proof}[Proof of Theorem~\ref{thm:only_K4Pk}]
We can quickly check with a computer that $\Ga_G$ is minimized by $G=K_4+P_2$ among $6$-vertex connected graphs. Note that there are four other $6$-vertex graphs with $\Ga_G < \bestar$. For a given $n \geq 7$, we need to show that the only $n$-vertex graph with this property is $G=K_4 + P_{n-4}$. In Section~\ref{sec:pendant_paths} we showed that $K_4 + P_{n-4}$ indeed has this property, see Theorem~\ref{thm:indeed_below}. 

Now suppose that $G$ is a connected graph with $n \geq 7$ vertices such that $\Ga_G<\bestar$. Since $\bestar < 5.25$ and $\bestar < \Ga_{D +_s P_3} \approx 5.18$, Theorem~\ref{thm:525} implies that $N_G[o] = C \cup \{u\}$, where $C$ is a $4$-clique containing the master vertex $o$ and the only edge between $C$ and $V(G) \sm C$ is $(o,u)$, see Figure~\ref{fig:K4_plus_Gpr}. It is easy to see that if $\deg_G(v) \leq 2$ for every vertex $v \notin C$, then the rest of the graph must be a single path, and hence $G \cong K_4 + P_{n-4}$ as claimed. We prove by contradiction: we assume that there exists a vertex $v \notin C$ with $\deg_G(v) \geq 3$ and we take the one closest to $o$ (which may coincide with $u$). Let $\ell:=\dist(o,v)\geq 1$. Now consider the induced subgraph $H$ containing $C$ and the unique path between $o$ and $v$ along with two further neighbors of $v$: $v'$ and $v''$. Then the rooted graph $(H,o)$ must be one of the two graphs below (depending on whether $v' v''$ is an edge or not):
\[ \KKKKPr \cdots \cherry \quad \text{and} \quad  
\KKKKPr \cdots \tri \, .\]
We also know that only $v,v',v''$ may have neighbors outside $H$. Therefore, setting $\Vact:=\{ v,v',v'' \}$ and $\Uc := \Pc_+(\Vact)$, we conclude that $G$ is a $\Uc$-extension of $(H,o)$. For $\ell=1,2$ a computer check (with $\beta=\bestar$) verifies that condition \eqref{eq:coeff_check_poly} holds in these cases. Therefore, Theorem~\ref{thm:Ga_xt_bound_spec} can be applied and we obtain that $\Ga_G \geq \bestar$ for any $\Uc$-extension $G$ of the two kernels above, contradiction. For $\ell \geq 3$ we can use Theorem~\ref{K4_tail} to conclude that $\Ga_G > \bestar$ and again reach a contradiction. The proof is complete.
\end{proof}
\begin{proof}[Proof of Theorem~\ref{thm:only_S5Pk}]
The proof of this theorem follows essentially the same argument as that of Theorem~\ref{thm:only_K4Pk}. 

We start by checking with a computer that for any given $n=8, \ldots, 13$, the $n$-vertex tree with the smallest $\Ga_T$ is indeed $T=S_5 + P_{n-5}$. For these values of $n$, there is at least one other $T$ with $\Ga_T<\betr$ (specifically, $T=S_6 + P_{n-6}$). For a given $n \geq 14$, we need to show that the only $n$-vertex graph with this property is $T=S_5 + P_{n-5}$. Theorem~\ref{thm:indeed_below} proves that $S_5 + P_{n-5}$ has this property.

Suppose that $\Ga_T<\betr$ for a tree $T$ with $n \geq 14$ vertices. Since the exceptions in Theorem~\ref{thm:10kernel} have at most $13$ vertices, it follows that the $10$-kernel of any such $T$ must be $\big( S_5 +_o P_5, o \big)$. Take the vertex $v \neq o$ closest to $o$ with $\deg_T(v) \geq 3$. (If there is no such $v$, then $T$ is necessarily isomorphic to $S_5 + P_{n-5}$, and we are done.) Note that $\ell := \dist(o,v)$ must be at least $4$ (otherwise the $10$-kernel would be different). Now take the induced subgraph $H$ as in the previous proof. Since $T$ is a tree, it cannot contain a triangle, so this time the neighbors $v', v''$ of $v$ cannot be adjacent, and we only have one option for $(H,o)$:
\[ \SSSSSPPr \cdots \cherry \, .\]
For $\ell=4,5,6,7$ we need to perform one last computer check: for $\beta=\betr$ we verify \eqref{eq:coeff_check_poly} to conclude that for any $\big\{ \{v\} , \{v'\} ,\{v''\} \big\}$-extension $T$ of $(H,o)$ we have $\Ga_T \geq \betr$. For $\ell\geq 8$, we can use Theorem~\ref{S5_tail} to reach the same conclusion, and hence a contradiction. The proof is complete. 
\end{proof}

\medskip

\noindent \textbf{Acknowledgment.} The first author is very grateful to Krystal Guo for suggesting the problem and for fruitful discussions. He also thanks Sebastian Cioab\u{a} for helpful discussion on the problem.

\bibliography{bibliography}
\bibliographystyle{plain}

\bigskip

\addresseshere

\newpage

\section{Appendix: a concrete example} \label{sec:appendix}

In order to better demonstrate how our bounding method works, here we carry out the necessary computations for a specific kernel $(H,o)$. We picked the following $6$-vertex rooted graph:
\[ (H,o) := \, \KKKPPPr \, = \big( K_3 +_o P_3 , o\big) .\]
This is actually one of the four ``exceptional`` kernels in the proof of Theorem~\ref{thm:525} for which the simple method fails and a \texttt{two-step verification} is required (see Section~\ref{subsec:comp_one}).

We denote the six vertices as below ($o$, $v_1$, $v_2$  forming the $3$-clique, while the pendant path is $u_1 u_2 u_3$, with $u_3$ being the single leaf of $H$).

\begin{center}
\KKKPPPrl
\end{center}

\noindent Since $\dist(o,u_3)=3$, the set of active vertices for this kernel is 
\[ \Vact=\{ u_2, u_3\} \text{, and hence the relevant subsets are } 
\Uc=\bigg\{ \{u_2\}, \{u_3\}, \{u_2,u_3\} \bigg\} .\]
The adjacency matrix of $H$ is 
\[ A_H = 
\left(\begin{array}{rrrrrr}
0 & 1 & 1 & 1 & 0 & 0 \\
1 & 0 & 1 & 0 & 0 & 0 \\
1 & 1 & 0 & 0 & 0 & 0 \\
1 & 0 & 0 & 0 & 1 & 0 \\
0 & 0 & 0 & 1 & 0 & 1 \\
0 & 0 & 0 & 0 & 1 & 0
\end{array}\right) \]
with the following characteristic polynomial: 
\[ P(\la) = \det(\la I - A_H)= \la^{6} - 6\la^{4} - 2\la^{3} + 8\la^{2} + 4\la - 1 .\]
Then $P(\la) B = \adj(\la I - A_H)$ is the following $6 \times 6$ matrix:\footnote{The order of the vertices in the rows and columns are $o, v_1, v_2, u_1, u_2, u_3$.}
\begin{multline*}
\left(\begin{array}{rrr}
\la^{5} - 3\la^{3} + 2\la & \la^{4} + \la^{3} - 2\la^{2} - 2\la & \la^{4} + \la^{3} - 2\la^{2} - 2\la \\
\la^{4} + \la^{3} - 2\la^{2} - 2\la & \la^{5} - 4\la^{3} + 3\la & \la^{4} + \la^{3} - 3\la^{2} - 2\la + 1 \\
\la^{4} + \la^{3} - 2\la^{2} - 2\la & \la^{4} + \la^{3} - 3\la^{2} - 2\la + 1 & \la^{5} - 4\la^{3} + 3\la \\
\la^{4} - 2\la^{2} + 1 & \la^{3} + \la^{2} - \la - 1 & \la^{3} + \la^{2} - \la - 1 \\
\la^{3} - \la & \la^{2} + \la & \la^{2} + \la \\
\la^{2} - 1 & \la + 1 & \la + 1
\end{array}\right. \quad \cdots \\
\cdots \quad \left. \begin{array}{rrr}
\la^{4} - 2\la^{2} + 1 & \la^{3} - \la & \la^{2} - 1 \\
\la^{3} + \la^{2} - \la - 1 & \la^{2} + \la & \la + 1 \\
\la^{3} + \la^{2} - \la - 1 & \la^{2} + \la & \la + 1 \\
\la^{5} - 4\la^{3} - 2\la^{2} + 3\la + 2 & \la^{4} - 3\la^{2} - 2\la & \la^{3} - 3\la - 2 \\
\la^{4} - 3\la^{2} - 2\la & \la^{5} - 4\la^{3} - 2\la^{2} + \la & \la^{4} - 4\la^{2} - 2\la + 1 \\
\la^{3} - 3\la - 2 & \la^{4} - 4\la^{2} - 2\la + 1 & \la^{5} - 5\la^{3} - 2\la^{2} + 4\la + 2
\end{array}\right)
\end{multline*}
Everything below will be expressed using the polynomials in this matrix and the characteristic polynomial $P$.

Given a target ratio $\beta$, in order to show that $\Ga_G \geq \beta$ for every $\Uc$-extension of $(H,o)$, we need to check a polynomial inequality for every pair $U,V \in \Uc$. The roles of $U$ and $V$ are symmetric so there are $6$ possible pairs. For $\beta=5.25$ the check fails for one of these pairs: $U=V=\{ u_3 \}$. We include the computations only for this pair. For $U=\{ u_3 \}$ we have
\begin{align*}
P(\la) \Bt_{o,U}(\la) &= P(\la) B_{o,u_3}(\la) \\ 
&= \la^2-1 ;\\
P(\la) s_U(\la) &= \sum_{v\in V(H)} P(\la) B_{v,u_3}(\la) \\ 
&= \la^5+\la^4-4\la^3-5\la^2+\la+2 ;\\
P^2(\la) c_{U,U}(\la) &=  \sum_{v\in V(H)} \big( P(\la) B_{v,u_3}(\la) \big)^2 \\
&= \la^{10} - 9\la^{8} - 4\la^{7} + 26\la^{6} + 20\la^{5} - 23\la^{4} - 24\la^{3} + 13\la^{2} + 28\la + 12 .
\end{align*}
Given some target ratio $\beta$, we need to consider the polynomial $Q=Q_{U,V}$ defined in \eqref{eq:coeff_check_poly}. In the special case $U=V$:
\[ Q_{U,U}(\la) = \bigg( P(\la) s_U(\la) + P(\la) \bigg)^2-\be\bigg( P^2(\la) c_{U,U}(\la) + P^2(\la) \Bt_{o,U}(\la) \bigg) .\]
For $U=\{u_3\}$ and $\beta=5.25=5+1/4$ we get 
\begin{multline*}
Q(\la) = Q_{\{u_3\},\{u_3\}}(\la) = \la^{12} + 2 \la^{11} - \frac{57}{4} \la^{10} - 22 \la^{9} + 61 \la^{8} \\
+ 97 \la^{7} - \frac{327}{4} \la^{6} - \frac{357}{2} \la^{5} - \frac{55}{4} \la^{4} + \frac{225}{2} \la^{3} + 10 \la^{2} - 116 \la - \frac{269}{4} .
\end{multline*}
What needs to be checked is that $Q(\la)\geq 0$ for all $\la \geq \lat = \max( \la_{U}, \la_{V} )$. In this case $\lat = \la_{\{u_3\}} = \la_{K_3+P_4} \approx 2.2332$. This would follow if the coefficients of the translated polynomial $\Qt(\ka) = Q(\ka+\lat)$ had positive coefficients. In our case the constant term turns out to be negative: using approximate coefficients, for $\Qt(\ka)$ we have:
\begin{multline*}
\ka^{12} + 28.8 \ka^{11} + 364.04 \ka^{10} + 2658.62 \ka^{9} + 12408 \ka^{8} + 38599.44 \ka^{7} + 80880.76 \ka^{6} \\
+ 112535.68 \ka^{5} + 99776.84 \ka^{4} + 52219.63 \ka^{3} + 14485.75 \ka^{2} + 2162.81 \ka - 5.54
\end{multline*}
Due to this check fail, we needed to handle this (and three other) kernels slightly differently in the proof of Theorem~\ref{thm:525}. In fact, for $\beta=5+1/8=5.125$ we would already get a $\Qt$ with positive coefficients. So even the simple method actually implies that $\Ga_G \geq 5.125 > \bestar$ for every $\Uc$-extension $G$ of $(H,o)$.

To get $\Qt$ we need to translate $Q$ by $\la_U$. Recall that for a fixed $H$ we defined $\la_U$ as the top eigenvalue of the graph ``$H +_U \{w\}$'' obtained from $H$ by connecting a new vertex $w$ to each $u \in U$; see Definition~\ref{def:laU}. In our case these eigenvalues are the following.
\begin{align*}
\la_H &\approx 2.2283 
\quad \text{correspondig to } \; \KKKPPP  \\
\la_{\{u_3\}} &\approx 2.2332 
\quad \text{correspondig to } \; \KKKPPPP  \\
\la_{\{u_2\}} &\approx 2.2533 
\quad \text{correspondig to } \; \KKKPPFF\\
\la_{\{u_2,u_3\}} &\approx 2.3429 
\quad \text{correspondig to } \; \KKKPKKK
\end{align*}

As a final point, we would like to demonstrate the difference between the bounds (i) and (iii) provided in Theorem~\ref{thm:Ga_xt_bound}. To this end, we will plot the bounds as a function of $\la=\la_G$. It is easy to see from Proposition~\ref{prop:B_properties}(e) that for each version of the bound ($j=1,2,3$) it holds that for any $U,V$ 
\[ \frac{a_{U,V}\big( \la \big)}{b^{(j)}_{U,V}\big( \la \big)}
\to \Ga_H \quad \text{as } \la \to \la_H .
\]
Therefore, if $\Ga_H$ is less than our target ratio (which is the case now), then we need a lower bound on $\la$, which is stronger than simply using $\la>\la_H$. Of course, we know that $\la_G \geq \la_{\partial w}$ for any outside vertex $w$. Consequently, it suffices to look at the plots for $\la \geq \lat = \max( \la_{U}, \la_{V} )$. When $U=V$, and hence $\lat=\la_U$, it can be seen easily that bound (i) is actually sharp at the starting point:
\[  \frac{a_{U,U}\big( \la_U \big)}{b^{(1)}_{U,U}\big( \la_U \big)} = \Ga_{H'} ,\]
where $H'$ is the graph $H +_U \{w\}$ mentioned above. This indeed shows that the bound is sharp because $H'$ extends $H$ with top eigenvalue $\la_U$. 

Due to this sharpness at $\la_U$, bound (i) will always be stronger in a short interval $\big( \la_U,\la_U+\varepsilon \big)$ but eventually it will drop below bound (iii). In Figure~\ref{fig:bound_comparison} we plotted 
\[ 
\frac{a_{U,U}\big( \la \big)}{b^{(1)}_{U,U}\big( \la \big)} \text{ in red and } 
\frac{a_{U,U}\big( \la \big)}{b^{(3)}_{U,U}\big( \la \big)} 
\text{ in blue for } U=\{u_3\}.\] 
The black dots show the points $\big( \la_{H} \, , \, \Ga_{H} \big)$ and 
$\big( \la_{H'} \, , \, \Ga_{H'} \big)$. Both curves start at the former point (due to the asymptotics above), while the latter lies on the red curve (due to the sharpness mentioned previously). The green line shows the target ratio $\be=5.25$. The coefficient check failed for this $\beta$, because the blue curve is under the green line at $\la_U = \la_{H'}$, indicated by the second black dot.
\begin{figure}
\begin{center}
\includegraphics[width=0.95\textwidth]{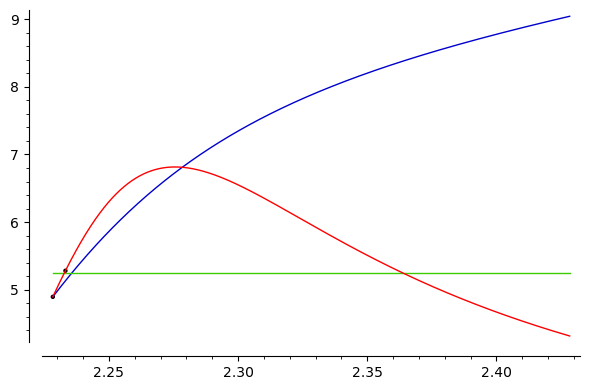}
\end{center}
\caption{Comparison of our bounds as functions of $\la$}
\label{fig:bound_comparison}
\end{figure}
%


\end{document}